\documentclass[10pt,reqno]{amsart}
\usepackage[svgnames,table]{xcolor}
\usepackage[colorlinks=true,citecolor=Magenta,linkcolor=Cyan]{hyperref} 
\usepackage{amsmath,amsfonts,latexsym,amssymb}
\usepackage{amssymb,amsfonts, amsmath}
\usepackage{mathrsfs}
\usepackage[latin1]{inputenc}
\usepackage[T1]{fontenc}
\usepackage{ae,aecompl}
\usepackage{braket}
\usepackage{comment}
\usepackage{amssymb,amsfonts, amsmath}
 \usepackage{mathtools} 
\usepackage[all,arc]{xy}
\usepackage{enumerate}
\usepackage{mathrsfs}
\usepackage{mathabx}
\usepackage{enumitem} 
\usepackage{color}
\usepackage{graphicx}
\usepackage{subfig}
\usepackage{verbatim} 
\usepackage{amssymb,amsfonts, amsmath}
\usepackage[all,arc]{xy}
\usepackage{enumerate}
\usepackage{mathrsfs}
\usepackage{hyperref}
\usepackage{xstring}
\usepackage{thmtools}
\newcommand{\p}{{\bf{p}}}

\newtheorem{theorem}{Theorem}[section]
\newtheorem{lemma}[theorem]{Lemma}
\newtheorem{proposition}[theorem]{Proposition}
\newtheorem{corollary}[theorem]{Corollary}
\newtheorem{definition}[theorem]{Definition}
\newtheorem{definitions}[theorem]{Definitions}
\newtheorem{Example}[theorem]{Example}

\newtheorem{remark}[theorem]{\normalfont{\em{Remark}}}
\newtheorem{remarks}[theorem]{\normalfont{\em{Remarks}}}
\newtheorem{fact}[theorem]{Fact}

\makeatletter
\renewcommand*\env@matrix[1][\arraystretch]{%
\edef\arraystretch{#1}%
\hskip -\arraycolsep
\let\@ifnextchar\new@ifnextchar
\array{*\c@MaxMatrixCols c}}
\makeatother

\title[\tiny{Anosov representations, strongly convex cocompact groups and weak eigenvalue gaps}]{Anosov representations, strongly convex cocompact groups and weak eigenvalue gaps}
\author{Konstantinos Tsouvalas}
\date{\today}
\AtEndDocument{\bigskip{\footnotesize
  \textsc{} \par  
\textsc{Max Planck Institute for Mathematics in the Sciences, Inselstrasse 22, 04103 Leipzig, Germany} \par  
 \textit{E-mail address}: \texttt{konstantinos.tsouvalas@mis.mpg.de}}}
\date{\today}
\begin{document}
\frenchspacing
\maketitle
\begin{abstract} We provide characterizations of Anosov representations of word hyperbolic groups into real semisimple Lie groups in terms of the existence of equivariant limit maps on the Gromov boundary, the Cartan property and the uniform gap summation property introduced by Guichard--Gu\'eritaud--Kassel--Wienhard in \cite{GGKW}. We also study representations of finitely generated groups satisfying weak uniform gaps in eigenvalues and establish conditions to be Anosov. As an application, we also obtain a characterization of strongly convex cocompact subgroups of the projective linear group $\mathsf{PGL}_d(\mathbb{R})$. 

\end{abstract}
\vspace{0.6cm}
\section{Introduction}
\par Anosov representations of fundamental groups of closed negatively curved Riemannian manifolds were introduced by Labourie \cite{labourie} in his study of the Hitchin component. Labourie's definition was later extended by Guichard--Wienhard in \cite{GW} for general word hyperbolic groups. Anosov representations have been extensively studied during the last decade by Guichard--Wienhard \cite{GW}, Kapovich--Leeb--Porti \cite{KLP1,KLP2, KLP3}, Bochi--Potrie--Sambarino \cite{BPS}, Gu\'eritaud--Guichard--Kassel--Wienhard \cite{GGKW}, Danciger--Gu\'eritaud--Kassel \cite{DGK0}, Zimmer \cite{Zimmer} and others, and are now are recognized as a higher rank analogue of convex cocompact representations of word hyperbolic groups into simple Lie groups of real rank $1$. Moreover, recently, there have been introduced certain generalizations of classical Anosov representations for relatively hyperbolic groups and other groups; we refer to the work of Kapovich--Leeb \hbox{\cite{KLrel}, Zhu \cite{Zhu} and Weisman \cite{Weisman} for more details.}

\par Based on the existing characterizations established in \cite{GW, GGKW, KLP1, KLP2, KLP3, BPS, kassel-potrie}, one may define Anosov representations of a hyperbolic group into a semisimple Lie group in terms of the existence of a pair of well-behaved limit maps from the Gromov boundary of the domain group to the corresponding flag spaces, or entirely in terms of uniform gaps in the Cartan or Lyapunov projection of the image of the representation. The purpose of the present paper is to provide new characterizations and strengthen some of the existing ones. Our characterizations are in terms of the existence of limit maps, the Cartan property (see subsection \ref{Cartan0}) and the uniform gap summation property introduced in \cite{GGKW}. As an application of our main results, we also obtain characterizations of strongly convex cocompact subgroups of the projective linear group $\mathsf{PGL}_d(\mathbb{R})$ \hbox{(see subsection \ref{scc0}).} More generally, we study linear representations of finitely generated groups satisfying weak uniform gaps in eigenvalues and we establish sufficient conditions for the domain group to be word hyperbolic and the representation to be Anosov (see sub-section \ref{weak0}). In order to provide such conditions, we study the relation between strong property (U), introduced by Kassel--Potrie in \cite{kassel-potrie}, and the uniform gap summation property. More precisely, we prove that a finitely generated non-virtually nilpotent group $\Gamma$ which admits a linear representation with the uniform gap summation property (see Definition \ref{ugspdef}), then $\Gamma$ satisfies strong property (U) which is a condition relating the word length and the stable translation length of certain group elements (see Theorem \ref{nontrivial}). 

\subsection{Characterizations in terms of limit maps and the Cartan property} \label{Cartan0} Let $\Gamma$ be an infinite word hyperbolic group, $G$ be a linear, non-compact semisimple Lie group with finitely many connected components and fix $K$ a maximal compact subgroup of $G$. We also fix a Cartan subspace $\mathfrak{a}$ of $\mathfrak{g}$, $\overline{\mathfrak{a}}^{+}$ a closed Weyl chamber of $\mathfrak{a}$, a Cartan decomposition $G=K\exp(\overline{\mathfrak{a}}^{+})K$ and consider the Cartan projection $\mu:G \rightarrow \overline{\mathfrak{a}}^{+}$. 

\par Every subset $\theta \subset \Delta$ of simple restricted roots of $G$ defines a pair of opposite parabolic subgroups $P_{\theta}^{+}$ and $P_{\theta}^{-}$, well-defined up to conjugation. Labourie's dynamical definition of a $P_{\theta}$-Anosov representation $\rho:\Gamma \rightarrow G$ requires the existence of a pair of continuous $\rho$-equivariant maps from the Gromov boundary $\partial_{\infty}\Gamma$ to the flag spaces $G/P_{\theta}^{+}$ and $G/P_{\theta}^{-}$ called the \emph{Anosov limit maps} of $\rho$ (see Definition \ref{Def-of-Anosov-rep}). Our first characterization of Anosov representations is based on the existence of a pair of transverse continuous, equivariant limit maps on the Gromov boundary of the domain group, one of which satisfies the Cartan property: \begin{theorem}\label{maintheorem} Let $\Gamma$ be a word hyperbolic group, $G$ a real semisimple Lie group, $\theta \subset \Delta$ a subset of simple restricted roots of $G$ and $\rho:\Gamma \rightarrow G$ a representation. Then $\rho$ is $P_{\theta}$-Anosov if and only if the following conditions are simultaneously satisfied:
\begin{enumerate}[label=(\roman*)]
\item \label{maintheorem-1} $\rho$ is $P_{\theta}$-divergent.
\item \label{maintheorem-2} There exists a pair of continuous, $\rho$-equivariant transverse maps 
$$\xi^{+}:\partial_{\infty}\Gamma \rightarrow G/P_{\theta}^{+} \ \ and \ \ \xi^{-}:\partial_{\infty}\Gamma \rightarrow G/P_{\theta}^{-}$$ 
and the map $\xi^{+}$ satisfies the Cartan property.\end{enumerate} \end{theorem}

Let us now briefly explain the assumptions of Theorem \ref{maintheorem}. For a representation $\rho:\Gamma \rightarrow G$ of a hyperbolic group $\Gamma$, two $\rho$-equivariant maps \hbox{$\xi^{+}:\partial_{\infty}\Gamma \rightarrow G/P_{\theta}^{+}$} and $\xi^{-}:\partial_{\infty}\Gamma \rightarrow G/P_{\theta}^{-}$ are \emph{transverse}, if for any two distinct points  $x^{+},x^{-} \in \partial_{\infty}\Gamma$ there is $g \in G$ such that $\xi^{+}(x^{+})=gP_{\theta}^{+}$ and $\xi^{-}(x^{-})=gP_{\theta}^{-}$. The representation \hbox{$\rho:\Gamma \rightarrow G$} is \hbox{$P_{\theta}$-divergent} if for every infinite sequence $(\gamma_n)_{n \in \mathbb{N}}$ of elements of $\Gamma$ and $\alpha \in \theta$, the sequence $\big( \alpha(\mu(\rho(\gamma_n))) \big)_{n \in \mathbb{N}}$ goes to infinity. The map $\xi^{+}:\partial_{\infty}\Gamma \rightarrow G/P_{\theta}^{+}$ satisfies the \emph{Cartan property} if for every sequence $(\gamma_n)_{n \in \mathbb{N}}$ of elements of $\Gamma$ converging to a point $x\in \partial_{\infty}\Gamma$ in the Gromov boundary, then $\xi^{+}(x)=\lim_{n}k_{n}P_{\theta}^{+}$, where $\rho(\gamma_n)=k_{n}\exp(\mu(\rho(\gamma_n)))k_{n}'$, \hbox{$k_n,k_n'\in K$,} is written in the Cartan decomposition of $G$. Examples of maps with this property are the limit maps of an Anosov representation (see \cite{BPS} and \cite[Thm. 1.3 (4) \& 5.3 (4)]{GGKW}). We discuss the Cartan property in more detail in \S \ref{scartan}, where we prove (see Corollary \ref{Cartan}) that for any Zariski dense representation $\rho:\Gamma \rightarrow G$ a (necessarily unique if it exists) continuous $\rho$-equivariant map $\xi:\partial_{\infty}\Gamma \rightarrow G/P_{\theta}^{\pm}$ has to satisfy the Cartan property.

\par In Theorem \ref{maintheorem} the assumption that the map $\xi^{+}$ satisfies the Cartan is necessary and cannot be dropped (see Example \ref{cp-necessity}). Moreover, we do not assume that the image $\rho(\Gamma)$ contains a $P_{\theta}$-proximal element in $G/P_{\theta}^{\pm}$ or that the pair of maps $(\xi^{+},\xi^{-})$ is compatible at some point $x \in \partial_{\infty}\Gamma$, i.e. the intersection $\textup{Stab}_{G}(\xi^{+}(x)) \cap \textup{Stab}_{G}(\xi^{-}(x))$ is a parabolic subgroup of $G$. Under the assumption that both maps $(\xi^{+}, \xi^{-})$ satisfy the Cartan property, Theorem \ref{maintheorem} also follows from \cite[Thm 1.7]{KLP3}. We explain how Theorem \ref{maintheorem} is related to \cite[Thm. 1.7]{KLP3}, \cite[Thm. 5.47]{KLP1} and \cite[Thm. 1.3]{GGKW} at the end of this section.

\par Let $\Gamma$ be a finitely generated group. We fix a left invariant word metric $d_{\Gamma}$ on $\Gamma$ induced by a finite generating subset of $\Gamma$ and let $|\cdot|_{\Gamma}:\Gamma \rightarrow \mathbb{N}$ be the word length function defined by $|\gamma|_{\Gamma}=d_{\Gamma}(\gamma,e), \gamma \in \Gamma$.  As an application of Theorem \ref{maintheorem}, we deduce the following characterization of Anosov representations entirely in terms of the growth of the Cartan projection of the image of a representation. 

\begin{corollary} \label{CCartan} Let $\Gamma$ be an infinite word hyperbolic group, $G$ a real semisimple Lie group, $\theta \subset \Delta$ a subset of simple restricted roots of $G$, $\{\omega_{\alpha}\}_{\alpha \in \theta}$ the associated set of fundamental weights. Fix $|\cdot|_{\Gamma}:\Gamma \rightarrow \mathbb{N}$ a word length function  on $\Gamma$.  A representation $\rho:\Gamma \rightarrow G$ is $P_{\theta}$-Anosov if and only if the following conditions are simultaneously satisfied:
\begin{enumerate}[label=(\roman*)]
\item \label{CCartan-1}There exist $C,c>1$ such that for every $\gamma \in \Gamma$ non-trivial and $\alpha \in \theta$,
$$\alpha\big(\mu(\rho(\gamma))\big)  \geqslant c \log |\gamma |_{\Gamma}-C.$$

\item \label{CCartan-2} There exist $B,b>0$ such that for every $\gamma \in \Gamma$ and $\alpha\in \theta$, $$  \omega_{\alpha} \big(2\mu (\rho(\gamma))-\mu (\rho(\gamma^2)) \big)  \leqslant B\big (2|\gamma |_{\Gamma}-|\gamma^2|_{\Gamma} \big)+b.$$\end{enumerate}  \end{corollary}

\par Now let $\rho:\Gamma \rightarrow G$ be a Zariski dense representation which admits a pair of $\rho$-equivariant, continuous limit maps \hbox{$\xi^{+}:\partial_{\infty} \Gamma \rightarrow G/P_{\theta}^{+}$} and $\xi^{-}:\partial_{\infty}\Gamma \rightarrow G/P_{\theta}^{-}$. In \cite[Thm. 5.11]{GW}, Guichard--Wienhard proved that $\rho$ is $P_{\theta}$-Anosov if and only if $\xi^{+}$ and $\xi^{-}$ are compatible and transverse. By Theorem \ref{maintheorem} and Corollary \ref{Cartan}, we obtain the following slightly improved version of their theorem. For a quasi-convex subgroup $H$ of $\Gamma$ we denote by $\iota_{H}:\partial_{\infty}H \xhookrightarrow{} \partial_{\infty}\Gamma$ the Cannon--Thurston map extending the natural inclusion $H \xhookrightarrow{} \Gamma$.

\begin{theorem} \label{Zariskidense} Let $\Gamma$ be a word hyperbolic group, $H$ a quasiconvex subgroup of $\Gamma$, $G$ a semisimple Lie group, $\theta \subset \Delta$ a subset of simple restricted roots of $G$ and $\rho:\Gamma \rightarrow G$ a Zariski dense representation. Suppose that $\rho$ admits continuous, $\rho$-equivariant maps $\xi^{+}: \partial_{\infty}\Gamma \rightarrow G/P_{\theta}^{+}$ and $\xi^{-}:\partial_{\infty}\Gamma \rightarrow  G/P_{\theta}^{-}$. Then the restriction $\rho|_{H}:H\rightarrow G$ is $P_{\theta}$-Anosov if and only if the maps $\xi^{+}\circ \iota_ {H}:\partial_{\infty}H \rightarrow G/P_{\theta}^{+}$ and $\xi^{-}\circ \iota_{H}:\partial_{\infty}H \rightarrow G/P_{\theta}^{-}$ are transverse.\end{theorem}

For a matrix $g \in \mathsf{GL}_d(\mathbb{R})$ we denote by $\ell_1(g)\geq \cdots \geq \ell_d(g)$ and $\sigma_1(g)\geq \cdots \geq \sigma_d(g)$ the moduli of eigenvalues and the singular values of $g$ respectively in non-increasing order. Let $\rho_{i}:\Gamma \rightarrow \mathsf{SL}_{m_{i}}(\mathbb{R})$, $i\in \{1,2\},$ be two representations such that $\rho_2$ is $P_1$-Anosov. We recall that the stretch factors associated with the representations $\rho_1$ and $\rho_2$ of $\Gamma$ are: $$\textup{dil}_{-}(\rho_1,\rho_2):=\underset{\gamma \in \Gamma_{\infty}}{\textup{inf}} \frac{\log \ell_1(\rho_1(\gamma))}{\log \ell_1(\rho_2(\gamma))},\ \textup{dil}_{+}(\rho_1,\rho_2):=\underset{\gamma \in \Gamma_{\infty}}{\textup{sup}}\frac{\log \ell_1(\rho_1(\gamma))}{\log \ell_1(\rho_2(\gamma))}$$\\ where $\Gamma_{\infty}$ denotes the set of infinite order elements of $\Gamma$. Observe that since $\rho_{2}$ is a quasi-isometric embedding (see Theorem \ref{mainproperties}(i)), the stretch factors $\textup{dil}_{\pm}(\rho_1,\rho_2)$ are well-defined. As a corollary of Theorem \ref{maintheorem} we obtain the following approximation result  for particular pairs of representations $(\rho_1,\rho_2)$, which refines a consequence of the density result of \hbox{Benoist obtained in \cite{limitcone2} in this case.}

\begin{corollary} \label{interval} Let $\Gamma$ be a word hyperbolic group and fix $|\cdot|_{\Gamma}:\Gamma \rightarrow \mathbb{N}$ a word length function on $\Gamma$. Suppose that $\rho_{1}:\Gamma \rightarrow \mathsf{SL}_{m_1}(\mathbb{R})$ and $\rho_{2}:\Gamma \rightarrow \mathsf{SL}_{m_2}(\mathbb{R})$ are two representations such that $\rho_2$ is $P_1$-Anosov and $\rho_1$ satisfies one of the following conditions:
\begin{enumerate} [label=(\roman*)]
\item $\rho_1$ is $P_1$-Anosov.
\item $\rho_{1}(\Gamma)$ is contained in a semisimple $P_1$-proximal Lie subgroup of $\mathsf{SL}_{m_1}(\mathbb{R})$ of real rank $1$.\end{enumerate}
\noindent Then for every $\epsilon>0$ and $p,q \in \mathbb{N}$ with $\textup{dil}_{-}(\rho_1,\rho_2) \leqslant \frac{p}{q} \leqslant \textup{dil}_{+}(\rho_1,\rho_2)$, there exists an infinite sequence $(\gamma_n)_{n \in \mathbb{N}}$ of elements of $\Gamma$ such that for every $n\in \mathbb{N}$:  $${\Bigg|\frac{p}{q}-\frac{\log \sigma_1(\rho_1(\gamma_n))}{\log \sigma_1(\rho_2(\gamma_n))} \Bigg |\leqslant \frac{\epsilon}{q} \cdot \frac{\log |\gamma_n|_{\Gamma}}{|\gamma_n|_{\Gamma}}}.$$ \end{corollary}

\subsection{Weak uniform gaps in eigenvalues and strong property $\textup{(U)}$.} \label{weak0}
 Kassel--Potrie introduced the following definition in \cite{kassel-potrie}:

\begin{definition} Let $\Gamma$ be a finitely generated group, $\rho:\Gamma \rightarrow \mathsf{GL}_d(\mathbb{R})$ a representation and fix $1\leqslant i \leqslant d-1$. The representation $\rho$ has a weak uniform i-gap in eigenvalues if there exists $\varepsilon>0$ such that for every $\gamma \in \Gamma$ we have $$\log \frac{\ell_i(\rho(\gamma))}{\ell_{i+1}(\rho(\gamma))} \geqslant \varepsilon |\gamma |_{\infty},$$ where $|\gamma|_{\infty}=\lim_{n}\frac{|\gamma^n|_{\Gamma}}{n}$ denotes the stable translation length of $\gamma$. \end{definition}

\par The existence of a uniform $i$-gap in eigenvalues for $\rho$ is not a sufficient condition to guarantee that the representation is Anosov, and it is a natural question to determine additional conditions guaranteeing that this happens. Gu\'eritaud--Guichard--Kassel--Wienhard proved that if $\Gamma$ is word hyperbolic, $\rho$ has a weak uniform $i$-gap in eigenvalues and admits a pair of continuous, $\rho$-equivariant, dynamics preserving and transverse maps $\xi^{+}:\partial_{\infty}\Gamma \rightarrow \mathsf{Gr}_{i}(\mathbb{R}^d)$ and $\xi^{-}:\partial_{\infty}\Gamma \rightarrow \mathsf{Gr}_{d-i}(\mathbb{R}^d)$, then $\rho$ is $P_i$-Anosov (see \cite[Thm. 1.7 (c)]{GGKW}). Kassel--Potrie proved \cite[Prop. 4.12]{kassel-potrie} that if $\Gamma$ satisfies weak property (U) (see Definition \ref{PropU1}) and $\rho$ has a weak uniform $i$-gap in eigenvalues, then $\rho$ has a strong $i$-gap in singular values: there exist $C,c>0$ such that for every $\gamma \in \Gamma$, $$\log \frac{\sigma_i(\rho(\gamma))}{\sigma_{i+1}(\rho(\gamma))}  \geqslant c |\gamma |_{\Gamma}-C,$$ hence $\Gamma$ is hyperbolic and $\rho$ is $P_i$-Anosov by the work of Kapovich--Leeb--Porti \cite{KLP2} and Bochi--Potrie--Sambarino \cite{BPS}. The following theorem, motivated by \cite[Ques. 4.9]{kassel-potrie}, provides further conditions under which a linear representation $\rho:\Gamma \rightarrow \mathsf{GL}_d(\mathbb{R})$ of a finitely generated group $\Gamma$ with a weak uniform $i$-gap in eigenvalues is $P_i$-Anosov and $\Gamma$ is hyperbolic. For the definition of the Floyd boundary we refer the reader to \cite{Floyd}, see also \S \ref{Background}.

\begin{theorem} \label{wg} Let $\Gamma$ be a finitely generated infinite group which is not virtually cyclic and fix $|\cdot|_{\Gamma}:\Gamma \rightarrow \mathbb{N}$ a word length function on $\Gamma$. Suppose that $\rho:\Gamma \rightarrow \mathsf{GL}_d(\mathbb{R})$ is a representation which has a weak uniform $i$-gap in eigenvalues for some $1 \leq i \leq d-1$. Then the following conditions for $\Gamma$ and $\rho$ are equivalent:
\begin{enumerate}[label=(\roman*)]
\item $\Gamma$ is word hyperbolic and $\rho$ is $P_{i}$-Anosov.
\item There exists a Floyd function $f$ such that the Floyd boundary $\partial_{f}\Gamma$ of $\Gamma$  is uncountable.
\item $\Gamma$ admits a representation \hbox{$\rho_1:\Gamma \rightarrow \mathsf{GL}_m(\mathbb{R})$ satisfying the uniform gap summation property.}
\item $\Gamma$ admits a semisimple representation $\rho_2:\Gamma \rightarrow \mathsf{GL}_r(\mathbb{R})$ with the property $$\lim_{|\gamma|_{\Gamma}\rightarrow \infty}\frac{\log\sigma_1(\rho_2(\gamma))-\log \sigma_r(\rho_2(\gamma))}{\log |\gamma |_{\Gamma}}=+\infty.$$\end{enumerate} \end{theorem}

 We prove that each one of the conditions (ii), (iii) and (iv) implies that $\Gamma$ has strong property (U) (see Definition \ref{PropU1}), so \textup{(i)} will follow by the eigenvalue gap characterization from \cite[Prop. 1.2]{kassel-potrie}. The uniform gap summation property is a summability condition for gaps between singular values, see \cite[Def. 5.2]{GGKW} and Definition \ref{ugspdef} for the precise definitions. For example, condition \textup{(iii)} of the previous theorem is satisfied when there exist $1 \leq j \leq m-1$ and $C,c>1$ \hbox{such that for every $\gamma \in \Gamma$} $$\log \frac{\sigma_j(\rho(\gamma))}{\sigma_{j+1}(\rho(\gamma))} \geq c\log |\gamma |_{\Gamma}-C.$$

For the proof of implication $\textup{(ii)} \Rightarrow \textup{(i)}$ in Theorem \ref{wg} we establish that a torsion-free finitely generated group whose Floyd boundary is uncountable, satisfies strong property (U).

\begin{theorem} \label{nontrivial} Let $\Gamma$ be a finitely generated group and fix $|\cdot|_{\Gamma}:\Gamma\rightarrow \mathbb{N}$ a word length function on $\Gamma$. Suppose that there exists a Floyd function $f:\mathbb{N}\rightarrow (0,\infty)$ such that the Floyd boundary $\partial_{f}\Gamma $ of $\Gamma$ is non-trivial. Let $H$ be a torsion-free subgroup of $\Gamma$ whose limit set $\Lambda(H)$ in $\partial_{f}\Gamma$ contains at least three points. Then there exists a finite subset $F$ of $H$ and $C>0$, depending only on $H$, with the property: for every $\gamma \in H$ there exists $g \in F$ such that $$|g \gamma |_{\Gamma}-|g \gamma |_{\infty} \leqslant C.$$ In particular, if $\Gamma$ is virtually torsion-free then it satisfies strong property (U). \end{theorem} 

As a corollary of the previous theorem we deduce that a non-virtually nilpotent group which admits a representation with the uniform gap summation property admits a non-trivial Floyd boundary.

\begin{corollary} \label{ugspu} Let $\Gamma$ be a finitely generated group which is not virtually nilpotent, $G$ a semisimple Lie group and $\theta \subset \Delta$ a subset of simple restricted roots of $G$. Let $\rho:\Gamma \rightarrow G$ be a representation  which satisfies the uniform gap summation property with respect to $\theta$ and a Floyd function \hbox{$f:\mathbb{N}\rightarrow (0,\infty)$.} Then the Floyd boundary $\partial_{f}\Gamma $ of $\Gamma$ with respect to $f$ is non-trivial. In particular, $\Gamma$ satisfies strong property (U). \end{corollary}

\subsection{Characterizations of strongly convex cocompact groups} \label{scc0} Anosov representations of hyperbolic groups are closely related to real projective geometry and geometric structures. Fix an integer $d \geqslant 3$. A subset $\Omega$ of the projective space $\mathbb{P}(\mathbb{R}^d)$ is called \emph{properly convex} if it is contained in an affine chart on which $\Omega$ is bounded and convex. The domain $\Omega$ is called \emph{strictly convex} if it is properly convex and $\partial \Omega$ does not contain projective line segments. \par Let $\Gamma$ be a discrete subgroup of $\mathsf{PGL}_d(\mathbb{R})$ which preserves a properly convex domain $\Omega$ of $\mathbb{P}(\mathbb{R}^d)$. The full orbital limit set $\Lambda_{\Omega}(\Gamma)$ of $\Gamma$ in $\Omega$ is the set of accumulation points of all $\Gamma$-orbits in $\partial \Omega$ (see \cite[Def. 1.10]{DGK0}). The group $\Gamma$ acts \emph{convex cocompactly on $\Omega$} if the convex hull of $\Lambda_{\Omega}(\Gamma)$ in $\Omega$ is non-empty and has compact quotient by $\Gamma$ (see \cite[Def. 1.11]{DGK0}). The group $\Gamma$ is called {\em strongly convex cocompact in $\mathbb{P}(\mathbb{R}^d)$} if it acts convex cocompactly on some properly convex domain $\Omega$ with strictly convex and $C^1$-boundary. The work of Danciger--Gu\'eritaud--Kassel \cite{DGK0} and independently of Zimmer \cite{Zimmer}, shows that Anosov representations can be essentially (up to composition with a Lie group homomorphism) viewed as convex cocompact actions on properly convex domains in some real projective space. We refer the reader to \cite[Thm. 1.4 \& 1.15]{DGK0} and \cite[Thm. 1.22 \& 1.25]{Zimmer}. There are also related results in the more broad setting of naively convex cocompact groups, \hbox{see \cite[Thm. 1.13]{IZ}.}

\par For the definition of a {\em $P_k$-Anosov} representation $\rho:\Gamma \rightarrow G$, where $G$ is either $\mathsf{PGL}_d(\mathbb{R})$ or $\mathsf{GL}_d(\mathbb{R})$, we refer to Definition \ref{Def-of-Anosov-rep}. The following result from \cite{DGK0} offers a connection between Anosov representations and strongly convex cocompact actions on properly convex domains.

\begin{theorem} \label{sccompact} \textup{(}\cite[Thm. 1.4]{DGK0}\textup{)} Let $\Gamma$ be an infinite discrete subgroup of $\mathsf{PGL}_d(\mathbb{R})$ which preserves a properly convex domain of $\mathbb{P}(\mathbb{R}^d)$. Then $\Gamma$ is strongly convex cocompact in $\mathbb{P}(\mathbb{R}^d)$ if and only if $\Gamma$ is word hyperbolic and the natural inclusion $\Gamma \xhookrightarrow{} \mathsf{PGL}_d(\mathbb{R})$ is $P_{1}$-Anosov. \end{theorem}

For a properly convex domain $\Omega\subset \mathbb{P}(\mathbb{R}^d)$ let $d_{\Omega}$ be the Hilbert metric defined on $\Omega$. As an application of Theorem \ref{maintheorem}, we obtain the following geometric characterization of strongly convex cocompact subgroups of $\mathsf{PGL}_d(\mathbb{R})$ which are semisimple, i.e. their Zariski closure in $\mathsf{PGL}_d(\mathbb{R})$ is a reductive Lie group.

\begin{theorem} \label{stronglyconveccocompact} Let $\Gamma$ be a finitely generated subgroup of $\mathsf{PGL}_d(\mathbb{R})$. Suppose that $\Gamma$ preserves a strictly convex domain of $\mathbb{P}(\mathbb{R}^d)$ with $C^1$-boundary and the natural inclusion $\Gamma \xhookrightarrow{} \mathsf{PGL}_d(\mathbb{R})$ is semisimple. Then the following conditions are equivalent:
\begin{enumerate}[label=(\roman*)]
\item $\Gamma$ is strongly convex cocompact in $\mathbb{P}(\mathbb{R}^d)$.
\item The inclusion $\Gamma \xhookrightarrow{} \mathsf{PGL}_d(\mathbb{R})$ is a quasi-isometric embedding, $\Gamma$ preserves a properly convex domain $\Omega$ of $\mathbb{P}(\mathbb{R}^d)$ and there exists a $\Gamma$-invariant closed convex subset $\mathcal{C}$ of $\Omega$ such that $\big(\mathcal{C},d_{\Omega} \big)$ is Gromov hyperbolic.\end{enumerate}\end{theorem}

The previous theorem generalizes the well-known fact that a discrete subgroup $\Gamma$ of $\mathsf{PO}(d,1)$, $d \geq 2$, is convex cocompact if and only if $\Gamma \xhookrightarrow{} \mathsf{PO}(d,1)$ is a quasi-isometric embedding.

\subsection{Gromov product.} We also intoduce a definition of a Gromov product on $G \times G$ which we use for the proof of Theorem \ref{stronglyconveccocompact} (see Lemma \ref{main}). Let us remark that there are similar notions of Gromov products in \cite[\S 3]{Beyrer} and \cite[\S 8]{BPS} defined on appropriate flag spaces of $G$. The Gromov product from \cite{BPS} is also vector valued into a Cartan subspace of the Lie algebra of $G$.

\begin{definition} \label{defproduct} Let $G$ be a real semisimple Lie group. For every linear form $\varphi \in \mathfrak{a}^{\ast}$, define the Gromov product relative to $\varphi$ to be the map $(\, \cdot \,)_{\varphi}:G\times G \rightarrow \mathbb{R}$ defined as follows: for $g,h \in G$, $$\big(g\cdot h)_{\varphi}:=\frac{1}{4} \varphi\Big(\mu(g)+\mu(g^{-1})+\mu(h)+\mu(h^{-1})-\mu(g^{-1}h)-\mu(h^{-1}g) \Big).$$ \end{definition} 

 We prove that for every $P_{\theta}$-Anosov representation $\rho:\Gamma \rightarrow G$, the restriction of the Gromov product on $\rho(\Gamma)\times \rho(\Gamma)$, with respect to a fundamental weight $\omega_{\alpha}$, $\alpha \in \theta$, grows coarsely as the Gromov product on $\Gamma\times \Gamma$ with respect to a world length function on $\Gamma$.

\begin{proposition} \label{Gromovproduct1} Let $G$ be a real semisimple Lie group, fix $\theta \subset \Delta$ a subset of simple restricted roots of $G$ and let $\{\omega_{\alpha}\}_{\alpha \in \theta}$ be the associated set of fundamental weights. Suppose that $\Gamma$ is a word hyperbolic group and $\rho:\Gamma \rightarrow G$ is a $P_{\theta}$-Anosov representation. There exist $C,c>1$ with the property that for every $\alpha \in \theta$ and $ \gamma_1 ,\gamma_2 \in \Gamma$ we have $$C^{-1}(\gamma_1 \cdot \gamma_2)_{e} -c\leqslant \big(\rho(\gamma_1)\cdot \rho(\gamma_2)\big)_{\omega_{\alpha}} \leqslant C(\gamma_1 \cdot \gamma_2)_{e}+c.$$  \end{proposition}

 We remark that in the case where $\omega_{\alpha}=\varepsilon_1$, where $\varepsilon_1(x_1,\ldots,x_m)=x_1$ is the projection in the first coordinate, the double inequality in the previous proposition is not enough to guarantee that $\rho$ is $P_1$-Anosov (see Example \ref{complex-quaternionic}). However, if $\rho:\Gamma \rightarrow \mathsf{PGL}_d(\mathbb{R})$ preserves a properly convex domain $\Omega$ of $\mathbb{P}(\mathbb{R}^d)$ with strictly convex and $C^1$-boundary and the Gromov product on the Cartan projection of $\rho(\Gamma)$ with respect to $\varepsilon_1\in \mathfrak{a}^{\ast}$ grows coarsely as the Gromov product on $\Gamma$, then $\rho$ is $P_1$-Anosov (see Proposition \ref{main}).
\par We prove Proposition \ref{Gromovproduct1} as follows: by \cite[Prop. 1.8]{GGKW} any semisimplification $\rho^{ss}$ of $\rho$ is $P_{\theta}$-Anosov and hence, by using Lemma \ref{weights}, we may replace $\rho$ with $\rho^{ss}$. Then we compare the Gromov product relative to the fundamental weight $\{\omega_{\alpha}\}_{\alpha \in \theta}$ with the Gromov product with respect to the Hilbert metric $d_{\Omega}$ for some properly convex domain and then use Theorem \ref{sccompact}.
\subsection*{Comparison to previous characterizations and related results.} We first explain how Theorem \ref{maintheorem} is related to the equivalence $(3) \Leftrightarrow (5)$ in \cite[Thm. 1.7]{KLP3}, see also \cite[Thm. 5.47]{KLP1}. A subgroup $\Gamma$ of a real reductive Lie group $G$ is called $\tau_{\textup{mod}}$-asymptotically embedded \cite[Def. 6.12]{KLP3}, if it is \hbox{$\tau_{\textup{mod}}$-regular}, \hbox{$\tau_{\textup{mod}}$-antipodal}, word hyperbolic and there exists a $\Gamma$-equivariant homeomorphism $\nu:\partial_{\infty}\Gamma \rightarrow \Lambda_{\tau_{\textup{mod}}}(\Gamma)$. Here $\tau_{\textup{mod}}$ corresponds to the choice of a subset of simple restricted roots $\eta \subset \Delta$ of $G$, \hbox{$\tau_{\textup{mod}}$-antipodal} means that the map $\nu$ is transverse to itself i.e. for $x \neq y$ the pair $(\nu(x),\nu(y))$ is transverse and \hbox{$\tau_{\textup{mod}}$-regular} corresponds to $P_{\eta}$-divergence. 
\par Theorem \ref{maintheorem} follows from a theorem of Kapovich--Leeb--Porti \cite[Thm. 1.7]{KLP3} in the case where both limit maps \hbox{$\xi^{+}:\partial_{\infty}\Gamma \rightarrow G/P_{\theta}^{+}$} and \hbox{$\xi^{-}:\partial_{\infty}\Gamma \rightarrow G/P_{\theta}^{-}$} satisfy the Cartan property (see Definition \ref{defcartan}). Under this assumption, there exists $\rho$-equivariant embedding $\xi:\partial_{\infty}\Gamma \rightarrow G/P$ with $P=P_{\theta}^{+}\cap P_{\theta^{\star}}^{+}$, where $^{\star}:\Delta \rightarrow \Delta$ denotes the opposition involution and $\theta^{\star}=\{\alpha^{\star}:\alpha \in \theta\}$. Note that the pair of maps $(\xi^{+},\xi^{-})$ is compatible and transverse, hence $\xi$ is injective. The map $\xi$ satisfies the Cartan property, maps onto the $\tau_{\textup{mod}}$-limit set $\Lambda_{\tau_{\textup{mod}}}(\rho(\Gamma))$ hence $\rho(\Gamma)$ is $\tau_{\textup{mod}}$-asymptotically embedded and the assumptions of \cite[Thm. 1.7]{KLP3} are satisfied.
\par We also remark that Guichard-Gu\'eritaud-Kassel-Wienhard proved in \cite[Thm. 1.3, (1)$\Leftrightarrow$(2)]{GGKW} that a representation $\rho:\Gamma \rightarrow G$ is $P_{\theta}$-Anosov if and only if $\rho$ is $P_{\theta}$-divergent and admits a pair of continuous, $\rho$-equivariant, dynamics preserving and transverse maps $\xi^{\pm}:\partial_{\infty}\Gamma \rightarrow G/P_{\theta}^{\pm}$. Theorem \ref{maintheorem} follows by \cite[Thm. 1.3, (1)$\Leftrightarrow$(2)]{GGKW} under the additional assumption that both limit maps are dynamics preserving. 

\subsection*{Organization of the paper.} In \S \ref{Background} we provide the necessary background from Lie theory, hyperbolic groups and the notion of the Floyd boundary and recall Labourie's dynamical definition of Anosov representations. In  \S\ref{contraction} we prove some preliminary results which we use for the proof of Theorem \ref{maintheorem}. In \S \ref{scartan} we define the Cartan property for an equivariant map $\xi:\partial_{\infty}\Gamma \rightarrow G/P_{\theta}^{\pm}$ and discuss the uniform gap summation property of \cite{GGKW} in the more general setting of finitely generated groups. In \S\ref{PropU} we discuss (strong) property (U) and prove Theorem \ref{wg} and Corollary \ref{ugspu}. In \S \ref{Gromovproduct} we define a Gromov product for a representation $\rho$ and prove that is comparable with the usual Gromov product on the domain group when $\rho$ is Anosov. Next, in \S \ref{Main} we prove Theorem \ref{maintheorem} and in \S \ref{Scc} we give the proof of Theorem \ref{stronglyconveccocompact}. In \S \ref{sub} we provide conditions for the direct product of two representations to be Anosov. Finally, in  \S \ref{examples} we provide examples of discrete and faithful representations of surface groups showing that the assumptions of our main results are necessary.

\subsection*{Acknowledgements.} I would like to thank Richard Canary and Fanny Kassel for their support during the course of this work and for their comments in previous versions of this paper. I would also like to thank Sami Douba, Michael Kapovich, Giuseppe Martone, Andr\'es Sambarino and Feng Zhu for helpful discussions. Finally, I would also like to thank the anonymous referee for carefully reading the paper and their comments and suggestions that improved the exposition. The author was partially supported by grants DMS-1564362 and DMS-1906441 from the National Science Foundation as well as from the European Research Council (ERC) under the European's Union Horizon 2020 research and innovation programme (ERC starting grant DiGGeS, grant agreement No 715982).

\section{Background} \label{Background}
In this section, we recall definitions from Lie theory, review several facts for hyperbolic groups, the Floyd boundary, provide Labourie's dynamical definition of Anosov representations and also discuss several facts for semisimple representations. \hbox{We mainly follow the notation from \cite[\S 2]{GGKW}.} 

\subsection*{Conventions} Throughout this paper $\Gamma$ is a finitely generated group equipped with a finite generating subset $S$, inducing a left invariant word metric $d_{\Gamma}$ on the Cayley graph $C_{\Gamma}$ of $\Gamma$. For $\gamma \in \Gamma$ we set $|\gamma|_{\Gamma}:=d_{\Gamma}(\gamma,e)$, where $e\in \Gamma$ is the identity element. A linear representation $\rho:\Gamma \rightarrow \mathsf{GL}_d(\mathbb{R})$, $d \geq 2$, is called {\em irreducible} if $\rho(\Gamma)$ does not preserve any non-trivial proper vector subspace of $\mathbb{R}^d$. The representation $\rho$ is called {\em strongly irreducible} if for every finite-index subgroup $H$ of $\Gamma$ the restriction $\rho|_{H}$ is irreducible. We  equip the vector space $\mathbb{R}^d$ with the canonical basis $(e_1,\ldots,e_d)$, where $e_i$ is the vector with $1$ on the $i^{\textup{th}}$ coordiante and zero everywhere else, and the standard Euclidean inner product $\langle \cdot,\cdot\rangle$. For a subspace $V\subset \mathbb{R}^d$, $V^{\perp}=\{v\in \mathbb{R}^d: \langle v,v'\rangle=0,  \forall v' \in V\}$ is the orthogonal complement of $V$.

\subsection{Lie theory} We will always consider $G$ to be a semisimple Lie subgroup of $\mathsf{SL}_m(\mathbb{R})$, $m \geqslant 2$, of non-compact type with finitely many connected components. The Zariski topology on $G$ is the subspace topology induced from the Zariski topology on $\mathsf{SL}_m(\mathbb{R})$.

% We denote by \hbox{$\textup{Ad}:G \rightarrow \mathsf{GL}(\mathfrak{g})$} the adjoint representation and by $\textup{ad}:\mathfrak{g} \rightarrow \mathfrak{gl}(\mathfrak{g})$ its derivative. The Killing form $B: \mathfrak{g}\times \mathfrak{g} \rightarrow \mathbb{R}$ is the bilinear form $B(X,Y)=\textup{tr}( \textup{ad}_{X}\circ \textup{ad}_{Y})$ and is non-degenerate as soon as $\mathfrak{g}$ is semisimple.

\par We fix a maximal compact subgroup $K$ of $G$, unique up to conjugation, a Cartan decomposition $\mathfrak{g}=\mathfrak{t}\oplus \mathfrak{p}$ where $\mathfrak{t}=\textup{Lie}(K)$, $\mathfrak{p}$ is the orthogonal complement of $\mathfrak{t}$ with respect to the Killing form on $\mathfrak{g}$, and the Cartan subspace $\mathfrak{a}\subset \mathfrak{g}$ which is a maximal abelian subalgebra of $\mathfrak{g}$ contained in $\mathfrak{p}$. The {\em real rank} of $G$ is the dimension of $\mathfrak{a}$ as a real vector space. 

%For a linear form $\beta \in \mathfrak{a}^{\ast}$ we shall use the notation $\left \langle \beta, H \right \rangle=\beta(H)$ for $H \in \mathfrak{a}$. 

There is a decomposition of $\mathfrak{g}$ into the common eigenspaces of the transformations $X\mapsto [H,X],H \in \mathfrak{a}$, called the restricted root decomposition $$\mathfrak{g}=\mathfrak{g}_0 \oplus \bigoplus_{\alpha \in \Sigma} \mathfrak{g}_{\alpha}$$ where $\mathfrak{g}_{\alpha}=\big \{ X \in \mathfrak{g}:[H,X]= \alpha(H) X , \forall H \in \mathfrak{a} \big \}$ and $\Sigma=\big \{\alpha \in \mathfrak{a}^{\ast}: \mathfrak{g}_{\alpha} \neq 0 \big \}$ is the set of restricted roots of $G$. Fix $H_0 \in \mathfrak{a}$ with $\alpha(H_0) \neq 0$ for every $\alpha \in \Sigma$. Denote by $\Sigma^{+}=\big \{\alpha \in \Sigma: \alpha( H_0)>0\big \}$ the set of positive roots and fix $\Delta \subset \Sigma^{+}$ the simple positive roots. For any simple restricted root $\alpha \in \Delta$, denote by $\omega_{\alpha}$ the \emph{fundamental weight} with respect to $\alpha\in \Delta$, see \cite[\S 3.1]{GGKW}.

 \par For every $\theta \subset \Delta$, $\Sigma_{\theta}$ denotes the set of all roots in $\Sigma$ which are linear combinations of elements of $\theta$. We consider the parabolic Lie algebras $$\mathfrak{p}_{\theta}^{\pm}=\mathfrak{g}_{0} \oplus \bigoplus_{\alpha \in \Sigma^{\pm}\cup \Sigma_{\Delta \smallsetminus \theta}} \mathfrak{g}_{\alpha}$$ and denote by $P_{\theta}^{\pm}=N_{G}(\mathfrak{p}_{\theta}^{\pm})$. A subgroup $P$ of $G$ is {\em parabolic} if it normalizes some parabolic subalgebra. A pair of parabolic subgroups $(P^{+},P^{-})$ of $G$ are called {\em opposite} if there $\theta \subset \Delta$ and $g \in G$ such that $(P^{+},P^{-})=(gP_{\theta}^{+}g^{-1},gP_{\theta}^{-}g^{-1})$. 

 Let $\overline{\mathfrak{a}}^{+}:= \big\{ H \in \mathfrak{a}:\alpha (H) \geqslant 0,  \forall \alpha \in \Delta  \big\}$. There exists a decomposition $$G=K\exp(\overline{\mathfrak{a}}^{+})K$$ called the {\em Cartan decomposition} where each element $g \in G$ is written as $$g=k_{g}\exp(\mu(g))k_{g}'\ \ k_{g},k_{g}' \in K,$$ and $\mu(g)\in \overline{\mathfrak{a}}^{+}$ denotes the Cartan projection of $g$. The map $\mu:G \rightarrow \overline{\mathfrak{a}}^{+}$ is called the {\em Cartan projection} and is continuous and proper. The {\em Lyapunov projection} $\lambda:G \rightarrow \overline{\mathfrak{a}}^{+}$ is the map defined as follows for $g\in G$, $$\lambda(g)=\lim_{n \rightarrow \infty} \frac{1}{n}\mu(g^n).$$ 

An element $g \in G$ is called {\em $P_\theta$-proximal} if $\min_{\alpha \in \theta} \alpha( \lambda(g)))>0$. Equivalently, $g$ has two fixed points $x_{g}^{+}\in G/P_{\theta}^{+}$ and \hbox{$V_{g}^{-} \in  G/P_{\theta}^{-}$} such that the pair $(x_{g}^{+},V_{g}^{-})$ is transverse and for every $x \in G/P_{\theta}^{+}$  transverse to $V_{g}^{-}$, we have $\lim_{n}g^n x=x_{g}^{+}$. The element $g$ is called $P_\theta$-\emph{biproximal} if $g$ and $g^{-1}$ are both $P_{\theta}$-proximal and we denote by $x_{g}^{-}$ the attracting fixed point of $g^{-1}$ in $G/P_{\theta}^{-}$. 

\medskip

For a matrix $h=(h_{ij})_{i,j=1}^{d}$ in $\mathsf{GL}_d(\mathbb{R})$ its transpose is $h^{t}:=(h_{ji})_{i,j=1}^{d}$.

\begin{Example} The case of $G=\mathsf{SL}_d(\mathbb{R})$. \normalfont{Recall that $(e_1,\ldots,e_d)$ denotes the canonical basis of $\mathbb{R}^d$ and $e_{j}^{\perp}:=\bigoplus_{j \neq i}\mathbb{R}e_j$. The group $\mathsf{SO}(d)=\{g \in \mathsf{SL}_d(\mathbb{R}):gg^{t}=I_d \big \}$ is the unique, up to conjugation, maximal compact subgroup of $\mathsf{SL}_d(\mathbb{R})$. A Cartan subspace for $\mathfrak{g}$ is the subspace $\mathfrak{a}=\textup{diag}_{0}(d)$ of all diagonal matrices with zero trace. Let $\varepsilon_{i} \in \mathfrak{a}^{\ast}$ be the projection to the $(i,i)$-entry. The closed dominant Weyl chamber of $\mathfrak{a}$ is $\overline{\mathfrak{a}}^{+}:=\big \{\textup{diag}(a_1,\ldots,a_d):a_1 \geqslant \ldots \geqslant a_d, \ \sum_{i=1}^{d}a_{i}=0\big \}$ and we have the Cartan decomposition $\mathsf{SL}_d(\mathbb{R})=\mathsf{SO}(d)\exp(\overline{\mathfrak{a}}^{+})\mathsf{SO}(d)$. The restricted root decomposition is $\mathfrak{sl}_d(\mathbb{R})=\mathfrak{a} \oplus \bigoplus_{i \neq j} \mathbb{R}E_{ij}$, where $E_{ij}$ denotes the $d \times d$ elementary matrix with $1$ at the $(i,j)$ entry and $0$ everywhere else. The set of restricted roots is $\big\{ \varepsilon_{i}-\varepsilon_{j}: i \neq j \big \}$ and of simple positive roots $\big\{\varepsilon_{i}-\varepsilon_{i+1}:i=1,\ldots,d-1 \big\}$. For each $i=1,\ldots,d-1$, the associated fundamental weight is $\omega_{\varepsilon_i-\varepsilon_{i+1}}=\sum_{k=1}^{i}\varepsilon_{k}$. For an element $g \in \mathsf{SL}_d( \mathbb{R})$ we denote by $\sigma_i(g)$ and $\ell_i(g)$ the $i$-th singular value and modulus of eigenvalues of $g$. Recall the connection between moduli of eigenvalues and singular values $\sigma_{i}(g)=\sqrt[]{\ell_i(gg^t)}$. The Cartan and Lyapunov projections of \hbox{$g\in \mathsf{SL}_d(\mathbb{R})$ respectively are} \begin{align*} \mu(g) &=\textup{diag}\big(\log \sigma_1(g),\ldots,\log \sigma_d(g)\big)\\ \lambda(g) & =\textup{diag}\big(\log \ell_1(g),\ldots,\log\ell_d(g)\big).\end{align*}
For any integer $1 \leq i \leq \frac{d}{2}$ we denote by $P_{i}^{+}$ (resp. $P_{i}^{-}$) the stabilizer of the plane $\langle e_1,\ldots, e_i\rangle$ (resp. $\langle e_{i+1},\ldots, e_{d}\rangle$). The pair of parabolic subgroups $(P_{i}^{+},P_{i}^{-})$ is opposite. An element $g \in \mathsf{GL}_d(\mathbb{R})$ is  $P_i$-proximal if and only if $\ell_{i}(g)>\ell_{i+1}(g)$. In this case $g$ admits a unique attracting fixed point in the flag space $G/P_{i}^{+}=\mathsf{Gr}_{i}(\mathbb{R}^d)$.}\end{Example}

\subsection{Gromov hyperbolic spaces.}\label{Grhyp} Let $(X,d)$ be a proper geodesic metric and $x_0 \in X$ a fixed basepoint. For an isometry $\gamma:X\rightarrow X$  define $|\gamma|_{X}:=d(\gamma x_0,x_0)$. The {\em translation length} and the {\em stable translation length} of the isometry $\gamma$ respectively are: $$\ell_{X}(\gamma)=\inf_{x \in X}d(\gamma x,x), \ |\gamma|_{X,\infty}=\lim_{n \rightarrow \infty} \frac{|\gamma^n|_{X}}{n}.$$ The \emph{Gromov product} with respect to $x_0$ is the map $X\times X\rightarrow [0,\infty)$ defined as follows $$(x\cdot y)_{x_0}:=\frac{1}{2}\Big(d(x,x_0)+d(y, x_0)-d(x,y) \Big).$$ 
\par A proper geodesic metric space space $(X,d)$ is called \emph{Gromov hyperbolic} if there exists $\epsilon \geqslant 0$ with the following property: for every $x,y,z \in X$ $$(x \cdot y)_{x_0} \geqslant \min  \big \{(x \cdot z)_{x_0}, (z \cdot y)_{x_0}  \big \}-\epsilon.$$ The {\em Gromov boundary} of $X$ is denoted by $\partial_{\infty}X$.

\par A finitely generated group $\Gamma$ is called {\em word hyperbolic} (or {\em Gromov hyperbolic}) if the Cayley graph of $\Gamma$ equipped with the word metric $d_{\Gamma}$ is a Gromov hyperbolic space. In this case, every infinite order element $\gamma \in \Gamma$ has exactly two fixed points $\gamma^{+}, \gamma^{-}\in \partial_{\infty}\Gamma$, called the attracting and repelling fixed points of $\gamma$ respectively. For more details on Gromov hyperbolic  spaces and their boundaries we refer the reader to \cite[Chap. III.H \& III.$\Gamma$]{BH} and \cite{CDP}.

\subsection{The Floyd boundary.}  A non-increasing function $f:\mathbb{N} \rightarrow (0,\infty)$ is called a {\em Floyd function} if it  satisfies the following two conditions:\\
\noindent \textup{(i)} \hbox{$\sum_{n=1}^{\infty} f(n)<+\infty$}.\\
\noindent \textup{(ii)} there exists $0<\epsilon<1$ such that $\epsilon f(n) \leqslant f(n+1) \leqslant f(n)$ for every $n \in \mathbb{N}$.
\medskip

 Let $\Gamma$ be a finitely generated group. Given a Floyd function $f:\mathbb{N}\rightarrow (0,\infty)$ there exists a metric $d_{f}$ on the Cayley graph of $\Gamma$ with respect to $S$ defined as follows (see \cite{Floyd}): for two adjacent vertices $g, h \in \Gamma$ their distance is defined as $d_{f}(g,h)=f(\max\{|g|_{\Gamma},|h|_{\Gamma}\})$. The length of a finite path $\p$ defined by the sequence of adjacent vertices $\p=\{x_{0},x_{1},\ldots,x_{k}\}$ is $L_{f}(\p)=\sum_{i=0}^{k-1}d_{f}(x_{i},x_{i+1})$. For two arbitrary vertices $g,h \in \Gamma$ their distance is $d_{f}(g,h)=\inf \big \{L_{f}(\p): \p \ \textup{is a path from} \ g \ \textup{to} \ h \big \}$. It is easy to verify that $d_{f}$ defines a metric on $\Gamma$ and let $\overline{\Gamma}$ be the the metric completion of $\Gamma$ with respect to $d_f$. Every two points $x,y \in \overline{\Gamma}$ are represented by Cauchy sequences $(\gamma_n)_{n \in \mathbb{N}}, (\delta_n)_{n \in \mathbb{N}}$ with respect to $d_f$ and their distance is $d_{f}(x,y)=\lim_{n}d_{f}(\gamma_n,\delta_n)$. The \emph{Floyd boundary of $\Gamma$ with respect to $f$} is defined to be the complement $\partial_{f}\Gamma:=\overline{\Gamma} \smallsetminus \Gamma$ equipped with the metric $d_{f}$. The Floyd boundary $\partial_{f}\Gamma$ is called {\em non-trivial} if it contains at least three points. For every infinite order element $\gamma \in \Gamma$  the limit $\lim_{n \rightarrow \infty}\gamma^{n}$ exists (see for example \cite[Prop. 4]{Karlsson}) \hbox{and is denoted by $\gamma^{+}$.}
\par If $\Gamma$ is a word hyperbolic group, there exists $\varepsilon>0$ such that the Floyd boundary of $\Gamma$ with respect to $f(x)=e^{-\varepsilon x}$ is the Gromov boundary of $\Gamma$ equipped with a visual metric (see \cite{Gromov}). For more details and properties of the Floyd boundary we refer the reader to \cite{Floyd, Gromov, Karlsson}.

\subsection{Flow spaces for hyperbolic groups.}\label{Flows} Flow spaces for hyperbolic groups were introduced by Gromov in \cite{Gromov} and further developed by Champetier \cite{Champetier} and Mineyev \cite{Mineyev}. For any word hyperbolic group $\Gamma$ there exists a metric space $\big( \hat{\Gamma}, \varphi_{t} \big)$ equipped with an $\mathbb{R}$-action $\{\varphi_{t}\}_{t \in \mathbb{R}}$ called the geodesic flow with the following properties:\\
\noindent (a) The action of $\Gamma$ commutes with the action of the geodesic flow.

\noindent (b) The group $\Gamma$ acts properly discontinuously and cocompactly with isometries \hbox{on the flow space $\hat{\Gamma}$.}

\noindent (c) There exist $C,c>0$ such that for every $\hat{m} \in \hat{\Gamma}$, the map $t \mapsto\varphi_{t}(\hat{m})$ is a $(C,c)$-quasi-isometric embedding \hbox{$(\mathbb{R},d_{\mathbb{E}}) \rightarrow (\hat{\Gamma},d_{\hat{\Gamma}})$}.

\noindent The last property guarantees that the map $(\tau^{+},\tau^{-}):\hat{\Gamma} \rightarrow \partial_{\infty} \Gamma\times \partial_{\infty}\Gamma \smallsetminus \big\{ (x,x) \ | \ x \in \partial_{\infty}\Gamma \big \}$ $$\big(\tau^{+}(\hat{m}), \tau^{-}(\hat{m})\big)=\Big( \lim_{t \rightarrow \infty} \varphi_{t}(\hat{m}), \ \lim_{t \rightarrow \infty}\varphi_{-t}(\hat{m}) \Big)$$ is well-defined, continuous and equivariant with respect to the action of $\Gamma$. For example, if $(M, g)$ is a closed negatively curved Riemannian manifold, a flow space for $\pi_1(M)$ satisfying the previous conditions is the unit tangent bundle $T^{1}\widetilde{M}$ equipped with the standard geodesic flow. \par  Benoist proved that a torsion-free, discrete subgroup $\Gamma\subset \mathsf{PGL}_d(\mathbb{R})$ acting geometrically on a strictly convex domain $\Omega \subset \mathbb{P}(\mathbb{R}^d)$ is word hyperbolic (see \cite[Thm. 1]{benoist-divisible1}). A choice of a flow space for $\Gamma$ is the manifold $ T^1\Omega$ equipped with the Hilbert geodesic flow.

\subsection{Anosov representations}\label{Anosovdfn}
Let $\rho:\Gamma \rightarrow G$ be a representation and fix $\theta \subset \Delta$ a subset of simple restricted roots of $G$. We denote by $L_{\theta}=P_{\theta}^{+} \cap P_{\theta}^{-}$ the common Levi subgroup of $P_{\theta}^{+},P_{\theta}^{-}$. There exists a $G$-equivariant embedding $G/L_{\theta} \rightarrow G/P_{\theta}^{+} \times G/P_{\theta}^{-}$ mapping the coset $gL_{\theta}$ to the pair $(gP_{\theta}^{+},gP_{\theta}^{-})$. The tangent space of $G/L_{\theta}$ at $(gP_{\theta}^{+},gP_{\theta}^{-})$ splits as the direct sum $\mathsf{T}_{gP_{\theta}^{+}}G/P_{\theta}^{+} \oplus \mathsf{T}_{gP_{\theta}^{-}}G/P_{\theta}^{-}$ and induces a $G$-equivariant splitting of the tangent bundle $\mathsf{T}(G/L_{\theta})=\mathcal{E}\oplus \mathcal{E}^{-}$. We consider the quotient spaces: $$\mathcal{X}_{\rho}=\Gamma \backslash  \big(\hat{\Gamma} \times G/L_{\theta} \big), \   \mathcal{E}_{\rho}^{\pm}=\Gamma \backslash \big( \hat{\Gamma} \times \mathcal{E}^{\pm} \big)$$ where the action of $\gamma \in \Gamma$ on $\mathsf{T}(G/L_{\theta})$ is given by the differential $dL_{\rho(\gamma)}$ of the left translation by $\rho(\gamma)$, denoted $L_{\rho(\gamma)}:G/L_{\theta} \rightarrow G/L_{\theta}$. Let $\pi: \mathcal{X}_{\rho} \rightarrow \Gamma \backslash \hat{\Gamma}$ and $\pi_{\pm}: \mathcal{E}_{\rho}^{\pm} \rightarrow \mathcal{X}_{\rho}$ be the natural projections. The projections $\pi_{\pm}$ define vector bundles over the space $\mathcal{X}_{\rho}$ where the fiber over the point $[\hat{m}, (gP_{\theta}^{+},gP_{\theta}^{-})]_{\Gamma}$ is identified with the vector space $\mathsf{T}_{gP_{\theta}^{\pm}}G/P_{\theta}^{\pm}$. The geodesic flow $\{\varphi_{t}\}_{t \in \mathbb{R}}$ commutes with the action of $\Gamma$ and there exists a lift of the geodesic flow on the quotients $\mathcal{X}_{\rho}$ and $\mathcal{E}_{\rho}^{\pm}$ which we continue to denote by $\{\varphi_{t}\}_{t \in \mathbb{R}}$.

\begin{definition} \textup{(}\cite{GW, labourie}\textup{)}\label{Def-of-Anosov-rep} Let $\Gamma$ be a word hyperbolic group and fix $\theta \subset \Delta$ a subset of restricted roots of $G$. A representation $\rho:\Gamma \rightarrow G$ is called $P_{\theta}$-Anosov if:
\begin{enumerate}  \item There exists a section $\sigma:\Gamma \backslash \hat{\Gamma} \rightarrow \mathcal{X}_{\rho}$ flat along the flow lines.
\item The lift of the geodesic flow  $\{\varphi_t\}_{t \in \mathbb{R}}$ on the pullback bundle $\sigma_{\ast} \mathcal{E}^{+}$ \textup{(}resp. $\sigma_{\ast}\mathcal{E}^{-}$\textup{)} is dilating \textup{(}resp. contracting\textup{)}.\end{enumerate} \end{definition}

Two maps $\xi^{+}:\partial_{\infty}\Gamma \rightarrow G/P_{\theta}^{+}$ and $\xi^{-}:\partial_{\infty}\Gamma \rightarrow G/P_{\theta}^{-}$ are called \emph{transverse} if for any pair of distinct points $(x,y) \in  \partial^{(2)}_{\infty} \Gamma$ there exists $h \in G$ such that $(\xi^{+}(x),\xi^{-}(y))=(hP_{\theta}^{+},hP_{\theta}^{-})$. The previous definition is equivalent to the existence of a pair of continuous  $\rho$-equivariant transverse maps $\xi^{+}:\partial_{\infty} \Gamma \rightarrow  G/P_{\theta}^{+}$ and $\xi^{-}:\partial_{\infty}\Gamma \rightarrow G/P_{\theta}^{-}$ defining the flat section $\sigma: \Gamma \backslash \hat{\Gamma}\rightarrow \mathcal{X}_{\rho}$ $$\sigma([\hat{m}]_{\Gamma}):=\big[\hat{m}, (\xi^{+}(\tau^{+}(\hat{m})), \xi^{-}(\tau^{-}(\hat{m})))\big]_{\Gamma},$$ and a continuous equivariant family of norms $(||\cdot||_{x})_{x \in \Gamma \backslash \hat{\Gamma}}$ with the property that there exist \hbox{$C,a>0$} such that for every $x=[\hat{m}]_{\Gamma} $, $t \geqslant 0$, and $v \in \mathsf{T}_{\xi^{+}(\tau^{+}(\hat{m}))}G/P_{\theta}^{+}$ \textup{(resp. $v \in \mathsf{T}_{\xi^{-}(\tau^{-}(\hat{m}))}G/P_{\theta}^{-}$)}: $$\big|\big|\varphi_{-t} \big(X_{v}^{+} \big)\big|\big|_{\varphi_{-t}(x)} \leqslant Ce^{-at}\big|\big|X_{v}^{+}\big|\big|_{x} \ \ \ \big(\textup{resp.} \ \big|\big|\varphi_{t} \big(X_{v}^{-} \big)\big|\big|_{\varphi_{t}(x)} \leqslant Ce^{-at}\big|\big|X_{v}^{-}\big|\big|_{x}\big)$$ where $X_{v}^{+}$ (resp. $X_{v}^{-}$) denotes the copy of the vector $v\in \pi^{-1}_{+}(x)$ $\big(\textup{resp.} \ v\in \pi^{-1}_{-}(x) \big)$.
\medskip

We recall now some of the key properties of Anosov representations. For more background and for the main properties of Anosov representations  see \cite{Canary-notes, GGKW, GW, KLP1, KLP2, KLP3,labourie}. For a coset $gP_{\theta}^{\pm}$, the stabilizer $\textup{Stab}_{G}(gP_{\theta}^{\pm})$ is the parabolic subgroup $gP_{\theta}^{\pm}g^{-1}$ of $G$. A pair of maps \hbox{$\xi^{+}:\partial_{\infty}\Gamma \rightarrow G/P_{\theta}^{+}$} and $\xi^{-}:\partial_{\infty}\Gamma \rightarrow G/P_{\theta}^{-}$ are called \emph{compatible} if for any $x \in \partial_{\infty}\Gamma$ the intersection $\textup{Stab}_{G}(\xi^{+}(x)) \cap \textup{Stab}_{G}(\xi^{-}(x))$ is a parabolic subgroup of $G$. We also say that $\xi^{+}$ (resp. $\xi^{-}$) is \emph{dynamics preserving} if for every infinite order element $\gamma \in \Gamma$, $\rho(\gamma)$ is proximal in $G/P_{\theta}^{+}$ (resp. $G/P_{\theta}^{-}$) and $\xi^{+}(\gamma^{+})$ (resp. $\xi^{-}(\gamma^{+})$) is the attracting fixed point of $\rho(\gamma)$ in $G/P_{\theta}^{+}$ (resp. $G/P_{\theta}^{-}$). We fix an Euclidean norm $||\cdot||$ on the Cartan subspace $\mathfrak{a} \subset \mathfrak{g}$ and recall that $\mu:G \rightarrow \overline{\mathfrak{a}}^{+}$ denotes the Cartan projection. 

\begin{theorem}\label{mainproperties} \textup{(}\cite{GW,labourie,KLP2}\textup{)} Let $\Gamma$ be a word hyperbolic group and $\theta \subset \Delta$ a subset of simple restricted roots of $G$. Suppose that $\rho:\Gamma \rightarrow G$ is a $P_{\theta}$-Anosov representation.
\begin{enumerate}[label=(\roman*)]
\item There exist $C,c>1$ such that for every $\gamma \in \Gamma$, $$\min_{\alpha \in \theta}\alpha \big(\mu(\rho(\gamma)) \big) \geqslant c^{-1} \big|\big| \mu (\rho(\gamma))\big|\big|-c \geq C^{-1}|\gamma|_{\Gamma}-C.$$ In particular, $\rho$ is a quasi-isometric embedding, $\textup{ker}(\rho)$ is finite and $\rho(\Gamma)$ is discrete in $G$.

\item $\rho$ admits a pair of compatible, continuous, $\rho$-equivariant, dynamics preserving and transverse maps $\xi^{+}:\partial_{\infty}\Gamma \rightarrow G/P_{\theta}^{+}$ and $\xi^{-}:\partial_{\infty}\Gamma \rightarrow G/P_{\theta}^{-}.$
\item The set of $P_{\theta}$-Anosov representations of $\Gamma$ in $G$ is open in $\textup{Hom}(\Gamma,G)$ and the map assigning a $P_{\theta}$-Anosov representation to its Anosov limit maps is continuous.\end{enumerate}
\end{theorem}

\par Let $G$ be a semisimple linear Lie group. A representation \hbox{$\tau: G \rightarrow \mathsf{GL}_d(\mathbb{R})$} is called {\em proximal} if $\tau(G)$ contains a $P_1$-proximal element. For an irreducible and proximal  representation $\tau$ we denote by $\chi_{\tau}$ the highest weight of $\tau$. The functional $\chi_{\tau} \in \mathfrak{a}^{\ast}$ is of the form $\chi_{\tau}=\sum_{\alpha \in \Delta} n_{\alpha}\omega_{\alpha}$ and the representation $\tau$ is called \emph{$\theta$-compatible} if \hbox{$\theta=\big\{ \alpha \in \Delta: n_{\alpha}>0 \big \}$.}
\par The following result is the content of \cite[Prop. 4.3]{GW} and \cite[Lem. 3.7]{GGKW} and is used to reduce statements for $P_{\theta}$-Anosov representations to statements for $P_1$-Anosov representations.

\begin{proposition}\textup{(}\cite{GGKW,GW}\textup{)} \label{higherdimension} Let $G$ a real semisimple Lie group, $\theta \subset \Delta$ a subset of simple restricted roots of $G$. There exists an irreducible, $\theta$-compatible representation $\tau_{\theta}:G \rightarrow \mathsf{GL}_d(\mathbb{R})$, $d=d(G,\theta)$,  such that $\tau_{\theta}(P_{\theta}^{+})$ and $\tau_{\theta}(P_{\theta}^{-})$ stabilize the line $[e_1]$ and the hyperplane $e_{1}^{\perp}=\langle e_1,\ldots ,e_{d-1} \rangle$ respectively, so that there exist continuous and $\tau_{\theta}$-equivariant embeddings $$\iota^{+}:G/P_{\theta}^{+} \xhookrightarrow{} \mathbb{P}(\mathbb{R}^d), \ \iota^{-}:G/P_{\theta}^{-}\xhookrightarrow{} \mathsf{Gr}_{d-1}(\mathbb{R}^d)$$ induced by $\tau$. Moreover, a representation $\rho:\Gamma \rightarrow G$ is $P_{\theta}$-Anosov if and only if $\tau_{\theta} \circ \rho:\Gamma \rightarrow \mathsf{GL}_d(\mathbb{R})$ is $P_1$-Anosov. In this case, the pair of Anosov limit maps of $\tau_{\theta} \circ \rho$ is $(\iota^{+} \circ \xi^{+}, \iota^{-}\circ \xi^{-})$, where $(\xi^{+},\xi^{-})$ is the pair of the limit maps of $\rho$.\end{proposition}

\subsection{Semisimple representations.} Let $G$ be a semisimple Lie subgroup of $\mathsf{SL}_d(\mathbb{R})$ and \hbox{$\rho:\Gamma \rightarrow G$} a representation. The representation $\rho$ is called {\em semisimple} if $\rho$ is a direct sum of irreducible reprrsentations. In this case the Zariski closure of $\rho(\Gamma)$ in $G$ is a reductive algebraic Lie group. 

The following result was proved by Benoist using a result of Abels--Margulis--Soifer \cite{AMS} and allows one to control the Cartan projection of a semisimple representation in terms of its Lyapunov projection. We refer the reader to \cite[Thm. 4.12]{GGKW} for a proof.

\begin{theorem}\label{finitesubset} \textup{(}\cite{AMS} \& \cite{benoist-limitcone}\textup{)} Let $G$ be a real reductive Lie group, $\Gamma$ be a discrete group and \hbox{$\big\{\rho_{i}:\Gamma \rightarrow G\big\}_{i=1}^{s}$} semisimple representations. Then there exists $C>0$ and a finite subset $F$ of $\Gamma$ such that for every $\gamma \in \Gamma$ there exists $f \in F$ with the property: $$\max_{1 \leqslant i\leqslant s} \Big| \Big| \mu \big(\rho_i(\gamma)\big)-\lambda \big(\rho_{i}(\gamma f)\big)  \Big|\Big| \leqslant C$$\end{theorem}

 Gu\'eritaud--Guichard--Kassel--Wienhard in \cite{GGKW} observe that from $\rho$ one may define the \emph{semisimplification} $\rho^{ss}$ which is a semisimple representation and a limit of conjugates of $\rho$. We shall use the following result for the semisimplification of a representation.

\begin{proposition} \label{semisimplification} \textup{(}\cite[Prop. 1.8]{GGKW}\textup{)} \label{semisimplification} Let $\Gamma$ be a finitely generated group, $G$ a real semisimple Lie group, $\theta \subset \Delta$ a subset of simple restricted roots of $G$ and $\rho:\Gamma \rightarrow G$ be a representation with semisimplification $\rho^{ss}:\Gamma \rightarrow G$. Then for every $\gamma \in \Gamma$, $\lambda(\rho(\gamma))=\lambda(\rho^{ss}(\gamma))$  and $\rho$ is $P_{\theta}$-Anosov if and only if $\rho^{ss}$ is $P_{\theta}$-Anosov. \end{proposition}

\subsection{Convex cocompact groups} A subset $\Omega$ of the projective space $\mathbb{P}(\mathbb{R}^d)$ is called \emph{properly convex} if it is contained in an affine chart in which $\Omega$ is bounded and convex. The domain $\Omega$ is called \emph{strictly convex} if it is properly convex and $\partial \Omega$ does not contain projective line segments. Suppose that $\Omega$ is bounded  and convex in some affine chart $A$. We fix an Euclidean metric $d_{\mathbb{E}}$ on $A$. We denote by $d_{\Omega}$ the Hilbert metric on $\Omega$ defined as follows $$d_{\Omega}(x,y)=\frac{1}{2}\log \frac{d_{\mathbb{E}}(y,a) d_{\mathbb{E}}(x,b)}{d_{\mathbb{E}}(a,x) d_{\mathbb{E}}(y,b)},$$ where $a,b$ are the intersection points of the projective line $[x,y]$ with $\partial \Omega$, $x$ is between $a$ and $y$, and $y$ is between $x$ and $b$. The group $$\textup{Aut}(\Omega)=\big \{g \in \mathsf{PGL}_d(\mathbb{R}): g\Omega=\Omega \big \}$$ is a Lie subgroup of $\mathsf{PGL}_d(\mathbb{R})$ and acts by isometries for the Hilbert metric $d_{\Omega}$. Any discrete subgroup of $\textup{Aut}(\Omega)$ acts properly discontinuously on $\Omega$.

We shall use the following estimate obtained by Danciger--Gu\'eritaud--Kassel in \cite{DGK0} showing that the inclusion of a convex cocompact subgroup in $\mathsf{PGL}_d(\mathbb{R})$ is a quasi-isometric embeddeding.

\begin{proposition} \label{control1} \textup{(}\cite[Prop. 10.1]{DGK0}\textup{)} Let $\Omega$ be a properly convex domain of $\mathbb{P}(\mathbb{R}^d)$. For any $x_0 \in \Omega$, there exists $\kappa>0$ such that for any $g \in \textup{Aut}(\Omega)$, $$\frac{1}{2}\log \frac{\sigma_1(g)}{\sigma_d(g)} \geqslant d_{\Omega}(gx_0,x_0)-\kappa.$$ \end{proposition}

\noindent Let $\Gamma$ be a subgroup of $\mathsf{PGL}_d(\mathbb{R})$ preserving a properly convex domain $\Omega$. By using the previous proposition we can control the Gromov product with respect to $\varepsilon_1 \in \mathfrak{a}^{\ast}$ as follows.

\begin{lemma}\label{control} Let $\Gamma$ be a subgroup of $\mathsf{PGL}_d(\mathbb{R})$ which preserves a properly convex domain $\Omega$ of $\mathbb{P}(\mathbb{R}^d)$. Suppose that the natural inclusion of $\Gamma \xhookrightarrow{} \mathsf{PGL}_d(\mathbb{R})$ is semisimple. Then for every $x_0 \in \Omega$ there exists $C>0$ such that  for every $\gamma, \delta \in \Gamma$, $$\Big|\frac{1}{2} \log \frac{\sigma_1(g)}{\sigma_d(g)}- d_{\Omega}(\gamma x_0,x_0) \Big| \leqslant C,\  \Big|(\gamma \cdot \delta)_{\varepsilon_1}-(\gamma x_0\cdot \delta x_0)_{x_0} \Big| \leqslant C.$$ \end{lemma}

\begin{proof} By Theorem \ref{finitesubset} there exists a finite subset $F$ of $\Gamma$  and $M>0$ such that for every $\gamma \in \Gamma$ there exists $f \in F$ such that $\log \frac{\ell_1(\gamma f)}{\ell_d(\gamma f)}\geq \log \frac{\sigma_1(\gamma))}{\sigma_d(\gamma)}-M$. The translation length of an isometry $g \in \textup{Aut}(\Omega)$ is exactly $\frac{1}{2}\log\frac{\ell_1(g)}{\ell_d(g)}$, see \cite[Prop. 2.1]{CLT}. In particular, if $\gamma \in \Gamma$ and $f \in F$ are as previously, we have that 
\begin{align}\label{control-eq1}\nonumber 2d_{\Omega}(\gamma x_0, x_0) & \geqslant 2d_{\Omega}(\gamma fx_0,x_0)-2d_{\Omega}(fx_0,x_0) \\
\nonumber & \geqslant \log \frac{\ell_1(\gamma f)}{\ell_d(\gamma f)}-2d_{\Omega}(fx_0,x_0) \\
 &  \geqslant \log \frac{\sigma_1(\gamma)}{\sigma_d(\gamma)}-M-2d_{\Omega}(fx_0,x_0).
\end{align} Then, by Proposition \ref{control1} and (\ref{control-eq1}), we obtain $L>0$ such that $$ \ \Big |\log \frac{\sigma_1(\rho(\gamma))}{\sigma_d(\rho(\gamma))}-2d_{\Omega}(\gamma x_0,x_0) \Big| \leqslant L$$ for every $\gamma \in \Gamma$. The conclusion follows. \end{proof}

\begin{definitions} \textup{(}\cite{DGK0}\textup{)} Let $\Gamma$ be an infinite discrete subgroup of $\mathsf{PGL}_d(\mathbb{R})$ preserving a properly convex domain $\Omega$ of $\mathbb{P}(\mathbb{R}^d)$ and $\Lambda_{\Omega}(\Gamma)\subset \partial \Omega$ be the set of accumulation points of all $\Gamma$-orbits. The group $\Gamma$ acts convex cocompactly on $\Omega$, if the convex hull of $\Lambda_{\Omega}(\Gamma)$ in $\Omega$ is non-empty and acted on cocompactly by $\Gamma$. The group $\Gamma$ is called strongly convex cocompact in $\mathbb{P}(\mathbb{R}^d)$ if $\Gamma$ acts convex cocompactly on some strictly convex domain $\Omega$ with $C^1$-boundary. \end{definitions}

The following lemma follows immediately from \cite[Thm. 1.4]{DGK0} and \cite[Thm. 1.27]{Zimmer} and is used to pass from a $P_1$-Anosov representation to a convex cocompact action in some  projective space.

\begin{lemma}  \label{replace} Let $V_{d}$ be the vector space of $d\times d$-symmetric matrices and $\mathsf{S}_{d}:\mathsf{GL}_d(\mathbb{R}) \rightarrow \mathsf{GL}(V_{d})$ be the representation defined as follows $\mathsf{S}_d (g)X=gXg^t$ for $g\in \mathsf{GL}_d(\mathbb{R})$ and $X \in V_d$. For every $P_1$-Anosov representation $\rho:\Gamma \rightarrow \mathsf{GL}_d(\mathbb{R})$, the representation $\mathsf{S}_{d}\circ \rho$ is $P_1$-Anosov and $\mathsf{S}_d (\rho(\Gamma))$ is a strongly convex cocompact subgroup of $\mathsf{GL}(V_d)$. \end{lemma}

\par Given two representations \hbox{$\rho_1:\Gamma \rightarrow \mathsf{GL}_m(\mathbb{R})$} and $\rho_2:\Gamma \rightarrow \mathsf{GL}_d(\mathbb{R})$, we say that {\em $\rho_1$ uniformly dominates $\rho_2$} if there is $0<\epsilon<1$ with the property that for every $\gamma \in \Gamma$,
 $$(1-\epsilon)\log \ell_1 (\rho_1(\gamma))  \geqslant \log \ell_1(\rho_2(\gamma)).$$

We will also need the following lemma for the proof of Proposition \ref{Gromovproduct1}, which allows us to control the Cartan projection of an Anosov representation $\rho$ in terms of the Cartan projection of a semisimplification $\rho^{ss}$ (of $\rho$). We expect that this fact follows by the techniques of Guichard--Wienhard in \cite[\S 5]{GW} showing that Anosov representations have strong proximality properties.

\begin{lemma} \label{weights} Let $\Gamma$ be a word hyperbolic group, $G$ a real semisimple Lie group and $\theta \subset \Delta$ a subset of simple restricted roots of $G$. Suppose $\psi:\Gamma \rightarrow G$ is a $P_{\theta}$-Anosov representation with semisimplification $\psi^{ss}:\Gamma \rightarrow G$. There is a constant $C_{\psi}>0$, depending only on $\psi$, such that for every $\gamma \in \Gamma$, $$ \max_{\alpha \in \theta}\big| \omega_{\alpha} \big( \mu(\psi(\gamma))-\mu(\psi^{ss}(\gamma)) \big) \big| \leqslant C_{\psi}.$$ \end{lemma}
\begin{proof} 
Let us first observe that for any linear representation $\phi:\Gamma \rightarrow \mathsf{GL}_m(\mathbb{R})$ and any semisimplification $\phi^{ss}$ of $\phi$, by Theorem \ref{finitesubset}, there exists a constant $M>0$, \hbox{depending only on $\phi$, such that} $$\log \sigma_1(\phi(\gamma))\geq \log \sigma_1(\phi^{ss}(\gamma))-M$$ for every $\gamma \in \Gamma$. Moreover, by \cite[Thm. 7.2]{Tits}, for every $\alpha \in \theta$, there exists $N_{\alpha}>0$ such that $N_{\alpha}\omega_{\alpha}$ is the highest weight of an irreducible proximal representation $\tau_{\alpha}:G \rightarrow \mathsf{GL}_m(\mathbb{R})$. In particular, by the definition of the highest weight, there exists $M'>0$, depending only on $\tau_{\alpha}$, such that $$\forall h \in G, \  \big|\log \sigma_1(\tau_{\alpha}(h))-N_{\alpha} \omega_{\alpha}(\mu(h)) \big|\leq M'.$$ 

Given the representation $\psi:\Gamma \rightarrow G$, by the previous two facts, there is $C_{1}>0$, depending only on $\psi$ and $G$, such that \hbox{for every $\gamma \in \Gamma$,} \begin{align}\label{upper-semisimple}\omega_{\alpha}\big(\mu(\psi(\gamma))-\mu(\psi^{ss}(\gamma))\big) \geq -C_{1}.\end{align} Now we prove that there is $D>0$ such that for every $\gamma \in \Gamma$,  \begin{align*} \omega_{\alpha}\big(\mu(\psi(\gamma))- \mu(\psi^{ss}(\gamma))\big) \leqslant D.\end{align*} 

By Proposition \ref{higherdimension}, we may compose $\psi$ with an irreducible representation \hbox{$\tau_{\theta}:G \rightarrow \mathsf{GL}_n(\mathbb{R})$} such that $\rho:=\tau_{\theta} \circ \psi$ and its semisimplification $\rho^{ss}=\tau_{\theta} \circ \psi^{ss}$ are $P_1$-Anosov. Clearly if $\rho$ is semisimple then the bound follows by Theorem \ref{finitesubset}. Hence, we continue by assuming that $\rho$ is not semisimple (hence non irreducible) and preserves some proper subspace of $\mathbb{R}^n$. Up to composing $\rho$ with the representation $\mathsf{S}_{n}$ from Lemma \ref{replace}, we may further assume that $\rho(\Gamma)$ and the dual $\rho^{\ast}(\Gamma)$ preserve (possibly different) properly convex domains in $\mathbb{P}(\mathbb{R}^n)$. Moreover, up to conjugating $\rho$, and possibly considering the dual representation of this conjugate, we may assume $\rho(\Gamma)$ preserves a properly convex domain $\Omega_0\subset \mathbb{P}(\mathbb{R}^n)$ and there is a decomposition $\mathbb{R}^n=V_1 \oplus \cdots \oplus V_{\ell}$ \hbox{such that} $$\rho=\begin{pmatrix}
\rho_1 &  \ast& \ast \\ 
0  & \ddots  & \ast \\ 
 0 & 0  & \rho_{\ell}
\end{pmatrix}, \   \rho^{ss}=\begin{pmatrix}
\rho_1 &  0& 0 \\ 
0  & \ddots  & 0  \\ 
 0 & 0  & \rho_{\ell}
\end{pmatrix}$$ where $\{\rho_i:\Gamma \rightarrow \mathsf{GL}(V_i)\}_{i=1}^{\ell}$ are irreducible representations and $\rho_1$ is the restriction of $\rho^{ss}$ on the image of its limit map $\xi_{\rho^{ss}}^{+}:\partial_{\infty}\Gamma \rightarrow \mathbb{P}(\mathbb{R}^n)$. By the definition of $V_1$, since the attracting fixed point of $\rho^{ss}(\gamma)$ lies in $V_1$ for every infinite order element $\gamma \in \Gamma$, the restriction of $\rho_1$ of $\rho^{ss}$ on $V_1$ uniformly dominates $\rho_i$ for every $2 \leqslant i \leqslant \ell.$ \par By using induction, it is enough to consider the case when $\ell=2$ and $$\rho(\gamma)=\begin{pmatrix}
\rho_1(\gamma) & u(\gamma)  \\ 
0 & \rho_2(\gamma) \\ 
\end{pmatrix},\ \gamma \in \Gamma$$ where $u:\Gamma \rightarrow \textup{Hom}(V_2,V_1)$ is an appropriate  matrix valued function. The group $\rho_1(\Gamma)$ preserves the properly convex domain $\Omega_{0}\cap \mathbb{P}(V_1)$ of $\mathbb{P}(V_1)$. By \cite{DGK0, Zimmer}, there exists a closed $\rho_1(\Gamma)$-invariant properly convex domain $\Omega_1 \subset \mathbb{P}(V_1)$ and a $\rho_1(\Gamma)$-invariant closed convex subset $\mathcal{C} \subset \Omega_1$ such that $\rho_1(\Gamma)\backslash \mathcal{C}$ is compact. We fix a basepoint $x_0 \in \mathcal{C}$ such that every point of $\mathcal{C}$ is within $d_{\Omega_1}$-distance $M>0$ from the orbit $\rho_1(\Gamma)\cdot x_0$. Let $g \in \Gamma$ and consider $x_0,x_1,\ldots,x_k \in [x_0,gx_0]$ with $\frac{1}{2}\leqslant d_{\Omega}(x_i,x_{i+1}) \leqslant 1$. For every $0 \leq i \leq k$, choose $g_i \in \Gamma$ such that $d_{\Omega}(\rho_1(g_i)x_0,x_0) \leqslant M$, where $g_0=e$ and $g_k=g$. Now we define $\{h_i\}_{i=0}^{k+1}$ as follows: $h_0=e$, $h_{i}=g_{i-1}^{-1}g_{i}$, $1 \leqslant i \leqslant k$ and $h_{k+1}=e$. Observe that $g=h_1\cdot\cdot\cdot h_k$ and a straightforward computation shows that \begin{align}\label{eq1-semisimple}u(g)=u\big(h_1 \cdot \cdot \cdot h_{k}\big)=\sum_{i=0}^{k-1}  \rho_1\big(h_{0} \cdot \cdot \cdot h_{i}) u(h_{i+1}\big) \rho_2\big(h_{i+2}\cdot \cdot \cdot h_{k+1}\big).\end{align} By using Theorem \ref{finitesubset} and the fact that $\rho_1$ is semisimple, $P_1$-Anosov and uniformly dominates $\rho_2$, we can find constants $A,E,a,b, \varepsilon>0$ such that for every $\gamma \in \Gamma$: \begin{align}\label{eq2-semisimple}b\frac{\sigma_1(\rho_1(\gamma))^{1-\varepsilon}}{\sigma_1(\rho_2(\gamma))}\geq 1, \ \sigma_1(\rho_1(\gamma)) \geqslant Ae^{ad_{\Omega_1}(\rho_{1}(\gamma)x_0,x_0)}, \ \Bigg|\log\frac{\sigma_1(\rho_1(\gamma))}{\sigma_{d_1}(\rho_1(\gamma))} -2d_{\Omega_1}(\gamma x_0,x_0) \Bigg| \leqslant E.\end{align} To simplify notation, for $i=1,\ldots,k$, set $w_i:=h_1 \cdot \cdot \cdot h_{i}$ and the triangle inequality shows \begin{align}\label{eq3-semisimple}\Big|d_{\Omega_1}\big(\rho_1(w_i)x_0,x_0\big)-d_{\Omega_1}\big(x_i,x_0\big)\Big| \leqslant M, \ \Big|d_{\Omega_1}\big(\rho_1(w_i)x_0,gx_0\big)-d_{\Omega_1}\big(x_i,\rho_1(g)x_0\big)\Big| \leqslant M.\end{align} Note that there exists $R>0$, independent of $g\in \Gamma$, such that $h_{i} \in \Gamma$ lie in a metric ball of radius $R>0$ of $\Gamma$. Therefore, by (\ref{eq1-semisimple}), (\ref{eq2-semisimple}) and (\ref{eq3-semisimple}), there exists $C_{R}>0$ independent of $g \in \Gamma$ such that: 
\begin{align*}\big|\big|u(g)\big|\big| & \leqslant C_{R}\sum_{i=1}^{k-1} \sigma_1\big(\rho_1(h_{1}\cdot \cdot \cdot h_{i})\big) \cdot \sigma_1\big(\rho_2(h_{i+1}\cdot \cdot \cdot h_{k})\big)\leqslant b C_{R}\sum_{i=0}^{k-1} \frac{\sigma_1(\rho_1(w_{i}^{-1}g)) \sigma_1(\rho_1(w_i))}{\sigma_1(\rho_1(w_{i}^{-1}g))^{\varepsilon}} \\   & =bC_{R} \sum_{i=0}^{k-1}\Bigg( \frac{1}{\sigma_1(\rho_1(g^{-1}w_i)) \sigma_1(\rho_1(w_{i}^{-1}))} \cdot \frac{\sigma_1 (\rho_1(w_{i}^{-1}g) )}{\sigma_{d_1}(\rho_1(w_{i}^{-1}g) )}\cdot  \frac{\sigma_1(\rho_1(w_{i})) }{\sigma_{d_1}(\rho_1(w_{i})) }\cdot \frac{1}{\sigma_1(\rho_1(w_{i}^{-1}g))^{\varepsilon}}\Bigg)\\ & \leqslant bC_R \sum_{i=0}^{k-1} \Bigg(\frac{e^{2E}}{\sigma_{1}(\rho_1(g^{-1}))}\cdot e^{2d_{\Omega_1}(\rho_1(w_i^{-1}g) x_0,x_0)} \cdot e^{2d_{\Omega_1}(\rho_1(w_i)x_0, x_0)}\cdot A^{-\varepsilon}e^{-a\varepsilon |w_{i}^{-1} g|_{\Gamma}}\Bigg)\\ & =\frac{bC_{R}e^{2E}}{\sigma_{1}(\rho_{1}(g^{-1}))A^{\varepsilon}} \sum_{i=0}^{k-1} \Bigg(e^{2d_{\Omega}(\rho(w_{i})x_0,\rho(g)x_0)}\cdot e^{2d_{\Omega}(\rho(x_{i})x_0,x_0)}\cdot e^{-a\varepsilon d_{\Omega}(\rho_1(w_{i}^{-1}g)x_0,x_0)}\Bigg)\\ & \leqslant \frac{bC_R e^{2E+2M+2Ma \varepsilon}}{A^{\varepsilon}\sigma_{1}(\rho_1(g^{-1}))} e^{2d_{\Omega_1}(\rho_1(g)x_0,x_0)} \Bigg( \sum_{i=0}^{k-1} e^{-\varepsilon a (k-i)}\Bigg)\\ &\leqslant \frac{bC_R e^{2E+2M+2Ma\varepsilon}}{A^{\varepsilon} (1-e^{-a \varepsilon})}\sigma_1(\rho_1(g)).\end{align*} 
We conclude that there exists $L>0$, depending only on $\rho$, such that for every $g\in \Gamma$, \begin{align}\label{semisimple-tau}\sigma_1(\rho^{ss}(g))\leqslant \sigma_1( \rho(g)) \leqslant L \sigma_1(\rho^{ss}(g)).\end{align} 

Now recall that $\rho=\tau_{\theta}\circ \psi$ and $\rho^{ss}=\tau_{\theta}\circ \psi^{ss}$, where $\psi:\Gamma \rightarrow G$ is a $\theta$-Anosov representation, $\psi^{ss}$ is a semisimplification of $\psi$ and $\tau_{\theta}:G\rightarrow \mathsf{GL}_n(\mathbb{R})$ is a $\theta$-compatible representation. The highest weight of $\tau_{\theta}$ is of the form $\chi_{\tau_{\theta}}=\sum_{\alpha \in \theta}n_a \omega_{\alpha}$, where $n_{\alpha}>0$ and $\omega_{\alpha}$ is the fundamental weight with respect to $\alpha$ (e.g. see the discussion in \cite[Subsec. 3.2]{GGKW}). By the definition of $\chi_{\tau_{\theta}}$, we may choose $D_1>0$, depending only on $\tau_{\theta}$, with the property for every $h\in G$, $$ \Big|\log\sigma_1(\tau_{\theta}(h))- \chi_{\tau_{\theta}}(\mu(h)) \Big|\leq D_1.$$ Since $\rho^{ss}=\tau_{\theta}\circ \psi^{ss}$, by  (\ref{semisimple-tau}), there is $D>0$ such that for every $g\in \Gamma$, $$\Big|\chi_{\tau_{\theta}}\big( \mu(\psi(\gamma))-\mu(\psi^{ss}(\gamma))\big)\Big|\leq D.$$

By (\ref{upper-semisimple}) there is $C_1>0$, depending only on $\psi$, with $\omega_{\alpha}\big(\mu(\psi(\gamma))-\mu(\psi^{ss}(\gamma))\big) \geq -C_1$ for every $\gamma \in \Gamma$, thus, for every $\alpha \in \theta$ $$ \chi_{\tau_{\theta}}\big(\mu(\psi(\gamma))-\mu(\psi^{ss}(\gamma))\big) \geq n_{\alpha} \omega_{\alpha}\big(\mu(\psi(\gamma))-\mu(\psi^{ss}(\gamma))\big)-C_1\sum_{\beta \in \theta \smallsetminus \{\alpha\}} n_{\beta}.$$ In particular, for every $\gamma \in \Gamma$ we conclude that $$\omega_{\alpha}\big(\mu(\rho(\gamma))-\mu(\rho^{ss}(\gamma))\big) \leq \frac{1}{n_{\alpha}}\Bigg(D+C\sum_{\beta \in \theta \smallsetminus \{\alpha\}}n_{\beta}\Bigg).$$  This concludes the proof of the lemma. \end{proof} 

\section{The contraction property} \label{contraction}
Let $\Gamma$ be a word hyperbolic group. Fix $\big(\hat{\Gamma}, \varphi_{t} \big)$ a flow space on which $\Gamma$ acts properly discontinuously and cocompactly. Fix also $\mathcal{F} \subset \hat{\Gamma}$ a compact subset of $\hat{\Gamma}$ whose $\Gamma$-translates cover $\hat{\Gamma}$.  Let \hbox{$\rho:\Gamma \rightarrow \mathsf{GL}_d(\mathbb{R})$} be a representation admitting a pair of transverse, $\rho$-equivariant maps \hbox{$\xi^{+}:\partial_{\infty} \Gamma \rightarrow \mathbb{P}(\mathbb{R}^d)$} and \hbox{$\xi^{-}: \partial_{\infty}\Gamma \rightarrow \mathsf{Gr}_{d-1}(\mathbb{R}^d)$} defining the flat section $\sigma: \Gamma \backslash \hat{\Gamma} \rightarrow \mathcal{X}_{\rho}$ of the fiber bundle $\pi:\mathcal{X}_{\rho} \rightarrow \Gamma \backslash \hat{\Gamma}$. We fix an equivariant family of norms $\big(||\cdot||_{x} \big)_{x \in \Gamma \backslash \hat{\Gamma}}$ on the fibers of the bundle \hbox{$\pi_{\pm}:\mathcal{E}_{\rho}^{\pm}\rightarrow \Gamma \backslash \hat{\Gamma}$}. Recall also the maps $\tau^{\pm}:\hat{\Gamma}\rightarrow \partial_{\infty}\Gamma$ defined in Subsection \ref{Flows}. For a given point $\hat{m} \in \hat{\Gamma}$, choose $h \in G$ so that $\xi^{+}(\tau^{+}(\hat{m}))=hP_1^{+}$  and $\xi^{-}(\tau^{-}(\hat{m}))=hP_1^{-}$ and denote by $L_h:G \rightarrow G$ the left translation by $h \in G$. Then consider the tangent spaces \begin{align*}\mathsf{T}_{hP_1^{+}}\mathbb{P}(\mathbb{R}^d)&=\Big\{ dL_{h}d\pi^{+}\left (X \right ):X \in \bigoplus_{i=2}^{d}\mathbb{R}E_{i1} \Big \}\\ \mathsf{T}_{hP_1^{-}}\mathsf{Gr}_{d-1}(\mathbb{R}^d)&=\Big \{dL_{h}d\pi^{-}\left (X\right ): X \in \bigoplus_{i=2}^{d}\mathbb{R}E_{1i} \Big \}\end{align*} where $E_{ij}$ is the $d\times d$ matrix whose $(i,j)$ entry is $1$ and all the others zero. For $u \in \{0\}\times \mathbb{R}^{d-1}$ we denote by $X_{u}^{+}\in \mathsf{T}_{hP_1^{+}}\mathbb{P}(\mathbb{R}^d) $ and $X_{u}^{-}\in  \mathsf{T}_{hP_1^{-}}\mathsf{Gr}_{d-1}(\mathbb{R}^d) $ the tangent vectors \begin{align*} X_{u}^{+} &=\Bigg[ \hat{m}, \big( \xi^{+}(\tau^{+}(\hat{m})), \xi^{-}(\tau^{-}(\hat{m})) \big), dL_{h}d \pi^{+}\left ( \begin{pmatrix} 
0 & 0\\ 
u & 0
\end{pmatrix} \right )\Bigg]_{\Gamma}\\ X_{u}^{-} &=\Bigg[\hat{m}, \big(\xi^{+}(\tau^{+}(\hat{m})), \xi^{-}(\tau^{-}(\hat{m})) \big), dL_{h}d \pi^{-}\left (\begin{pmatrix}
0 & u\\ 
0 & 0
\end{pmatrix} \right ) \Bigg]_{\Gamma} \end{align*} in the fibers of the bundles $\sigma_{\ast}\mathcal{E}^{\pm} \rightarrow \Gamma \backslash \hat{\Gamma}$ over $x=[\hat{m}]_{\Gamma}$ and $(\pi^{+},\pi^-):\mathsf{SL}_d(\mathbb{R})\rightarrow \mathbb{P}(\mathbb{R}^d)\times \mathsf{Gr}_{d-1}(\mathbb{R}^d)$ are the natural projections.

The following lemma shows that when the geodesic flow on $\sigma_{\ast}\mathcal{E}^{-}$ is weakly contracting then the geodesic flow on $\sigma_{\ast}\mathcal{E}^{+}$ is weakly dilating. Recall that $\langle \cdot,\cdot\rangle$ is the standard Euclidean inner product on $\mathbb{R}^d$.

\begin{lemma} \label{contractiondilation} Let $\rho:\Gamma \rightarrow \mathsf{GL}_d(\mathbb{R})$ be a representation. Suppose there exists a pair of continuous, \hbox{$\rho$-equivariant} transvserse maps $\xi^{+}:\partial_{\infty}\Gamma \rightarrow \mathbb{P}(\mathbb{R}^d)$ and $\xi^{-}:\partial_{\infty}\Gamma \rightarrow \mathsf{Gr}_{d-1}(\mathbb{R}^d)$. Then for any $x=[\hat{m}]_{\Gamma} \in \hat{\Gamma}$ and $u \in \{0\} \times \mathbb{R}^{d-1}$ we have: $$\varliminf_{t \rightarrow \infty}\big|\big|\varphi_{t}(X_{u}^{+})\big|\big|_{\varphi_t(x)}\cdot \big|\big|\varphi_{t} \big(X_{u}^{-} \big)\big|\big|_{\varphi_{t}(x)} >0.$$   \end{lemma}
\begin{proof}  For two sequences of positive real numbers $(a_n)_{n \in \mathbb{N}}$, $(b_n)_{n \in \mathbb{N}}$ we write $a_{n} \asymp b_{n}$ if there exists $R>0$ such that $R^{-1}a_{n} \leqslant b_{n} \leqslant Ra_{n} $ for every $n \in \mathbb{N}$. We may assume that $\rho(\Gamma)$ is contained in $\mathsf{SL}^{\pm}_{d}(\mathbb{R})$, otherwise we may replace $\rho$ with $\hat{\rho}(\gamma)=|\textup{det}(\rho(\gamma))|^{-1/d}\rho(\gamma)$, $\gamma \in \Gamma$ since $\xi^{\pm}$ are also $\hat{\rho}$-equivariant. \par Let $(t_n)_{n \in \mathbb{R}}$ be an increasing unbounded sequence. For each $n \in \mathbb{N}$, we may choose $\gamma_n \in \Gamma$ such that $\gamma_{n} \varphi_{t_n}(\hat{m})$ lies in the compact fundamental domain $\mathcal{F}$. There also exist $k_{1n},k_{2n}\in K$ so that $${\rho(\gamma_n)h=k_{1n}\begin{pmatrix}
\lambda_{n} & \ast \\ 
 0& A_{n} 
\end{pmatrix}=k_{2n}\begin{pmatrix}
s_{n} & 0\\ 
 \ast & B_{n} 
\end{pmatrix}}.$$ Notice that for $g \in P_1^{\pm}$ we have $dL_{g}\circ d\pi^{\pm}=d\pi^{\pm} \circ \textup{Ad}(g)$ and an elementary calculation \hbox{shows that} \begin{align*} dL_{\rho(\gamma_n)h} d\pi^{+}\left ( \begin{pmatrix}
0 & 0\\ 
u& 0
\end{pmatrix} \right )&=dL_{k_{1n}} \left (d\pi^{+}\left ( \textup{Ad}\left (\begin{pmatrix}
\lambda_n & \ast \\ 
0 & A_{n}
\end{pmatrix}  \right ) \begin{pmatrix}
0 & 0\\ 
u & 0
\end{pmatrix} \right) \right )\\ &=dL_{k_{1n}}\left (d\pi^{+}\left(\begin{pmatrix}
0 & 0 \\ 
\frac{1}{\lambda_n}A_{n}u& 0
\end{pmatrix} \right)  \right ). \end{align*} Similarly, we check that $${dL_{\rho(\gamma_n)h}d\pi^{-}\left ( \begin{pmatrix}
0 & u\\ 
0& 0
\end{pmatrix} \right )=dL_{k_{2n}}\left (d\pi^{-}\left(\begin{pmatrix}
0 & s_{n} B_{n}^{-t}u \\ 
0& 0
\end{pmatrix} \right)  \right )}.$$

\noindent By the continuity of the family of norms $\big(||\cdot||_{x}\big)_{x \in \Gamma \backslash\hat{\Gamma}}$ and since $k_{1n},k_{2n}\in K$ lie in a compact group, we deduce that $$\left \| \varphi_{t_n} \big(X_{u}^{+} \big) \right \|_{\varphi_{t_n}(x)} \asymp \frac{\left \| A_{n}u \right \|}{|\lambda_{n}|} \ \  \mathrm{and} \ \ \left \| \varphi_{t_n} \big(X_{u}^{-}  \big)\right \|_{\varphi_{t_n}(x)} \asymp |s_{n}|\cdot \left \| B_{n}^{-t}u \right \|,$$ where $||\cdot||$ denotes the usual Euclidean norm on $\mathbb{R}^{d-1}$. Up to passing to a subsequence, we may assume that  $\lim_{n}\gamma_{n}\varphi_{t_n}(\hat{m})=\hat{m}'$. Since the maps $\tau^{\pm}$ are continuous, we conclude, up to passing to a subsequence, that $(\gamma_n \tau^{+}(\hat{m}))_{n \in \mathbb{N}}$ and $(\gamma_n \tau^{-}(\hat{m}))_{n \in \mathbb{N}}$ converge to $\tau^{+}(\hat{m}')\in \partial_{\infty}\Gamma$ and $\tau^{-}(\hat{m}')\in \partial_{\infty}\Gamma$ respectively. We have $\xi^{+}(\tau^{+}(\gamma_n \hat{m}))=k_{1n}P_{1}^{+}$ and $\xi^{-}(\tau^{+}(\gamma_n \hat{m}))=k_{2n}P_{1}^{-}$ and by transversality, there exist $p_{n} \in P_{1}^{+}, q_{n} \in P_{1}^{-}$ and $g \in G$ such that $\lim_{n}k_{1n}p_{n}=\lim_{n} k_{2n}q_{n}=g$. Then there exist $z_{n},z_{n}' \in \mathbb{R}$ so that $\lim_{n} z_{n}k_{1n}e_1=ge_{1}$ and $\lim_{n} z_{n}'k_{2n}e_1= g^{-t}e_1$ and we observe that $|z_n|,|z_n'|$ converge respectively to $||ge_1||$ and $||g^{-t}e_1||$. Notice that $\lim_{n}z_nz_n'\langle k_{1n}e_1,k_{2n}e_1 \rangle=|\langle g^{-t}e_1,ge_1 \rangle|=1$ and so $\lim_{n} \langle k_{1n}e_1,k_{2n}e_1 \rangle=\frac{1}{||ge_1||\cdot ||g^{-t}e_1||}$. Recall that $k_{2n}^{-1}k_{1n}\begin{pmatrix}[0.8]
\lambda_{n} & \ast \\ 
 0& A_{n} 
\end{pmatrix}=\begin{pmatrix}[0.8]
s_{n} & 0\\ 
 \ast & B_{n} 
\end{pmatrix} $, hence, by looking at the $(1,1)$ entry of both sides, we obtain ${\left |\frac{s_{n}}{\lambda_{n}} \right |=\left | \left \langle k_{1n}e_1, k_{2n}e_1\right \rangle \right |}$ and so $L:=\inf_{n \in \mathbb{N}}\left |\frac{s_{n}}{\lambda_{n}} \right |>0$. Furthermore, we observe that $\begin{pmatrix}[0.6]
\lambda_n & 0 \\ 
\ast & A_{n}^{t}
\end{pmatrix}k_{1n}^{t}=\begin{pmatrix}[0.7]
s_{n} & \ast \\ 
0 & B_{n}^{t}
\end{pmatrix}k_{2n}^{t}$ and hence ${\begin{pmatrix}[0.7]
\ast & \ast \\ 
\ast & B_{n}^{-t}A_{n}^{t}
\end{pmatrix}=k_{2n}^{-1}k_{1n}}$ since $k_{1n}k_{1n}^t=k_{2n}k_{2n}^t=I_{d}$. Up to passing to a subsequence, we may assume that \hbox{${\lim_{n} B_{n}^{-t}A_{n}^{t}}=Q$} exists. Since \hbox{$|\lambda_n|\cdot |\det(A_n)|=|s_n|\cdot |\det(B_n)|$} we have ${\left|\det\left (B_{n}^{-t}A_{n}^{t} \right )\right |=\left | \frac{s_{n}}{\lambda_{n}} \right |}\geqslant L>0$. In particular, $Q$ is invertible and there is $M>0$ with \hbox{$\frac{1}{M} \leqslant \max\big(||B_{n}^{-t}A_{n}^{t}||, ||A_{n}^{-t}B_{n}^{t}||\big)  \leqslant M$} for every $n \in \mathbb{N}$. Therefore, for every $n\in \mathbb{N}$ we have $${\frac{\left \| A_{n}u \right \|}{|\lambda_{n}|}\geqslant \frac{\left \|u  \right \|^2}{|\lambda_{n}| \left \| A_{n}^{-t}u \right \|} = \frac{\left \|u  \right \|^2}{|\lambda_{n}|\left \| A_{n}^{-t}B_{n}^{t}(B_{n}^{-t}u) \right \|} \geqslant \frac{\left \|u  \right \|^2}{|\lambda_{n}| \left| \left| A_{n}^{-t}B_{n}^{t} \right| \right|\cdot \left \|B_{n}^{-t}u \right \|}\geqslant \frac{L\left \|u  \right \|^2}{M |s_{n}|\left \| B_{n}^{-t}u \right \|}},$$ since $||A_{n}u||\cdot ||A_{n}^{-t}u|| \geqslant ||u||^2$. Finally, $${\varliminf_{n \rightarrow \infty} \left \|  \varphi_{t_n}\big(X_{u}^{+} \big) \right \|_{\varphi_{t_n}(x)}\cdot \left \| \varphi_{t_n} \big(X_{u}^{-} \big) \right \|_{\varphi_{t_n}(x)} >0}$$ and since the sequence $(t_n)_{n \in \mathbb{N}}$ was arbitrary the conclusion follows.  \end{proof}

\begin{proposition}  \label{contractiondilation2} Let $\rho:\Gamma \rightarrow \mathsf{GL}_d(\mathbb{R})$ be a representation which admits a pair of continuous $\rho$-equivariant transverse maps $\xi^{+}:\partial_{\infty}\Gamma \rightarrow \mathbb{P}(\mathbb{R}^d)$ and $\xi^{-}:\partial_{\infty}\Gamma \rightarrow \mathsf{Gr}_{d-1}(\mathbb{R}^d)$. Fix $x=[\hat{m}]_{\Gamma}$, \hbox{$u \in \{0\} \times \mathbb{R}^{d-1}$} and suppose $\xi^{+}(\tau^{+}(\hat{m}))=hP_1^{+}$ and $\xi^{-}(\tau^{-}(\hat{m}))=hP_1^{-}$. Let $(\gamma_n)_{n \in \mathbb{N}}$ be a sequence of elements of $\Gamma$ such that $\big(\gamma_n \varphi_{t_n}(\hat{m})\big)_{n \in \mathbb{N}}$ lies in a compact subset of $\hat{\Gamma}$.
\medskip

\begin{enumerate}[label=(\roman*)]

\item $\underset{n \rightarrow \infty}{\lim}||\varphi_{t_n} \big(X_{u}^{+}\big)||_{\varphi_{t_n}(x)} =+\infty$ if and only if $${\lim_{n \rightarrow \infty}\frac{\left \| \rho(\gamma_n)hu \right \|}{\left \| \rho(\gamma_n)he_1 \right \|}=+\infty}.$$ 
\item $\underset{n \rightarrow \infty}{\lim}||\varphi_{t_n} \big(X_{u}^{-}\big)||_{\varphi_{t_n}(x)} =0$ if and only if $${\lim_{n \rightarrow \infty}\frac{\left \| \rho^{\ast}(\gamma_n)h^{-t}u \right \|}{\left \| \rho^{\ast}(\gamma_n)h^{-t}e_1 \right \|}=0}.$$ \end{enumerate} \end{proposition}

\begin{proof} Suppose that ${\rho(\gamma_n)h=k_{1n}\begin{pmatrix}[0.8]
\lambda_{n} & \ast \\ 
 0& A_{n} 
\end{pmatrix}=k_{2n}\begin{pmatrix}[0.8]
s_{n} & 0\\ 
 \ast & B_{n} 
\end{pmatrix}}$. Let $(\gamma_{r_n})_{n \in \mathbb{N}}$ be a subsequence of $(\gamma_n)_{n \in \mathbb{N}}$. A straightforward calculation shows that $${\frac{\left \| A_{r_n}u \right \|}{|\lambda_{r_n}|}=\frac{\left \| \rho(\gamma_{r_n})hu \right \|}{\left \| \rho(\gamma_{r_n})he_1 \right \|} \sin \measuredangle  \big (\rho(\gamma_{r_n})he_1, \rho(\gamma_{r_n})hu \big)}$$ where $\xi^{+}(x)=hP_{1}^{+}$ and $hu \in \xi^{-}(y)$. Up to passing to subsequence, we may assume that $\lim_{n} \gamma_{r_n}\varphi_{t_{r_n}}(\hat{m})$ exists and so $\lim_{n}\gamma_{r_n} \tau^{+}(\hat{m}) \neq \lim_{n}\gamma_{r_n} \tau^{-}(\hat{m})$. The maps $\xi^{+}$ and $\xi^{-}$ are transverse, hence there exists $g \in G$ and $p_n \in P_1^{+}$, $q_n \in P_1^{-}$ such that $\lim_{n} \rho(\gamma_{r_n})hp_n=\lim_{n}\rho(\gamma_{r_n})hq_n=g$. Let $v_{\infty} \in e_1^{\perp}$ be a limit point of the sequence $\Big(\frac{q_{n}^{-1}u}{||q_{n}^{-1}u||}\Big)_{n \in \mathbb{N}}$. Then, \hbox{$\lim_{n} \frac{1}{||q_{n}^{-1}u||}\rho(\gamma_{r_n})hu=gv_{\infty}$} and hence $ \lim_{n} \sin \measuredangle  \big (\rho(\gamma_{r_n})he_1, \rho(\gamma_{r_n})hu\big)=\sin \measuredangle \big(gv_{\infty}, ge_1 \big)>0$. Since we started with an arbitrary subsequence, there exists $\varepsilon>0$ with \hbox{$\big|\sin \measuredangle  \big (\rho(\gamma_{r_n})he_1, \rho(\gamma_{r_n})hu \big)\big|\geqslant \varepsilon$} for every $n \in \mathbb{N}$. Therefore, $\frac{\left \| A_{n}u \right \|}{|\lambda_{n}|} \asymp \frac{\left \| \rho(\gamma_{n})hu \right \|}{\left \| \rho(\gamma_{n})he_1 \right \|}$. By Proposition \ref{contractiondilation} we have that $$\big|\big|\varphi_{t_{n}} (X_{u}^{+})\big|\big|_{\varphi_{t_n}(x)}\asymp \frac{\left \| A_{n}u \right \|}{|\lambda_{n}|}, \ \ n \rightarrow \infty$$ and hence part \textup{(i)} follows. The argument for part \textup(ii) is similar. \end{proof}

\section{The Cartan property and the uniform gap summation property} \label{scartan} 

Let $G$ be a real semisimple Lie group of non-compact type, $K$ a maximal compact subgroup of $G$ and $\mu:G \rightarrow \overline{\mathfrak{a}}^{+}$ the associated Cartan projection. The \emph{restricted Weyl group} of $\mathfrak{a}$ in $\mathfrak{g}$ is the group $W=N_{K}(\mathfrak{a})/Z_{K}(\mathfrak{a})$, where $N_{K}(\mathfrak{a})$ (resp. $Z_{K}(\mathfrak{a})$) is the normalizer (resp. centralizer) of $\mathfrak{a}$ in $K$. The group $W$ is finite, acts simply transitively on the set of Weyl chambers of $\mathfrak{a}$ and contains a unique order two element $w_0Z_{K}(\mathfrak{a}) \in W$ such that $\textup{Ad}(w_0)\overline{\mathfrak{a}}^{+}=-\overline{\mathfrak{a}}^{+}$. The element $w_0 \in K$ defines the opposition involution $^{\star}:\Delta \rightarrow \Delta$ on the set of simple restricted roots $\Delta$ as follows: if $\alpha \in \Delta$ then $\alpha^{\star} \in \Delta$ is the unique root with $\alpha^{\star}(H)=-\alpha(\textup{Ad}(w_0)H)$ for every $H \in \mathfrak{a}$. 

\par Let $\Gamma$ be an infinite, finitely generated group. A representation $\rho:\Gamma \rightarrow G$  is $P_{\theta}${\em-divergent} if $$ \lim_{|\gamma|_{\Gamma} \rightarrow \infty} \alpha \big(\mu(\rho(\gamma)) \big)=+\infty$$ for every $\alpha \in \theta$. Notice that the representation $\rho$ is $P_{\theta}$-divergent if and only if $\rho$ is $P_{\theta^{\star}}$-divergent. For an element $g=k_{g}\exp(\mu(g))k_{g}'$ written in the Cartan decomposition of $G$, define $$\Xi_{\theta}^{+}(g):=k_{g}P_{\theta}^{+} \ \ \textup{and} \ \ \Xi_{\theta}^{-}(g):=k_{g}w_0P_{\theta}^{-}.$$ For a $\rho$-equivariant map $\xi^{-}:\partial_{\infty}\Gamma \rightarrow G/P_{\theta}^{-}$, the map $\xi^{\ast}:\partial_{\infty}\Gamma \rightarrow G/P_{\theta^{\ast}}^{+}$ is defined as follows $$\xi^{\ast}(x)=k_{x}w_0P_{\theta^{\ast}}^{+},$$ where $\xi^{-}(x)=k_{x}P_{\theta}^{-}$ and $k_{x}\in K$.

\begin{definition} \label{defcartan} Let $G$ be a real semisimple Lie group, $\Gamma$ a word hyperbolic group and $\rho:\Gamma \rightarrow G$ a \hbox{$P_{\theta}$-divergent} representation.

\begin{enumerate}
\item Suppose that $\rho$ admits a continuous $\rho$-equivariant map $\xi^{+}:\partial_{\infty}\Gamma \rightarrow G/P_{\theta}^{+}$. The map $\xi^{+}$ satisfies the Cartan property if for any $x \in \partial_{\infty}\Gamma$ and every infinite sequence $(\gamma_n)_{n \in \mathbb{N}}$ of elements of $\Gamma$ with $\lim_{n}\gamma_n=x$, $$\xi^{+}(x)=\lim_{n \rightarrow \infty}\Xi^{+}_{\theta}\big(\rho(\gamma_{n})\big)$$ 

\item Suppose that $\rho$ admits a continuous $\rho$-equivariant map $\xi^{-}:\partial_{\infty}\Gamma \rightarrow G/P_{\theta}^{-}$. The map $\xi^{-}$ satisfies the Cartan property if the map $\xi^{\ast}: \partial_{\infty}\Gamma \rightarrow G/P_{\theta^{\star}}^{+}$ satisfies the Cartan property. In other words, for every $x \in \partial_{\infty}\Gamma$ and every infinite sequence $(\gamma_n)_{n \in \mathbb{N}}$ of elements of $\Gamma$ with $\lim_{n}\gamma_n=x$,  $$\xi^{-}(x)=\lim_{n \rightarrow \infty}\Xi^{-}_{\theta}\big(\rho(\gamma_{n})\big).$$ \end{enumerate}  \end{definition}

\begin{remarks}\normalfont{ \noindent \textup{(i)} Let $\rho:\Gamma \rightarrow G$ be a $P_{\theta}$-divergent representation. The Cartan property for a continuous \hbox{$\rho$-equivariant map} $\xi^{+}:\partial_{\infty}\Gamma \rightarrow G/P_{\theta}^{+}$ (resp. $\xi^{-}$) is independent of the choice of the Cartan decomposition of $G$. This follows by the fact that all Cartan subspaces of $G$ are conjugate under the adjoint action of $G$ and the second part of \cite[Cor. 5.9]{GGKW}.\\
\noindent \textup{(ii)} For a  $\theta$-divergent sequence $(g_n)_{n\in \mathbb{N}}\subset G$, written in the Cartan decomposition of $G$ as $g_n=k_n\exp(\mu(g_n))k_n'$, the condition of $\lim_n k_nP_{\theta}^{+}=\lim_n \Xi_{\theta}^{+}(g_n)$ to exist, implies that $(g_n)_{n\in \mathbb{N}}$ $\tau_{\textup{mod}}$-flag converges to $x$, in the definition of Kapovich--Leeb--Porti \cite[Subsec. 4.5]{KLP1}. \par Given a $P_{\theta}$-divergent representation $\rho$ and a $\rho$-equivariant continuous map $\xi:\partial_{\infty}\Gamma \rightarrow G/P_{\theta}^{+}$ with the Cartan property, the map $\Xi_{\theta}^{+}:\Gamma \rightarrow G/P_{\theta}^{+}$, $\gamma \mapsto \Xi_{\theta}^{+}(\rho(\gamma))$ extends continuously to a map $\Gamma \cup \partial_{\infty}\Gamma\rightarrow G/P_{\theta}^{+}$  restricting to $\xi$ on $\partial_{\infty}\Gamma$.

} \end{remarks}

The following fact is immediate from the definition of the Cartan property:

\begin{fact} \label{Cartan2} Suppose that $\rho, \Gamma, G$ and $\theta$ are defined as in Definition \ref{defcartan} and let $\xi:\partial _{\infty}\Gamma \rightarrow G/P_{\theta}^{+}$ be a continuous $\rho$-equivariant map. Suppose that $\tau_{\theta}:G \rightarrow \mathsf{GL}_d(\mathbb{R})$ is an irreducible $\theta$-compatible proximal representation as in Proposition \ref{higherdimension} so that $\tau_{\theta}(P_{\theta}^{+})$ stabilizes a line in $\mathbb{R}^d$ and induces a $\tau_{\theta}$-equivariant embedding \hbox{$\iota^{+}: G/P_{\theta}^{+} \xhookrightarrow{} \mathbb{P}(\mathbb{R}^d)$}. The map $\xi$ satisfies the Cartan property if and only if $\iota^{+} \circ \xi$ satisfies the Cartan property.\end{fact}

We need the following estimates which help us verify, in several cases, the Cartan property of limit maps into the homogeneous spaces $G/P_{\theta}^{+}$ and $G/P_{\theta}^{-}$. The second part of the following proposition has been established in \cite[Lem. A4]{BPS} and \cite[Lem. 5.8 (i)]{GGKW} but for completeness we give a short proof.

\begin{proposition} \label{inequalities} Let $G$ be a real semisimple Lie group, $\theta \subset \Delta$ a subset of simple restricted roots of $G$ and \hbox{$\tau_{\theta}:G \rightarrow \mathsf{GL}_d(\mathbb{R})$} an irreducible, $\theta$-proximal representation such that $\tau_{\theta}(P_{\theta}^{+})$ stabilizes the line $[e_1]$ in $\mathbb{P}(\mathbb{R}^d)$. Let $\chi_{\tau_{\theta}}\in \mathfrak{a}^{\ast}$ be the highest weight of $\tau_{\theta}$ and $g,r\in G$.
\begin{enumerate}[label=(\roman*)]
\item If $g$ is $P_\theta$-proximal in $G/P_{\theta}^{+}$ with attracting fixed point $x_{g}^{+}\in G/P_{\theta}^{+}$, then \begin{align*} d_{G/P_{\theta}^{+}} \big(x_{g}^{+},\Xi_{\theta}^{+}(g) \big)& \leqslant \exp \Big(-\min_{\alpha \in \theta}\alpha(\mu(g))+ \chi_{\tau_{\theta}}(\mu(g)-\lambda(g))\Big)\end{align*}

\item If $\min_{\alpha \in \theta}\langle \alpha, \mu(g)\rangle>0$ and $\min_{\alpha \in \theta}\langle \alpha, \mu(gr)\rangle>0$, then $$ d_{G/P_{\theta}^{+}} \big(\Xi_{\theta}^{+}(gr),\Xi_{\theta}^{+}(g) \big) \leqslant C_{d,r} \exp\Big(-\min_{\alpha \in \theta} \alpha( \mu(g)) \Big)$$ where ${C_{d,r}=\sigma_{1}(\tau_{\theta}(r))\sigma_{1}(\tau_{\theta}(r^{-1})) \sqrt{d-1}}$.\end{enumerate} \end{proposition}

\begin{proof} By the definition of the metric $d_{G/P_{\theta}^{+}}$ and Proposition \ref{higherdimension} we may assume that $G=\mathsf{SL}_d(\mathbb{R})$, \hbox{$\theta=\{\varepsilon_1-\varepsilon_2\}$} and $G/P_{\theta}^{+}=\mathbb{P}(\mathbb{R}^d)$.
\medskip

\noindent \textup{(i)} Since $g$ is proximal there exist $h \in \mathsf{GL}_d(\mathbb{R})$, $A_{g} \in \mathsf{GL}_{d-1}(\mathbb{R})$ and $k_{g},k_{g}' \in \mathsf{O}(d)$ such that $$g=h\begin{pmatrix}
\ell'_1(g) & 0   \\ 
0 & A_g\\ \end{pmatrix}h^{-1}=k_{g}\exp(\mu(g))k_{g}', \ \left|\ell_1'(g)\right|=\ell_1(g).$$ We can write $\Xi_{1}^{+}(g)=k_{g}P_1^{+}$ and $x_{g}^{+}=hP_{1}^{+}=w_1P_1^{+}$ for some $w_1 \in \mathsf{O}(d)$. Note that $$h \begin{pmatrix}
\ell_1'(g) & 0\\ 
 0 & A_g
\end{pmatrix}h^{-1}=w_1 \begin{pmatrix}
\ell_1'(g) & \ast\\ 
 0 & \ast
\end{pmatrix}w_{1}^{-1}$$ hence ${k_{g}^{-1}w_{1}\begin{pmatrix}[0.7]
\ell_1'(g) & \ast\\ 
0 & \ast
\end{pmatrix}=}$ ${\exp(\mu(g))k_{g}'w_{1}}$ and $\left|\left \langle k_{g}^{-1}w_1e_1,e_i \right \rangle \right|=\frac{\sigma_{i}(g)}{\ell_1(g)}\left|\left \langle k_{g}'w_1e_1,e_i \right \rangle \right|$ for $i>1$. Therefore, $$d_{\mathbb{P}}\big(x_{g}^{+},\Xi_{1}^{+}(g)\big)^2=\sum_{i=2}^{d}\left \langle k_{g}^{-1}w_1e_1,e_i \right \rangle^2=\sum_{i=2}^{d}\frac{\sigma_{i}(g)^2}{\ell_{1}(g)^2}\left \langle k_{g}'w_1e_1,e_i \right \rangle^2\leqslant \frac{\sigma_{2}(g)^2}{\ell_1(g)^2}.$$ Since $\min_{\alpha \in \theta} \alpha\big(\mu(\rho(\gamma))\big)=\log\frac{\sigma_1(\tau_{\theta}(g))}{\sigma_2(\tau_{\theta}(g))}$ and $\chi_{\tau_{\theta}}(\lambda(g))=\log\ell_1(\tau_{\theta}(g))$ for $g\in G$, part (i) follows.
\medskip

\noindent \textup{(ii)} We have $k_{gr}\exp(\mu(gr))k_{gr}'=k_{g}\exp(\mu(g))k_{g}'r$, $k_{gr},k_{gr}'\in K$, and in particular $${\left \langle k_{g}^{-1}k_{gr}e_1,e_i \right \rangle=\frac{\sigma_{i}(g)}{\sigma_{1}(gr)}\left \langle k_{g}'r(k_{gr}')^{-1}e_1,e_i \right \rangle}$$ for every $2 \leqslant i \leqslant d$. Note that since $\sigma_{1}(gr) \geqslant \frac{\sigma_1(g)}{\sigma_{1}(r^{-1})}$ and $\left|\left \langle k_{g}'r(k_{gr}')^{-1}e_1,e_i \right \rangle \right| \leqslant \sigma_1(r)$, we have $$\left|\left \langle k_{g}^{-1}k_{gr}e_1,e_i \right \rangle \right|\leqslant \frac{\sigma_{i}(g)}{\sigma_{1}(g)}\sigma_{1}(r)\sigma_{1}(r^{-1}).$$ Finally, we obtain $$d_{\mathbb{P}}\big(\Xi_{1}^{+}(gr),\Xi_{1}^{+}(g)\big)^2=\sum_{i=2}^{d}\langle k_{g}^{-1}k_{gr}e_1,e_{i} \rangle^2=\sum_{i=2}^{d}\frac{\sigma_{i}(g)^2}{\sigma_{1}(gr)^2}\left \langle k_{g}'r(k_{gr}')^{-1}e_1,e_i \right \rangle^2\leqslant C_{d,r}^2 \frac{\sigma_{2}(g)^2}{\sigma_1(g)^2}.$$ This finishes the proof of the lemma. \end{proof}

Let $\mathcal{M}$ be a compact metrizable space and $\Gamma$ a group acting on $\mathcal{M}$ by homeomorphisms. The action is called a \emph{convergence group action} if for any infinite sequence $(\gamma_n)_{n \in \mathbb{N}}$ of elements of $\Gamma$ there exists a subsequence $(\gamma_{k_n})_{n \in \mathbb{N}}$ and $x,y \in \mathcal{M}$ such that for every compact subset $\mathcal{C} \subset \mathcal{M} \smallsetminus\{x\}$, $\gamma_{k_n}|_{\mathcal{C}}$ converges uniformly to the constant map $y$. For an infinite order element $\gamma \in \Gamma$, we denote by $\gamma^{\pm}$ the local uniform limit of the sequence $(\gamma^{\pm n})_{n \in \mathbb{N}}$. Examples of convergence group actions include the action of a non-elementary word hyperbolic group on its Gromov boundary (see \cite{Gromov}) and the action of a finitely generated group $\Gamma$ on its Floyd boundary $\partial_{f} \Gamma$ (see \cite{Gromov2} and \hbox{\cite[Thm. 2]{Karlsson}).}

We prove a version of \cite[Lem. 9.2]{Canary} which shows, in many cases, that a representation $\rho$ admitting a continuous $\rho$-equivariant limit map in the flag space $G/P_{\theta}^{+}$  is $\theta$-divergent.  For a subset $\mathcal{C} \subset \mathbb{P}(\mathbb{R}^d)$ define $\langle \mathcal{C} \rangle:=\textup{span}\big\{v\in \mathbb{R}^d\smallsetminus \{{\bf 0}\}: [v]\in \mathcal{C} \big\}$. We shall prove first the following lemma.

\begin{lemma} \label{mainlemma} Let $\mathcal{M}$ be a compact metrizable perfect space, $\Gamma$ a torsion-free group acting on $\mathcal{M}$ by homeomorphisms and $\rho:\Gamma \rightarrow \mathsf{GL}_d(\mathbb{R})$ a representation. Suppose that $\Gamma$ acts on $\mathcal{M}$ as a convergence group and there exists a continuous $\rho$-equivariant non-constant map $\xi: \mathcal{M} \rightarrow \mathbb{P}(\mathbb{R}^d)$. Then for every infinite sequence $(\gamma_n)_{n\in \mathbb{N}}$ of elements of $\Gamma$ we have $$\lim_{n \rightarrow \infty}\frac{\sigma_1(\rho(\gamma_n))}{\sigma_{d-p+2}(\rho(\gamma_n))}=+\infty$$ \textup{where} $p=\textup{dim}_{\mathbb{R}}\left \langle \xi\left(\mathcal{M}\right) \right \rangle$. \end{lemma}

\begin{proof} \noindent We first prove the statement when $p=d$. 

If the result does not hold, then there exists $\varepsilon>0$ and a subsequence, which we continue to denote by $(\gamma_n)_{n \in \mathbb{N}}$, such that $\frac{\sigma_2(\rho(\gamma_n))}{\sigma_{1}(\rho(\gamma_n))} \geqslant \varepsilon$. We may write $\rho(\gamma_n)=k_{n}\exp(\mu(\rho(\gamma_n))k_{n}'$, where $k_{n},k_{n}' \in \mathsf{O}(d)$. Up to passing to a subsequence, there exist $\eta,\eta' \in \mathcal{M}$ such that if $x \neq \eta'$ then $\lim_{n}\gamma_n x=\eta$ and  we may also assume that the sequences $(k_n)_{n \in \mathbb{N}}, (k_{n}')_{n \in \mathbb{N}}$ converge to $k,k' \in \mathsf{O}(d)$ respectively, $\lim_n\frac{\sigma_2(\rho(\gamma_n))}{\sigma_{1}(\rho(\gamma_n))}=C>0$  and for every $i>1$ the limit $\lim_n\frac{\sigma_i(\rho(\gamma_n))}{\sigma_1(\rho(\gamma_n))}$ exists. \par  For $z\in \partial_{\infty}\Gamma$ write $\xi(x)=k_{z}P_1^{+}$ for some $k_{z} \in \mathsf{O}(d)$.  Now let $x \in \partial_{\infty}\Gamma\smallsetminus \{\eta\}$. Since $\lim_{n}\rho(\gamma_n)\xi(x)=\xi(\eta)$, up to passing to a further subsequence, we may assume that \begin{equation}\label{limit-gamma}\lim_{n \rightarrow \infty}\frac{\exp \big(\mu(\rho(\gamma_n))\big)k_{n}'k_{x}e_1}{\big|\big|\exp \big(\mu(\rho(\gamma_n))\big)k_{n}'k_{x}e_1\big|\big|}=\epsilon_x k^{-1}k_{\eta}e_1\end{equation} where $\xi(\eta)=k_{\eta}P_1^{+}$, $\epsilon_x \in \{-1,1\}$. Since for every $i>1$, $\lim_n\frac{\sigma_i(\rho(\gamma_n))}{\sigma_1(\rho(\gamma_n))}$ exists, the limit $\lambda_{x}:=\epsilon_x \lim_{n}\frac{||\exp (\mu(\rho(\gamma_n)))k_{n}'k_{x}e_1||}{\sigma_1(\rho(\gamma_n))}$ also exists. By (\ref{limit-gamma}), for every $x \in \mathcal{M}$, we have that \begin{align*}\big \langle k'k_{x}e_1,e_1\big \rangle&=\lambda_{x}\big \langle k^{-1}k_{\eta}e_1,e_1 \big \rangle\\ \big \langle k'k_{x}e_1,e_2\big \rangle&=\lambda_{x}C^{-1}\big \langle k^{-1}k_{\eta}e_1,e_2 \big \rangle.\end{align*} Since  $\mathcal{M}$ is perfect we have $\big \langle \xi \left(\mathcal{M} \smallsetminus \{\eta' \right \}) \big\rangle =\mathbb{R}^d$ and also there exists $x_{0}\neq \eta'$ with $\lambda_{x_0} \neq 0$. Then for every $x\neq \eta'$ we observe that \begin{align*}\big \langle k'k_{x}e_1,e_1\big \rangle&=\frac{\lambda_{x}}{\lambda_{x_0}}\big \langle k k_{x_0}e_1,e_1 \big \rangle\\ \big \langle k'k_{x}e_1,e_2\big \rangle&=\frac{\lambda_{x}}{\lambda_{x_0}}\big \langle k k_{x_0}e_1,e_2 \big \rangle.\end{align*} Therefore, for every $x \neq \eta'$, $k\xi(x)$ lies in the subspace $V=\left \langle kk_{x_0}e_1 \right \rangle+e_{1}^{\perp}\cap e_{2}^{\perp}$, a contradiction since $\dim(V)\leqslant d-1$. This completets the proof when $p=d$. \par In the case where $p<d$, consider the subspace $V=\langle \xi(\mathcal{M}) \rangle$ and the restriction $\hat{\rho}:\Gamma \rightarrow \mathsf{GL}(V)$ of $\rho$. The map $\xi$ is $\hat{\rho}$-equivariant and a spanning map for $\hat{\rho}$. The conclusion follows by observing that for any $\gamma \in \Gamma$  we have ${\frac{\sigma_1(\hat{\rho}(\gamma))}{\sigma_{2}(\hat{\rho}(\gamma))} \leqslant \frac{\sigma_1(\rho(\gamma))}{\sigma_{d-p+2}(\rho(\gamma))}}$.\end{proof}

\begin{corollary} \label{Cartan} Let $\Gamma$ be a word hyperbolic group, $G$ a real semisimple Lie group and $\theta \subset \Delta$ a subset of simple restricted roots of $G$.
\begin{enumerate}[label=(\roman*)]
\item Suppose that $\rho:\Gamma \rightarrow \mathsf{SL}_d(\mathbb{R})$ is an irreducible representation admitting a continuous $\rho$-equivariant map \hbox{$\xi: \partial_{\infty}\Gamma \rightarrow \mathbb{P}(\mathbb{R}^d)$}. Then $\rho$ is $P_1$-divergent and $\xi$ satisfies the Cartan property.
\item Suppose that $\rho':\Gamma \rightarrow G$ is a Zariski dense representation admitting a continuous $\rho'$-equivariant map \hbox{$\xi': \partial_{\infty}\Gamma \rightarrow G/P_{\theta}^{+}$.} Then $\rho'$ is $P_{\theta}$-divergent and $\xi'$ satisfies the Cartan property.\end{enumerate} \end{corollary}

\begin{proof} \textup{(i)} We first claim that if $\rho(\gamma)$ is $P_1$-proximal, then $\xi(\gamma^{+})$ is the attracting fixed point in $\mathbb{P}(\mathbb{R}^d)$. Indeed, since $\rho$ is irreducible we have \hbox{$ \langle \xi (\partial_{\infty}\Gamma) \rangle=\mathbb{R}^d$}. If $\rho(\gamma)$ is $P_1$-proximal, we can find $x \in \partial_{\infty}\Gamma\smallsetminus \{\gamma^{-}\}$ such that $\xi(x)$ is not in the repelling hyperplane $V_{\gamma}^{-}$. Since $\lim_{n}\gamma^n x=\gamma^{+}$, we have $\xi(\gamma^{+})=x_{\rho(\gamma)}^{+}$. \par Since $\rho$ is irreducible it follows by Lemma \ref{mainlemma} that $\rho$ is $P_1$-divergent. Let $(\gamma_n)_{n \in \mathbb{N}}$ be an infinite sequence of elements of $\Gamma$ such that $\lim_{n}\gamma_n=x$. By the sub-additivity of the Cartan projection $\mu$ (see \cite[Fact 2.18]{GGKW}) and Theorem \ref{finitesubset}, there exists a finite subset $F$ and $C>0$ such that for every $\gamma \in \Gamma$, there exists $f \in F$ with $\big|\big|\lambda(\rho(\gamma f))-\mu(\rho(\gamma f))\big|\big| \leqslant C$. Then for large $n \in \mathbb{N}$, there exist $f_n \in F$ such that $\rho(\gamma_n f_n)$ is $P_1$-proximal and $$\log\ell_1(\rho(\gamma_nf_n))-\log\sigma_1(\rho(\gamma_nf_n))  \geqslant-C.$$ Notice also $\lim_{n}\gamma_n=\lim_{n}\gamma_nf_n=\lim_{n}(\gamma_nf_n)^{+}=x$ in the compactification $\Gamma \cup \partial_{\infty}\Gamma$ and so $\lim_{n}x_{\rho(\gamma_n f_n)}^{+}=\lim_{n}\xi((\gamma_n f_n)^{+})=\xi(x)$. Then, by using Proposition \ref{inequalities}, for every $n \in \mathbb{N}$ we obtain the estimate: \begin{align*} d_{\mathbb{P}} \big( x_{\rho(\gamma_n f_n)}^{+}, \Xi_{1}^{+}\big(\rho (\gamma_n)\big) \big)& \leqslant  d_{\mathbb{P}} \big( x_{\rho (\gamma_n f_n)}^{+}, \Xi_{1}^{+}\big(\rho (\gamma_n f_n)\big) \big)+d_{\mathbb{P}} \big( \Xi_{1}^{+}\big(\rho (\gamma_n f_n)\big), \Xi_{1}^{+}\big(\rho (\gamma_n)\big) \big)\\ &\leqslant \Big(e^{C}+\sup_{f \in F}C_{d, f} \Big) \frac{\sigma_2(\rho(\gamma_n))}{\sigma_1(\rho (\gamma_n))}\end{align*} where $C_{d,f}>0$ is defined as in Proposition \ref{inequalities} (ii). This shows \hbox{$\xi (x)=\lim_{n}\Xi_{1}^{+}\big(\rho (\gamma_n)\big)$} and finally that $\xi$ satisfies the Cartan property.
\medskip

\noindent \textup{(ii)} Let $\tau_{\theta}:G\rightarrow \mathsf{GL}_d(\mathbb{R})$ and $(\iota^{+},\iota^{-})$ be as in Proposition \ref{higherdimension}. Since $\rho'$ is Zariski dense, the representation $\tau_{\theta} \circ \rho'$ is irreducible. By Lemma \ref{mainlemma} the representation $\tau_{\theta} \circ \rho'$ is $P_1$-divergent and hence $\rho'$ is $P_{\theta}$-divergent. By part \textup{(i)}, the $\tau_{\theta} \circ \rho'$-equivariant map $\iota^{+} \circ \xi'$ satisfies the Cartan property. It follows by Fact \ref{Cartan2} that $\xi'$ satisfies the Cartan property.  \end{proof}

\par We are now aiming to generalize the uniform gap summation property \cite[Def. 5.2]{GGKW} for representations of arbitrary finitely generated groups. 

\begin{definition} \label{ugspdef} Let $\Gamma$ be a finitely generated group, $\rho:\Gamma \rightarrow G$ a representation and $\theta \subset \Delta$ a finite subset of restricted roots of $G$. We say that $\rho$ satisfies the uniform gap summation property with respect to $\theta$ and the Floyd function $f:\mathbb{N}\rightarrow (0,+\infty)$, if there exists $C>0$ such that $$\alpha\big(\mu(\rho(\gamma)) \big) \geqslant -\log f(|\gamma|_{\Gamma})-C$$ for every $\gamma \in \Gamma$ and $\alpha \in \theta$. We say that the representation $\rho$ satisfies the uniform gap summation property if there exists a Floyd function $f$, a subset of simple roots $\theta \subset \Delta$ and $C>0$ with the previous properties. \end{definition}  

\par Let $\rho:\Gamma \rightarrow G$ be a representation. If $\Gamma$ is word hyperbolic group and $\rho$ satisfies the uniform gap summation property, then it admits a pair of $\rho$-equivariant, continuous limit maps which satisfy the Cartan property (see \cite[Thm. 5.3 (3)]{GGKW}). If $\Gamma$ is not word hyperbolic, we may similarly construct a pair of $\rho$-equivariant continuous maps from a Floyd boundary of $\Gamma$, $\partial_{f}\Gamma$, into the flag spaces $G/P_{\theta}^{+}$ and $G/P_{\theta}^{-}$. Note that when $\partial_{f}\Gamma$ is non-trivial, the action of $\Gamma$ on $\partial_{f}\Gamma$ is a convergence group action (see \cite[Thm. 2]{Karlsson}) so we obtain additional information for the action of $\rho(\Gamma)$ on its limit set in $G/P_{\theta}^{+}$.  For a subgroup $H$ of $G$ containing a $P_{\theta}$-proximal element, its limit set in $G/P_{\theta}^{+}$ is the closure of attracting fixed points of $\theta$-proximal elements of $H$. In the case where $H\subset \mathsf{SL}_d(\mathbb{R})$ and $G/P_{\theta}^{+}=\mathbb{P}(\mathbb{R}^d)$, $P_{\theta}^{+}=P_1$, we denote by $\Lambda_H$ its proximal limit set in $\mathbb{P}(\mathbb{R}^d)$. If $H$ is in addition an irreducible subgroup of $\mathsf{GL}_d(\mathbb{R})$, then $H$ acts minimally on $\Lambda_H$, see \cite[Lem. 2.5]{Benoist-convexcones}.

We prove the following lemma that we use in the following section \hbox{for the proof of Theorem \ref{wg}.}

\begin{lemma} \label{existence} Let $\Gamma$ be a finitely generated group, $G$ a real semsimiple Lie group, $\theta \subset \Delta$ a subset of simple restricted roots of $G$ and $\rho:\Gamma \rightarrow G$ a representation. Suppose that $\rho$ satisfies the uniform gap summation property with respect to $\theta$ and the Floyd function $f:\mathbb{N}\rightarrow (0,\infty)$. There exists a constant $C>0$, depending only on $\rho$, such that $$d_{G/P_{\theta}^{\pm}}\big(\Xi_{\theta}^{\pm}\big(\rho(g)\big),\Xi_{\theta}^{\pm}\big(\rho(h)\big) \big) \leqslant C  d_{f}(g,h)$$ for all but finitely many $g,h \in \Gamma$. In particular, \hbox{there exists a pair of continuous $\rho$-equivariant maps} $$\xi_{f}^{+}:\partial_{f}\Gamma \rightarrow G/P_{\theta}^{+} \ \ and \ \ \xi_{f}^{-}:\partial_{f}\Gamma \rightarrow G/P_{\theta}^{-}.$$ Moreover, if $\rho(\Gamma)$ contains a $P_{\theta}$-proximal element, then $\xi_{f}^{+}(\partial_{f}\Gamma)$ maps onto the proximal limit set of $\rho(\Gamma)$ in $G/P_{\theta}^{+}$. \end{lemma}

\begin{proof} As in the proof of Proposition \ref{higherdimension}, we may assume that $\theta=\{\varepsilon_1-\varepsilon_2\}$ and $G=\mathsf{SL}_d(\mathbb{R})$ and $G/P_{\theta}^{+}=\mathbb{P}(\mathbb{R}^d)$. By definition, there exists a constant $C>0$ such that for every $\gamma\in \Gamma$, $$\frac{\sigma_2(\rho(\gamma)}{\sigma_1(\rho(\gamma))} \leqslant C f(|\gamma|_{\Gamma}).$$ Let $\p\subset C_{\Gamma}$ be a path in the Cayley graph of $\Gamma$ defined by the sequence of adjacent vertices $g_0=g,\ldots ,h=g_{n}$ with $L_{f}(\p)=d_{f}(g,h)$. Since for every $i$, $g_{i}^{-1}g_{i+1}$ lies in a fixed generating subset of $\Gamma$, by Proposition \ref{inequalities}, there is $C'>0$, depending only on $\rho$, such that: \begin{align}\label{existence-eq1} \nonumber d_{\mathbb{P}}\big(\Xi_{1}^{\pm }\big(\rho(g)\big),\Xi_{1}^{\pm}\big(\rho(h)\big) \big)& \leqslant \sum_{i=0}^{n-1}d_{\mathbb{P}}\big(\Xi_{1}^{\pm }\big(\rho(g_i)\big),\Xi_{1}^{\pm}\big(\rho(g_{i+1}\big) \big) \leqslant C'\sum_{i=0}^{n-1}\frac{\sigma_2(\rho(g_i)^{\pm 1})}{\sigma_1(\rho(g_{i})^{\pm 1})}\\ &\leqslant C'C\sum_{i=0}^{n-1}f(|g_{i}|_{\Gamma})=C'Kd_{f}(g,h).\end{align}
\par  Now define the maps $\xi_{f}^{+}:\partial_{f}\Gamma \rightarrow \mathbb{P}(\mathbb{R}^d)$ and $\xi_{f}^{-}:\partial_{f}\Gamma \rightarrow \mathsf{Gr}_{d-1}(\mathbb{R}^d)$ as follows: for a point $x \in \partial_{f}\Gamma$ represented by a Cauchy sequence $(\gamma_n)_{n \in \mathbb{N}}$ with respect to the metric $d_{f}$, define $\xi_{f}^{\pm}(x)$, $$\xi^{\pm}_{f}(x):=\lim_{n \rightarrow \infty} \Xi_{1}^{\pm}\big(\rho(\gamma_n)\big)$$ The bound (\ref{existence-eq1}) shows that the limit $\lim_{n}\Xi_{1}^{+}\big(\rho(\gamma_n)\big)$ is independent of the choice of the sequence $(\gamma_n)_{n \in \mathbb{N}}$ representing $x$, since for any other sequence $(\gamma_{n}')_{n \in \mathbb{N}}$ with $x=\lim_{n}\gamma_n'$, we have $\lim_{n}d_{f}(\gamma_n,\gamma_n')=0$. Finally, $\xi^{+}_{f}$ is well-defined and Lipschitz by (\ref{existence-eq1}) and hence continuous. By identifying $G/P_{\theta}^{-}$ with $G/P_{\theta^{\star}}^{+}$, we similarly obtain deduce that the limit map $\xi_{f}^{-}$ is well-defined and continuous. \par Suppose that $\rho(\Gamma)$ is $P_1$-proximal. By the definition of the map $\xi_{f}^{+}$ (resp. $\xi_{f}^{-}$), if $\rho(\gamma_0)$ is $P_1$-proximal (resp. $P_{d-1}$-proximal), then $\xi_{f}^{+}(\gamma_{0}^{+})$ (resp. $\xi_{f}^{-}(\gamma_{0}^{+})$) is the unique attracting fixed point of $\rho(\gamma_0)$ in $\mathbb{P}(\mathbb{R}^d)$ (resp. $\mathsf{Gr}_{d-1}(\mathbb{R}^d)$). Since $\Gamma$ acts minimally (e.g. see \cite{Karlsson}) on $\partial_{f}\Gamma$ we deduce that $\xi_{f}^{+}(\partial_{f}\Gamma)$ is the proximal  limit set of $\Gamma$ in $\mathbb{P}(\mathbb{R}^d)$.  \end{proof}

\section{Property (U), weak eigenvalue gaps and the uniform gap summation property} \label{PropU}

In this section, we prove Theorem \ref{wg} providing conditions under which a representation with a weak uniform gap in eigenvalues is Anosov. We also discuss (strong) property (U) and its relation with the uniform gap summation property. \par Property (U) and strong property (U) were introduced by Delzant--Guichard--Labourie--Mozes \cite{DGLM} and Kassel--Potrie \cite{kassel-potrie} respectively and are related to the growth of the translation length and stable translation length of group elements in terms of their word length. 

\begin{definition}\label{PropU1} Let $\Gamma$ be a finitely generated group and fix $|\cdot|_{\Gamma}:\Gamma \rightarrow \mathbb{N}$ a word length function on $\Gamma$. The group $\Gamma$ satisfies property (U) \textup{(}resp. strong property (U)\textup{)} if there exists a finite subset $F$ of $\Gamma$ and $C,c>0$ with the following property: for every $\gamma \in \Gamma$ there exists $w \in F$ such that $$ \ell_{\Gamma}(w \gamma)\geqslant c|\gamma|_{\Gamma}-C \ \ \big(\textup{resp.} \ |w \gamma|_{\infty}\geqslant c|\gamma|_{\Gamma}-C\big).$$\end{definition}

Note that (strong) property (U) is independent of the choice of the left invariant word metric on $\Gamma$ since any two such metrics on $\Gamma$ are quasi-isometric.
\par Delzant--Guichard--Labourie--Mozes \cite{DGLM} proved that every finitely generated group admitting a semisimple quasi-isometric embedding into a reductive Lie group satisfies (strong) property (U). We now prove Theorem \ref{nontrivial} which implies that virtually torsion-free finitely generated groups with non-trivial Floyd boundary\footnote{Kassel--Potrie established an analogue of the Abels--Margulis--Soifer lemma \cite[Thm. 5.17]{AMS} simultaneously for a linear representation $\rho:\Gamma \rightarrow \mathsf{GL}_d(\mathbb{R})$ of a word hyperbolic group and the abstract group $\Gamma$ equipped with a left invariant word metric (see \cite[Cor. 1.8]{kassel-potrie-AMSlemma}).} satisfy strong property (U) (and hence property (U)).

 Let us recall that the Floyd boundary $\partial_{f}\Gamma$ of $\Gamma$ with respect to a Floyd function $f$ is called non-trivial if $|\partial_{f}\Gamma|\geq 3$. For a subgroup $H$ of $\Gamma$, its limit set $\Lambda(H)\subset \partial_{f}\Gamma$ is the set of accumulation points of infinite sequences of elements of $H$ in $\partial_{f}\Gamma$.

\begin{proof}[Proof of Theorem \ref{nontrivial}.] Let $G:[1,\infty)\rightarrow (0,\infty)$ be the function $G(x):=10\sum_{k = \left \lfloor x/2 \right \rfloor}^{\infty} f(k)$. Note that $G$ is decreasing and $\underset{x \rightarrow \infty}{\lim}G(x)=0$. By Karlsson's estimate, see \cite[Lem. 1]{Karlsson}, we have $$d_{f}\big(g,h\big) \leqslant G\big((g \cdot h)_{e}\big),  \ d_{f}\big(g, g^{+}\big) \leqslant G\Big(\frac{1}{2}|g|_{\Gamma}\Big)$$ for every $g,h \in \Gamma$, where $g$ has infinite order.  Since $|\Lambda(H)|\geq 3$, by \cite[Prop. 5]{Karlsson}, we may find $f_1,f_2 \in H$ infinite order elements such that $\{ f_{1}^{+},f_{1}^{-}\}\cap \{f_{2}^{+},f_{2}^{-}\}=\emptyset$. \hbox{Let us set} $$\varepsilon:=10^{-2}\min \big\{ d_{f}(f_{1}^{+},f_{2}^{+}), d_{f}(f_{1}^{+},f_{2}^{-}), d_{f}\big(f_{1}^{-},f_{2}^{+}), d_{f}(f_{1}^{-},f_{2}^{-})\big\}>0$$ and make the following three choices of constants $M,R,N>0$ as follows:
\medskip

\noindent \textup{(i)} $M>0$ is chosen such that $G(x) \geqslant \frac{\varepsilon}{100}$ if and only if $x \leq M$,\\
\noindent \textup{(ii)} $R>0$ is chosen such that $G(x) \leq \frac{\varepsilon}{100}$ for every $x \geqslant R$,\\
\noindent \textup{(iii)} $N>0$ is chosen such that $\min \big \{\big|f_{1}^N\big|_{\Gamma}, \big|f_{2}^N\big|_{\Gamma}\big \} \geqslant 10(M+R)$.

Now we prove the folllowing claim: 
\medskip

\noindent {\em Claim. 1 Let $F:=\big \{ f_{1}^{N},f_{2}^{N},e \big \}$. For every non-trivial $\gamma \in H$ there exists $g \in F$ such that $d_{f}\big(g \gamma^{+},\gamma^{-}\big)\geqslant \varepsilon$.}

\begin{proof}[Proof of Claim 1] If $d_{f}(\gamma^{+},\gamma^{-}) \geqslant \varepsilon$ we choose $g=e$. So we may assume that $d_{f}(\gamma^{+},\gamma^{-})\leqslant \varepsilon$. We can choose $n_0 \in \mathbb{N}$ such that $G\big(\frac{1}{2}|\gamma^n|_{\Gamma}\big)<\varepsilon$ for $n \geqslant n_0$. Notice that we can find $i \in \{1,2\}$ such that $d_{f}(\gamma^{+}, f_{i}^{+}) \geqslant 50 \varepsilon$ and $d_{f}(\gamma^{+}, f_{i}^{-}) \geqslant 50 \varepsilon$. Indeed, if we assume that $\textup{dist}\big(\gamma^{+}, \big \{f_1^{+},f_{1}^{-}\big\}\big)<50 \varepsilon$ then $d_{f}(\gamma^{+}, f_{2}^{\pm}) \geqslant \textup{dist}(f_{2}^{\pm}, \{f_{1}^{+},f_{1}^{-}\})-50\varepsilon \geqslant 50 \varepsilon$. Without loss of generality we may assume $d_{f}(\gamma^{+}, f_{1}^{+}) \geqslant 50 \varepsilon$ and $d_{f}(\gamma^{+}, f_{1}^{-}) \geqslant 50 \varepsilon$. By our choices of $N,n_0>0$ we have \begin{align*} d_{f}\big(\gamma^{n}, f_{1}^{-N}\big)  &\geqslant d_{f}\big(\gamma^{+}, f_{1}^{-}\big)-d_{f}\big(f_{1}^{-},f_{1}^{-N}\big)-d_{f}\big(\gamma^{+},\gamma^{n}\big)\\ &\geqslant 50\varepsilon -G\Big(\frac{1}{2}|f_1^N|_{\Gamma}\Big)-G\Big(\frac{1}{2}|\gamma^n|_{\Gamma}\Big)\geqslant 48\varepsilon,\end{align*} hence $G\big(\big(\gamma^{n}\cdot f_{1}^{-N}\big)_{e}\big) \geq \varepsilon$ for $n \geqslant n_0$. By the choice of $M>0$ we have that $\big(\gamma^{n}\cdot f_{1}^{-N}\big)_{e} \leqslant M$ for $n \geqslant n_0$. Then, we choose an infinite sequence $(k_n)_{n \in \mathbb{N}}$ such that $\big |f_{1}^{k_n-N}\big |_{\Gamma}<\big|f_{1}^{k_n}\big|_{\Gamma}$ for every $n \in \mathbb{N}$. For $n \geqslant n_0$ we have \begin{align*} 2\big(f_{1}^{N}\gamma^{n}\cdot f_{1}^{k_n}\big)_{e}& =\big|f_{1}^{N}\gamma^n \big|_{\Gamma}+\big|f_{1}^{k_n}\big|_{\Gamma}-\big|f_{1}^{N-k_n}\gamma^n \big|_{\Gamma}\\  &=\big|\gamma^n \big|_{\Gamma}+\big|f_1^N \big|_{\Gamma}-2\big(\gamma^n \cdot f_{1}^{-N}\big)_{e}+\big|f_{1}^{k_n} \big|_{\Gamma}-\big|f_{1}^{N-k_n}\gamma^n \big|_{\Gamma}\\ &\geqslant -2M+\big|f_1^N \big|_{\Gamma}+\big|f_{1}^{k_n} \big|_{\Gamma}-\big|f_1^{N-k_n} \big|_{\Gamma}\\ & \geqslant \big|f_1^N\big|_{\Gamma}-2M \geqslant \frac{\big|f_1^N \big|_{\Gamma}}{2} \geqslant 2R. \end{align*} Thus, by the choice of $R>0$ we have $G\big(\big( f_1^N \gamma^{n}\cdot f_{1}^{k_n}\big)_{e}\big) \leq \varepsilon$, $n \geqslant n_0$. \hbox{It follows that $d_{f}\big(f_{1}^{N}\gamma^{+},f_{1}^{+}\big)\leqslant \varepsilon$ so} $$d_{f}\big(f_{1}^{N}\gamma^{+}, \gamma^{-}\big) \geqslant d_{f}(\gamma^{+},f_{1}^{+})-d_{f}\big(f_{1}^{N}\gamma^{+}, f_{1}^{+})-d_{f}\big(\gamma^{+}, \gamma^{-}) \geqslant 48 \varepsilon$$ and Claim 1 follows.\end{proof}
 \par Now, let $L_0:=10\max_{g \in F}|g|_{\Gamma}+2R$. If $\gamma\in  H$ and $|\gamma|_{\Gamma}<L_0$, then we choose $g =e$ and obviously $|\gamma|_{\Gamma}-|\gamma|_{\infty}\leq L_0$. Suppose that $\gamma \in H$ and $|\gamma|_{\Gamma}\geq L_0$. We may choose $g \in F$ such that $d_{f}\big((g\gamma g^{-1})^{+}, \gamma^{-} \big) \geqslant \varepsilon$, where $(g \gamma g^{-1})^{+}=g \gamma^{+}$ in $\partial_{f}\Gamma$. We observe that \begin{align*} d_{f}\big( (g \gamma g^{-1})^{+}, (g \gamma)^{+} \big)& \leqslant d_{f}\big( (g \gamma g^{-1})^{+}, g \gamma g^{-1} \big)+d_{f}\big(g\gamma g^{-1}, g \gamma \big)+d_{f} \big( (g \gamma)^{+}, g \gamma \big)\\ &\leqslant G\Big(\frac{1}{2}\big| g \gamma g^{-1}\big |_{\Gamma}\Big)+G\big((g\gamma g^{-1} \cdot g\gamma)_{e}\big)+G\Big(\frac{1}{2}\big| g \gamma \big |_{\Gamma}\Big)\\ &\leqslant 3 G\Big(\frac{1}{2}|\gamma|_{\Gamma}-2|g|_{\Gamma}\Big)\leqslant \frac{3\varepsilon}{100},\\ d_{f}\big(\gamma^{-}, \gamma^{-1} g^{-1} \big)& \leqslant d_{f}\big( \gamma^{-}, \gamma^{-1} \big) +d_{f}\big( \gamma^{-1}, \gamma^{-1} g^{-1} \big)\\ & \leqslant G\Big(\frac{1}{2}|  \gamma|_{\Gamma}\Big)+G\big((\gamma^{-1} \cdot \gamma^{-1}g^{-1})_{e}\big) \leqslant 2G\Big(\frac{1}{2}|\gamma|_{\Gamma}-2|g|_{\Gamma}\Big) \leqslant \frac{\varepsilon}{50},\end{align*} since $|\gamma|_{\Gamma}-2|g|_{\Gamma}>R$. Therefore, by the previous bounds we have $$d_{f}\big ((g \gamma)^{+}, \gamma^{-1} g^{-1} \big) \geqslant d_{f}\big (g \gamma^{+},\gamma^{-}\big)-d_{f}(g\gamma^{+}, (g\gamma)^{+}\big)-d_{f}\big(\gamma^{-1},\gamma^{-1}g^{-1}\big)\geqslant \frac{\varepsilon}{2}.$$ This shows that there is $n_1>0$ with $G\big( \big((g \gamma)^{n}\cdot (g \gamma)^{-1}\big)_{e}\big) \geqslant \frac{\varepsilon}{3}$ and $\big((g \gamma)^{n}\cdot (g \gamma)^{-1}\big)_{e}\leqslant M$ for every $n \geqslant n_1$. We can find a sequence $(m_n)_{n \in \mathbb{N}}$ such that $$\lim_{n \rightarrow \infty} \big( \big|(g \gamma)^{m_n+1}\big|_{\Gamma}-\big|(g \gamma)^{m_n}\big|_{\Gamma}\big)\leqslant |g \gamma|_{\infty}$$ so $\lim_{n} 2\big ((g \gamma)^{m_n}\cdot (g \gamma)^{-1} \big)_{e} \geqslant |g\gamma|_{\Gamma}-|g \gamma|_{\infty} $. Finally, since $R>M$, we conclude that $$|\gamma|_{\Gamma}-|g\gamma|_{\infty}-\big(\max_{g\in F}|g|_{\Gamma}\big)\leq |g \gamma|_{\Gamma}-|g \gamma|_{\infty} \leqslant 2M \leqslant L_0.$$ In particular, we conclude that $\Gamma$ has strong property (U) and this completes the proof of the theorem.\end{proof}

\subsection{Weak uniform gaps in eigenvalues.}  Recall for a  matrix $g\in \mathsf{GL}_d(\mathbb{R})$, $\ell_1(g)\geq \cdots \geq \ell_d(g)$ are the moduli of the eigenvalues of $g$. Recall also that a linear representation \hbox{$\rho:\Gamma \rightarrow \mathsf{GL}_d(\mathbb{R})$} has a weak uniform $i$-gap in eigenvalues if there exists $\varepsilon>0$ such that for every $\gamma \in \Gamma$, $$\log\frac{\ell_i(\rho(\gamma))}{\ell_{i+1}(\rho(\gamma))} \geq \varepsilon |\gamma |_{\infty}.$$

\par  For a group $\Gamma$ the {\em lower central series} $$\cdot \cdot \cdot \trianglelefteq \mathfrak{g}_3(\Gamma) \trianglelefteq \mathfrak{g}_2(\Gamma) \trianglelefteq \mathfrak{g}_1(\Gamma) \trianglelefteq \mathfrak{g}_{0}(\Gamma):=\Gamma$$ is inductively defined as $\mathfrak{g}_{k+1}(\Gamma)=\big[\Gamma, \mathfrak{g}_{k}(\Gamma)\big]$ for $k \geq 1$. For every $k$, $\mathfrak{g}_{k}(\Gamma)$ is a characteristic subgroup of $\Gamma$ and the quotient $\mathfrak{g}_{k}(\Gamma)/\mathfrak{g}_{k+1}(\Gamma)$ is a central subgroup of $\Gamma/\mathfrak{g}_{k+1}(\Gamma)$. The group $\Gamma$ is {\em nilpotent} if there exists $m \geqslant 0$ with $\mathfrak{g}_{m}(\Gamma)=1$.

\par First, we prove the following technical lemma showing that a nilpotent group $\Gamma$ which admits a representation with weak uniform eigenvalue $i$-gap has to be virtually cyclic.\footnote{This is not true when $\Gamma$ is assumed to be solvable. The Baumslag--Solitar group $\mathsf{BS}(1,2)$ admits a faithful representation into $\mathsf{GL}_2(\mathbb{R})$ with a weak uniform $1$-gap (see \cite[Ex. 4.8]{kassel-potrie}).}

\begin{lemma} \label{nil1} Let $\Gamma$ be a finitely generated nilpotent group. Suppose that $\rho:\Gamma \rightarrow \mathsf{GL}_d(\mathbb{R})$ has a weak uniform $i$-gap in eigenvalues for some $1 \leq i \leq d-1$. Then $\Gamma$ is virtually cyclic. \end{lemma}

\begin{proof} We need the following elementary observation: for a group $G_1$ and a central subgroup $G_2\subset Z(G_1)$ of $G_1$, if the quotient $G_1/G_2$ is virtually cyclic, then $G_1$ is virtually abelian.
\par Let $G$ be the Zariski closure of $\rho(\Gamma)$ in $\mathsf{GL}_d(\mathbb{R})$. We consider the Levi decomposition $G=L \ltimes U$, where $U$ is a connected normal unipotent subgroup of $G$ and $L$ is a reductive Lie group. The projection $\pi \circ \rho:\Gamma \rightarrow L$ is Zariski dense and $\lambda(\pi(\rho(\gamma)))=\lambda(\rho(\gamma))$ for every $\gamma \in \Gamma$. The Lie group $L$ is reductive and $\pi(\rho(\Gamma))$ is solvable, so $L$ has to be virtually abelian since it has finitely many connected components. We may find a finite-index subgroup $H$ of $\Gamma$ such that $\mathfrak{g}_1(H)=[H,H]$ is a subgroup of $\textup{ker}(\pi\circ \rho)$. Therefore, for $k\geqslant 1$ we obtain a well-defined representation $\rho_k:H/\mathfrak{g}_k(H) \rightarrow \mathsf{GL}_d(\mathbb{R})$ such that $\rho_k\circ \pi_k=\pi \circ \rho$, where $\pi_k:H \rightarrow H/\mathfrak{g}_k(H)$ is the quotient map. Note that for every $k \geq 1$ there exists $c_k\geq 1$ such that for every $h\in H$, $$|\pi_k(h)|_{H/\mathfrak{g}_k(H),\infty} \leqslant c_k |h |_{H,\infty}.$$ Since $\lambda(\rho_k(h))=\lambda(\rho(h))$ for every $h \in H$, $\rho_k$ has a weak uniform $i$-gap in eigenvalues for every $k \geq 1$. We may use induction on $k\in \mathbb{N}$ to see that $H/\mathfrak{g}_k(H)$ is virtually cyclic. The group $H/\mathfrak{g}_1(H)$ is abelian and satisfies strong property (U), so $\rho_1$ is $P_i$-Anosov by \cite[Prop. 4.12]{kassel-potrie} and $H/\mathfrak{g}_1(H)$ has to be virtually cyclic. Now suppose that $H/\mathfrak{g}_k(H)$ is virtually cyclic. Note that $\mathfrak{g}_k(H)/\gamma_{k+1}(H)$ is a central subgroup of $H/\mathfrak{g}_{k+1}(H)$ with virtually cyclic quotient $H/\mathfrak{g}_k(H)$. It follows that $H/\mathfrak{g}_{k+1}(H)$ is virtually abelian. In particular, $H/\mathfrak{g}_{k+1}(H)$ satisfies strong property (U), so $\rho_{k+1}$ is $P_i$-Anosov and $H/\mathfrak{g}_{k+1}(H)$ is virtually cyclic. Therefore, $H/\mathfrak{g}_k(H)$ has to be virtually cylic for every $k \geq 1$ and $H$ is virtually cyclic since $\mathfrak{g}_m(H)=1$ for some $m \geq 1$.
\end{proof}

 As a corollary of Theorem \ref{nontrivial}, we obtain Corollary \ref{ugspu} which shows that a non-virtually nilpotent group $\Gamma$ which admits a representation with the uniform gap summation property satisfies strong Property (U).

\begin{proof}[Proof of Corollary \ref{ugspu}.] By Proposition \ref{higherdimension} we may assume that $G=\mathsf{SL}_d(\mathbb{R})$ and $\theta=\{\varepsilon_1-\varepsilon_2\}$. Since $\rho$ satisfies the uniform gap summation property $\textup{ker}(\rho)$ is finite. It suffices to prove that a finite-index subgroup of $\Gamma'=\Gamma/\textup{ker}(\rho)$ satisfies strong property (U). By Selberg's lemma \cite{Selberg}, $\Gamma'$ is virtually torsion-free, so we may assume that $\Gamma$ is torsion-free and $\rho$ is faithful. By Lemma \ref{existence} there exists a continuous $\rho$-equivariant map \hbox{$\xi_{f}:\partial_{f}\Gamma \rightarrow \mathbb{P}(\mathbb{R}^d)$} for some Floyd function $f$.  We first prove that $\partial_{f}\Gamma$ is not a singleton.\par Suppose that $|\partial_{f}\Gamma|=1$. By the definition of the map $\xi_{f}$, the image $\xi_{f}(\partial_f \Gamma)$ is the $\tau_{\textup{mod}}$-limit set of $\Gamma$ in $\mathbb{P}(\mathbb{R}^d)$. Since $\Gamma$ is not virtually nilpotent, we may use \cite[Cor. 5.10]{KLrel} to check that $\partial_f \Gamma$ contains at least two points. We provide here the following different argument that we also use to show that $|\partial_f\Gamma|\neq 2$. 

Now assume that $\partial_{f}\Gamma$ is a singleton. We shall obtain a contradition. Up to conjugation, we may assume that $\xi_{f}(\partial_{f}\Gamma)=[e_1]$ and find a group homomorphism $\alpha:\Gamma \rightarrow \mathbb{R}^{\ast}$ such that for every $\gamma \in \Gamma$, $$\rho(\gamma)e_1=\alpha(\gamma)e_1.$$ We consider the representation $\hat{\rho}(\gamma)=\frac{1}{\alpha(\gamma)}\rho(\gamma)$. Note that $\hat{\rho}$ satisfies the uniform gap summation property (since $\rho$ does), $\xi_{f}$ is $\hat{\rho}$-equivariant and we can write $$\hat{\rho}(\gamma)=\begin{pmatrix}
1 & u(\gamma) \\ 
0 & \rho_0(\gamma)
\end{pmatrix}$$ for some group homomorphism $\rho_0:\Gamma \rightarrow \mathsf{GL}_{d-1}(\mathbb{R})$. Let $g \in \Gamma \smallsetminus \{e\}$. Since $\xi_{f}$ is constant we have \hbox{$\lim_{n}\Xi_{1}^{+}(\hat{\rho}(g^n))=\lim_{n}\Xi_{1}^{+}(\hat{\rho}(g^{-n}))=[e_1]$}. Let us write $\hat{\rho}(g^n)=k_{n}\exp \big(\mu(\hat{\rho}(g^n))\big)k_{n}'$ in the Cartan decomposition of $G$, and up to passing to a subsequence, we may assume $\lim_nk_{n}=k_{\infty}$ and $\lim_{n}k_{n}'=k_{\infty}'$. Then $k_{\infty}'P_{1}^{+}=wP_{1}^{+}$, $\langle k_{\infty}'e_1,e_1 \rangle=0$ and $|\langle k_{\infty}e_1,e_1 \rangle|=1$, so $$\lim_{n \rightarrow \infty}\frac{\hat{\rho}(g^n)}{\sigma_1(\hat{\rho}(g^n))}=k_{\infty}E_{11}k_{\infty}' \in \bigoplus_{i=2}^{d}\mathbb{R}E_{1i}.$$ If $\ell_1(\hat{\rho}(g))>1$, then $\ell_1(\rho_0(g))=\ell_1(\hat{\rho}(g))$. Let $p_1\in \mathbb{N}$ and $p_2\in \mathbb{N}$ be the largest possible dimension of a Jordan block for an eigenvalue of maximum modulus of $\hat{\rho}(g)$ and $\rho_0(g)$ respectively. A straightforward calculation shows that $$\sigma_1(\hat{\rho}(g^n)) \asymp n^{p_1-1}\ell_1(\hat{\rho}(g^n)),\ \sigma_1(\rho_0(g^n)) \asymp n^{p_2-1}\ell_1(\hat{\rho}(g^n)), \ \ n\rightarrow \infty$$ and $p_1>p_2$ since $\lim_{n}\frac{\rho_0(g^n)}{\sigma_1(\hat{\rho}(g^n))}=0$. In particular, there exists $C>0$ such that $$\big|\big|u(g^n)\big|\big| =\Big| \Big|\sum_{i=0}^{n} \rho_0(g^i)^{t} u(g)\Big| \Big| \leqslant \big|\big|u(g)\big|\big| \sum_{i=0}^{n} i^{p_2-1} \ell_1(\hat{\rho}(g))^{i} \leq Cn^{p_2-1}\ell_1(\hat{\rho}(g))^n$$ for every $n \in \mathbb{N}$. Since $p_1>p_2$ and $\ell_1(\rho_0(g))>1$ we have $$\lim_{n \rightarrow \infty} \frac{\sum _{i=0}^{n} i^{p_2-1} \ell_1(\hat{\rho}(g))^{i}}{n^{p_1-1}\ell_1(\hat{\rho}(g^n))}=0.$$ Therefore, $\lim_{n}\frac{||u(g^n)||}{\sigma_1(\hat{\rho}(g^n))}=0$ which is impossible since $\lim_{n}\frac{\hat{\rho}(g^n)}{\sigma_1(\hat{\rho}(g^n))}$ has at least one of its $(1,2),\ldots,(1,d)$ entries non-zero. It follows that $\ell_1(\hat{\rho}(g))\leq 1$ and $\ell_1(\rho(g)) \leq |\alpha(g)|$. Similarly, we obtain $\ell_d(\rho(g))^{-1}=\ell_{1}(\rho(g^{-1}))\leq |\alpha(g^{-1})|$. It follows that all the eigenvalues of $\rho(g)$ have modulus equal to $1$. Therefore, by Theorem \ref{finitesubset},  any semisimplification of $\rho$ has compact Zariski closure. Then, by using \cite[Thm. 3]{Auslander} and \cite[Thm. 10.1]{KLrel}, we conclude that $\rho(\Gamma)$ (and hence $\Gamma$) is virtually nilpotent. We have reached a contradiction, therefore, $\xi_{f}$ is non-constant and $\partial_{f}\Gamma$ contains at least two points. 

\par Now we conclude that $\Gamma$ has strong property (U) by showing that $|\partial_f\Gamma|\geq 3$.  If $|\partial_{f}\Gamma|=2$, consider the restriction \hbox{$\rho_{V}:\Gamma \rightarrow \mathsf{GL}(V)$} where $V=\langle \xi_{f}(\partial_{f}\Gamma) \rangle$ and $\dim(V)=2$. We show that all elements of $\rho(\textup{ker}(\rho_V))$ have all of their eigenvalues of modulus $1$. For this, since $\xi_{f}(\partial_{f}\Gamma)$ contains two points, up to passing to a finite-index subgroup of $\Gamma$ and conjugating $\rho_V$ by an element of $\mathsf{GL}(V)$, we may assume that $\rho_{V}(\Gamma)$ lies in the diagonal subgroup $\mathsf{GL}(V)$. Let $g \in \textup{ker}(\rho_{V})$. We may write $\rho(g^n)=w_{n}\exp(\mu(g^n))w_{n}'$ and assume, up to conjugating $\rho$, that, $\lim_n w_n=w_{\infty}, \lim_nw_n'=w_{\infty}'$, where $w_{\infty}P_{1}^{+}=P_{1}^{+}$. We see that $\lim_{n}\frac{\rho(g)^n}{||\rho(g)^n||}=w_{\infty}E_{11}w_{\infty}' \in \bigoplus_{i=1}^d\mathbb{R}E_{1i}$ and we may write for $n\in \mathbb{N}$, $$\rho(g^n)=\begin{pmatrix}
I_2 & \big(\sum_{i=0}^n A^{i} \big)^t B \\ 
 0 & A^n
\end{pmatrix}$$ such that $\lim_{n}\frac{1}{||\rho(g^n)||}A^n$ is the zero matrix. If $A$ has an eigenvalue of modulus greater than $1$, then $\ell_1(A)=\ell_1(\rho(g))$. By working similarly as in the previous case, we have $\lim_{n}\frac{1}{||\rho(g^n)||}\sum_{i=0}^{n}||A^{i}||=0$ and $\lim_{n}\frac{1}{||\rho(g^n)||}\rho(g^n)$ has all of its $(1,i)$ entries equal to zero, which is absurd. This shows that $\rho(g^{\pm 1})$ has all of its eigenvalues of modulus at most $1$ for $g\in \textup{ker}(\rho_V)$.

Similarly as in the previous case, we deduce that $\rho(\textup{ker}(\rho_{V}))$ (and hence $\textup{ker}(\rho_V)$) is virtually nilpotent and finitely generated. The quotient $\Gamma/\textup{ker}(\rho_V)$ is abelian, so $\Gamma$ has to be virtually polycyclic. Since $|\partial_{f}\Gamma|>1$, a theorem of Floyd \cite[p. 211]{Floyd} implies that $\Gamma$ has two ends, so $\Gamma$ is virtually cyclic. Since $\Gamma$ is assumed not to be virtually nilpotent, this is again a contradiction, hence $\partial_f\Gamma$ cannot contain two points. 

It follows that $|\partial_{f}\Gamma|\geq 3$. Therefore, Theorem \ref{nontrivial} shows that $\Gamma$ satisfies strong property (U). \end{proof}

\begin{proof}[Proof of Theorem \ref{wg}.] Suppose that \textup{(i)} holds, i.e $\rho$ is $P_i$-Anosov. Then \textup{(ii)} holds since the Floyd boundary identifies with the Gromov boundary of $\Gamma$. Moreover, by Theorem \ref{mainproperties} and Proposition \ref{semisimplification}, \textup{(iii)} and \textup{(iv)} hold true for any semisimplification $\rho^{ss}$ of the $P_i$-Anosov representation $\rho$. Now let us prove the other implications.
 We assume that there exists $\varepsilon >0$ such that  for every $\gamma \in \Gamma$, $$\log \frac{\ell_i(\rho(\gamma))}{\ell_{i+1}(\rho(\gamma))}  \geqslant \varepsilon |\gamma |_{\infty}.$$ By \cite[Prop. 4.12]{kassel-potrie} it is enough to prove that $\Gamma$ satisfies strong property (U).
\medskip

\noindent $\textup{(ii)} \Rightarrow \textup{(i)}$. We first observe that for every element $g \in \textup{ker}(\rho)$ we have $|g|_{\infty}=0$. We next show that $N:=\textup{ker}\rho$ is finite. If not, $N$ is an infinite normal subgroup of $\Gamma$ and $\Lambda(N)=\partial_{f}\Gamma$ since $\Gamma$ acts minimally on $\partial_{f}\Gamma$. By \cite[Thm. 1]{Karlsson} there exists a non cyclic free subgroup $H$ of $N$ with $|\Lambda(H)|\geq3$. In particular, by Theorem \ref{nontrivial} we can find $\gamma \in H$ such that $|\gamma|_{\infty}>0$. This is a contradiction since $\gamma \in N$. It follows that $N$ is finite. \par The Floyd boundary of $\Gamma'=\Gamma/N$ is non-trivial since $\Gamma'$ is quasi-isometric to $\Gamma$. Note that the representation $\rho$ induces a faithful representation $\rho':\Gamma'\rightarrow \mathsf{GL}_d(\mathbb{R})$ which also has a weak uniform $i$-gap in eigenvalues. Selberg's lemma \cite{Selberg} implies that $\Gamma'$ is virtually torsion-free, thus, by Theorem  \ref{nontrivial}, $\Gamma'$ satisfies strong property (U). We conclude that $\Gamma'$ and $\Gamma$ are word hyperbolic and $\rho$ is $P_i$-Anosov.
\medskip

\noindent $\textup{(iii)} \Rightarrow \textup{(i)}$. If $\Gamma$ is virtually nilpotent, Lemma \ref{nil1} implies that $\Gamma$ is virtually cyclic, contradicting our assumption. Since $\Gamma$ is not virtually nilpotent and $\rho_1$ satisfies the uniform gap summation property, $\Gamma$ has to satisfy strong property (U) by Corollary \ref{ugspu}. Therefore, \textup{(i)} holds. 

\medskip
\noindent $\textup{(iv)} \Rightarrow \textup{(i)}$. Let $\rho^{ss}$ be a semisimplification of $\rho$. By Proposition \ref{semisimplification}, $\lambda(\rho(g))=\lambda(\rho^{ss}(g))$ for every $g \in \Gamma$, hence there exists $c_{2}>0$, depending only on $\rho_2$, such that $$\log \frac{\ell_i(\rho^{ss}(\gamma))}{\ell_{i+1}(\rho^{ss}(\gamma))} \geqslant \varepsilon |\gamma |_{\infty} \geqslant \varepsilon c_2 \big| \big|\lambda(\rho_2(\gamma))\big| \big|$$ for every $\gamma \in \Gamma$. By Theorem \ref{finitesubset} there exists a finite subset $F$ of $\Gamma$ and $C>0$ such that for every $\gamma \in \Gamma$ there is $w \in F$ with $$\big| \big| \mu(\rho^{ss}(\gamma))-\lambda(\rho(\gamma w)) \big| \big| \leqslant C, \  \big| \big| \mu(\rho_2(\gamma))-\lambda(\rho_2(\gamma w)) \big| \big| \leqslant C.$$ In particular, we may choose $R>0$ such that  for every $\gamma \in \Gamma$, $$\log \frac{\sigma_i(\rho^{ss}(\gamma))}{\sigma_{i+1}(\rho^{ss}(\gamma))} \geqslant \varepsilon c_2  \big| \big|\mu(\rho_2(\gamma))\big| \big|-R.$$ By assumption, for all but finitely many $\gamma \in \Gamma$ we have $$\big| \big| \mu(\rho_2(\gamma))\big|\big| \geqslant \frac{2}{\varepsilon c_2}  \log|\gamma|_{\Gamma},$$ so there exists $R'>0$ such that $$\log \frac{\sigma_i(\rho^{ss}(\gamma))}{\sigma_{i+1}(\rho^{ss}(\gamma))} \geqslant 2 \log |\gamma |_{\Gamma}-R'$$ for all $\gamma \in \Gamma$ non-trivial. In particular, the semisimplification $\rho^{ss}$ of $\rho$ satisfies the uniform gap summation property. Therefore, since $\rho^{ss}$ has a weak uniform $i$-gap in eigenvalues, by implication $\textup{(iii)} \Rightarrow \textup{(i)}$, $\rho^{ss}$ is $P_i$-Anosov and $\Gamma$ is word hyperbolic. In particular, $\rho$ is $P_i$-Anosov. \end{proof}

\section{Gromov products} \label{Gromovproduct}
In this section, we recall the definition of the Gromov product (see Definition \ref{defproduct}) associated to an Anosov representation and prove Proposition \ref{Gromovproduct1}, and we show that it is comparable with the Gromov product on the domain hyperbolic group with respect to a fix word metric.

\begin{definition} Let $G$ be a real semisimple Lie group. For every linear form $\varphi \in \mathfrak{a}^{\ast}$, define the Gromov product relative to $\varphi$ to be the map $(\, \cdot \,)_{\varphi}:G\times G \rightarrow \mathbb{R}$ defined as follows: for $g,h \in G$, $$\big(g\cdot h)_{\varphi}:=\frac{1}{4}\varphi\Big(\mu(g)+\mu(g^{-1})+\mu(h)+\mu(h^{-1})-\mu(g^{-1}h)-\mu(h^{-1}g) \Big).$$ \end{definition} 

\par For a line $\ell \in \mathbb{P}(\mathbb{R}^d)$ and a hyperplane $V \in \mathsf{Gr}_{d-1}(\mathbb{R}^d)$, the distance $\textup{dist}(\ell, V)$ is computed by the formula $$\textup{dist}(\ell,V)=\big | \big \langle k_{\ell}e_1, k_{V}e_d \big \rangle \big|,$$ where $\ell=[k_{\ell}e_1]$, $V=[k_{V}e_{d}^{\perp}]$, $k_V, k_{\ell} \in \mathsf{O}(d)$ and $\langle \cdot,\cdot\rangle$ is the standard inner product. Recall that a representation $\rho:\Gamma \rightarrow \mathsf{PGL}_d(\mathbb{R})$ is called $P_1$-divergent if $\lim \frac{\sigma_1(\rho(\gamma_n))}{\sigma_2(\rho(\gamma_n))}=\infty$ as $|\gamma|_{\Gamma}\rightarrow \infty$.

 The following proposition relates the Gromov product with the limit maps of a representation $\rho$ and will be used in the following sections.

\begin{proposition}\label{calc} Let $\Gamma$ be a word hyperbolic group and $\rho:\Gamma \rightarrow \mathsf{PGL}_d(\mathbb{R})$ a representation. Suppose $\rho$ is $P_1$-divergent and there are continuous $\rho$-equivariant maps $\xi:\partial_{\infty}\Gamma \rightarrow \mathbb{P}(\mathbb{R}^d)$ and $\xi^{-}:\partial_{\infty}\Gamma \rightarrow \mathsf{Gr}_{d-1}(\mathbb{R}^d)$ satisfying the Cartan property. Then for $x,y \in \partial_{\infty}\Gamma$ and two sequences $(\gamma_n)_{n \in \mathbb{N}}$, $(\delta_n)_{n \in \mathbb{N}}$ of elements of $\Gamma$ with $\lim_{n}\gamma_n=x$ and $\lim_{n}\delta_n=y$ we have $$\lim_{n \rightarrow \infty}\exp\Big(-4\big(\rho(\gamma_n)\cdot \rho(\delta_n)\big)_{\varepsilon_1}\Big)=\textup{dist}\big(\xi(x),\xi^{-}(y)\big)\cdot \textup{dist}\big(\xi(y),\xi^{-}(x)\big).$$ \end{proposition}

\begin{proof} We may write $\rho(\gamma_{n})=w_{n}\exp(\mu(\rho(\gamma_{n})))w_{n}'$ and $\rho(\delta_{n})=k_n\exp(\mu(\rho(\delta_{n})))k_n'$ where $w_{n}, w_{n}',$ $ k_n,k_n' \in \mathsf{PO}(d)$. Since $\rho$ is $P_1$-divergent, $\lim_{n}\frac{\sigma_{d}(\rho(\gamma_n))}{\sigma_{j}(\rho(\gamma_n))}=\lim_{n}\frac{\sigma_{d}(\rho(\delta_n))}{\sigma_{j}(\rho(\delta_n))}=0$ for $1 \leq j \leq d-1$. Recall that $E_{ij}$ denotes the $d\times d$ elementary matrix with $1$ on the $(i,j)$-entry. \hbox{Then we notice that} $$\lim_{n \rightarrow \infty}\exp\Big(-4\Big(\rho(\gamma_{n}) \cdot \rho(\delta_{n}) \Big)_{\varepsilon_1} \Big)= \lim_{n \rightarrow \infty}\frac{\sigma_1(\rho(\gamma_{n}^{-1}\delta_{n}))\sigma_1(\rho(\delta_{n}^{-1}\gamma_{n}))}{\sigma_1(\rho(\gamma_{n}))\sigma_1(\rho(\gamma_{n}^{-1}))\sigma_1(\rho(\delta_{n}))\sigma_1(\rho(\delta_{n}^{-1}))}$$ \begin{align*}&=\lim_{n \rightarrow \infty} \Bigg( \Bigg|\Bigg|(k_n')^{-1} \textup{diag}\Bigg(\frac{\sigma_d(\rho(\delta_{n}))}{\sigma_1(\rho(\delta_{n}))},\ldots,1 \Bigg)k_n^{-1}w_{n} \textup{diag}\Bigg(1,\ldots,\frac{\sigma_d(\rho(\gamma_{n}))}{\sigma_1(\rho(\gamma_{n}))}\Bigg)w_{n}' \Bigg|\Bigg| \cdot\\ &\Bigg|\Bigg|(w_{{n}}')^{-1} \textup{diag}\Bigg(\frac{\sigma_d(\rho(\gamma_{n}))}{\sigma_1(\rho(\gamma_{n}))},\ldots,1\Bigg)w_{{n}}^{-1}k_{{n}} \textup{diag}\Bigg(1,\ldots,\frac{\sigma_d(\rho(\delta_{n}))}{\sigma_1(\rho(\delta_{n}))}\Bigg)k_{{n}}' \Bigg|\Bigg| \Bigg)\\ & =\lim_{n \rightarrow \infty} \big|\big| E_{1d}w_{n}^{-1}k_{n}E_{11}\big|\big| \cdot \big|\big| E_{1d}k_{n}^{-1}w_{n}E_{11}\big|\big| \\ & =\lim_{n \rightarrow \infty}\big|\langle w_{{n}}^{-1}k_{{n}}e_1,e_d \rangle \cdot \langle k_{n}^{-1}w_{n}e_1,e_d \rangle \big|\\ &=\lim_{n \rightarrow \infty} \textup{dist}\big(\Xi_{1}^{+}\big(\rho(\gamma_{{n}})\big),\Xi_{1}^{-}\big(\rho(\delta_{{n}})\big) \big) \cdot  \textup{dist}\big(\Xi_{1}^{+}\big(\rho(\delta_{{n}})\big),\Xi_{1}^{-}\big(\rho(\gamma_{n})\big) \big)\\ &= \textup{dist}\big(\xi(x),\xi^{-}(y)\big)\cdot \textup{dist}\big(\xi(y),\xi^{-}(x)\big),\end{align*} since $\xi$ and $\xi^{-}$ satisfy the Cartan property. This finishes the proof of the propositon. \end{proof}

\begin{proof}[Proof of Proposition \ref{Gromovproduct1}] Fix $\alpha \in \theta$. By \cite[Thm. 7.2]{Tits}, there exists $N_{\alpha}>0$ and an irreducible $\theta$-proximal representation $\tau_{\alpha}:G \rightarrow \mathsf{GL}_d(\mathbb{R})$ whose highest weight is $N_{\alpha}\omega_{\alpha}$, $N_{\alpha}\in \mathbb{N}$. Since $\rho$ is $P_{\{\alpha \}}$-Anosov, the representation $\tau_{\alpha} \circ \rho$ is $P_1$-Anosov. There exists $C_1>0$, depending only on $\tau_{\alpha}$, such that $$\Big|\log \sigma_1(\tau_{\alpha}(g))-N_{\alpha} \omega_{\alpha}(\mu(g))\Big|\leq C_1$$ for every $g \in G$. In particular, there exists $C_2>0$, depending only on $\tau$, such that \begin{align}\label{product-1} \Big|N_{\alpha}\big(\rho(\gamma)\cdot \rho(\delta) \big)_{\omega_{\alpha}}-\big(\tau_{\alpha}(\rho(\gamma))\cdot \tau_{\alpha}(\rho(\delta)) \big)_{\varepsilon_1}\Big| \leq C_2\end{align} for every $\gamma, \delta\in \Gamma$. Since $\rho$ is $P_{\{\alpha\}}$-Anosov, by Lemma \ref{weights}, we may replace $\rho$ with a semisimplification $\rho^{ss}$ such that there exists $C_3>0$ with $$\Big|\big(\rho(\gamma))\cdot \rho(\delta)\big)_{\omega_{\alpha}}-\big(\rho^{ss}(\gamma))\cdot \rho^{ss}(\delta))\big)_{\omega_{\alpha}}\Big|\leq C_3$$ for every $\gamma, \delta \in \Gamma$. Therefore, we may continue by assuming that $\rho$ is semisimple. By using Lemma \ref{replace}, we may further assume that $\tau_{\alpha} \big(\rho(\Gamma)\big)$ has reductive Zariski closure in $\mathsf{GL}_{d}(\mathbb{R})$ and preserves a properly convex open domain $\Omega$ of $\mathbb{P}(\mathbb{R}^d)$. Let us fix $x_0 \in \Omega$. By Lemma \ref{control} we can find $C_4>0$ such that for every $\gamma,\delta \in \Gamma$, \begin{align}\label{product-2}\Big| \big(\tau_{\alpha}(\rho(\gamma))\cdot \tau_{\alpha}(\rho(\delta))\big)_{\varepsilon_1} -\big(\tau_{\alpha}\big(\rho(\gamma)\big)x_0\cdot \tau_{\alpha}\big(\rho(\delta)\big)x_0\big)_{x_0} \Big| \leqslant C_4.\end{align} By \cite{DGK0} and \cite{Zimmer}, since $\tau_{\alpha} \circ \rho$ is $P_1$-Anosov, $\tau_{\alpha}\big(\rho(\Gamma)\big)$ acts cocompactly on a closed convex subset $\mathcal{C} \subset \Omega$. Fix $x_0 \in \mathcal{C}$. The Svarc--Milnor lemma implies that the orbit map $\gamma \mapsto \tau_{\alpha}(\rho(\gamma))x_0$ is a quasi-isometry between the Gromov hyperbolic spaces $(\Gamma,d_{\Gamma})$ and $(\mathcal{C},d_{\Omega})$. In particular, there exist $C_5,c_5>0$ such that for every $\gamma,\delta \in \Gamma$, \begin{align}\label{product-3}C_{5}^{-1}(\gamma \cdot \delta)_{e}-c_5 \leqslant \big(\tau_{\alpha}\big(\rho(\gamma)\big)x_0\cdot \tau_{\alpha}\big(\rho(\delta)\big)x_0\big)_{x_0} \leqslant C_5 (\gamma \cdot \delta)_{e}+c_5.\end{align} Therefore, by (\ref{product-1}), (\ref{product-2}) and (\ref{product-3}) we obtain the conclusion.\end{proof}

\section{Characterizations of Anosov representations} \label{Main}
This section is devoted to the proof of Theorems \ref{maintheorem} and \ref{Zariskidense} and Corollary \ref{CCartan}. Note that in Theorem \ref{maintheorem} we do not assume that the group $\rho(\Gamma)$ contains a $P_{\theta}$-proximal element, the pair of limit maps $(\xi^{+},\xi^{-})$ is compatible or the map $\xi^{-}$ satisfies the Cartan property.

\begin{proof}[Proof of Theorem \ref{maintheorem}] If $\rho$ is $P_{\theta}$-Anosov, the Anosov limit maps of $\rho$ are transverse and dynamics preserving and $\rho$ is $P_{\theta}$-divergent (see Theorem \ref{mainproperties}). Also, the fact that the Anosov limit maps satisfy the Cartan property is contained in \cite[Thm. 1.3 (4) \& 5.3 (4)]{GGKW}.
\par Now we assume that $\rho$ satisfies \textup{(i) and (ii)}. We first reduce to the case where $\Gamma$ is torsion-free. Since $\rho$ is $P_{\theta}$-divergent, every element of the kernel $\textup{ker}(\rho)$ has finite order, hence $\textup{ker}(\rho)$ is finite. The quotient group $\Gamma_1=\Gamma /\textup{ker}(\rho)$ is quasi-isometric to $\Gamma$ and by Selberg's lemma \cite{Selberg} $\Gamma_1$ contains a torsion-free and finite-index subgroup $\Gamma_2$. It is enough to prove that the induced representation $\hat{\rho}:\Gamma_2 \rightarrow G$ is $P_{\theta}$-Anosov. Notice that $\hat{\rho}$ satisfies the same assumptions as $\rho$ and the source group is torsion-free.
\par Thanks to Proposition \ref{higherdimension}, we may assume that $G=\mathsf{SL}_d(\mathbb{R})$, $\theta=\{ \varepsilon_1-\varepsilon_2 \}$, $P_{\theta}^{+}=\textup{Stab}_G(\mathbb{R}e_1)$ and \hbox{$P_{\theta}^{-}=\textup{Stab}_G(e_1^{\perp})$}. Recall the definition of the bundle $\mathcal{X}_{\rho}$ over the flow space $\Gamma \backslash \hat{\Gamma}$ as in sub-section \ref{Anosovdfn}. The pair of transverse maps $(\xi^{+},\xi^{-})$ defines the section $\sigma:\Gamma \backslash \hat{\Gamma} \rightarrow \mathcal{X}_{\rho}$,
$$\sigma([\hat{m}]_{\Gamma})=\big[\hat{m}, (\xi^{+}(\tau^{+}(\hat{m})), \xi^{-}(\tau^{-}(\hat{m})))\big]_{\Gamma}$$ inducing the splitting \hbox{$\sigma_{\ast}\mathcal{E}=\sigma_{\ast}\mathcal{E}^{+} \oplus \sigma_{\ast}\mathcal{E}^{-}$}, where $\mathcal{E}^{\pm}\subset \mathsf{T}(G/L_{\theta})$ are the sub-bundles defined in subsection \ref{Anosovdfn}. Then we fix $x=[\hat{m}]_{\Gamma}$ and choose an element $h\in G$ so that $\xi^{+}(\tau^{+}(\hat{m}))=hP_1^{+}$ and $\xi^{-}(\tau^{-}(\hat{m}))=hP_1^{-}$. Let $(t_n)_{n \in \mathbb{N}}$ be an increasing unbounded sequence and consider a sequence $(\gamma_n)_{n \in \mathbb{N}}$ of elements of $\Gamma$ such that $(\gamma_n \varphi_{t_n}(\hat{m}))_{n \in \mathbb{N}}$ lies in a compact subset of $\hat{\Gamma}$. We observe that $\lim_{n}\gamma_{n}^{-1}=\tau^{+}(\hat{m})$ in the bordification $\Gamma \cup \partial_{\infty}\Gamma$. Moreover, observe that we can write $\rho(\gamma_{n}^{-1})=(k_{n}')^{-1}w \exp\big(\mu(\rho(\gamma_{n}^{-1}))\big)wk_{n}^{-1}$, where \hbox{$w=\sum_{i=1}^{d}E_{i(d+1-i)} \in \mathsf{O}(d)$.} Since $\xi^{+}$ is assumed to satisfy the Cartan property and $(\gamma_n)_{n \in \mathbb{N}}$ is $P_{\theta}$-divergent, up to subsequence, we may assume that  $\lim_{n}\Xi_{1}^{+}\big(\rho(\gamma_{n}^{-1})\big)=\lim_{n} (k_{{n}}')^{-1}wP_{\theta}^{+}=hP_{\theta}^{+}$. Equivalently, if $k'=\lim_{n}k_{n}'$ then $k'h=w\begin{pmatrix}[0.7]
s & \ast\\ 
0 & B
\end{pmatrix}$ for some $B\in \mathsf{GL}_{d-1}(\mathbb{R})$. Fix $u \in \{0\} \times \mathbb{R}^{d-1}$. Then, since $k'(k')^{t}=I_{d}$, we observe $$k'h^{-t}u=w_{d-1}B^{-t}u +0e_{d},\ k'h^{-t}e_1=\frac{1}{s}e_{d}+\sum_{i=1}^{d-1} \zeta_{i}e_{i}$$ for some $s \neq 0$, $\zeta_{1},\ldots,\zeta_{d-1}\in \mathbb{R}$ and $w_{d-1} \in \mathsf{O}(d-1)$ is a permutation matrix with $w_{d-1}e_1=e_{d-1}$ and $w_{d-1}e_{d-1}=e_1$. \hbox{Equivalently, we write:}
$$ k_{n}'h^{-t}u=\sum_{i=1}^{d}\chi_{i,n}e_{i}, \ k_{n}'h^{-t}e_1=\sum_{i=1}^{d} \zeta_{i,n}e_{i} $$ and we have that $\lim_{n}\chi_{d,n} =0$, $\lim_{n}\zeta_{d,n}=\frac{1}{s}$. A computation shows that \begin{align*} \frac{\left \| \rho^{\ast}(\gamma_{n})h^{-t}u \right \|^2}{\left \| \rho^{\ast}(\gamma_{n})h^{-t}e_1 \right \|^2}&=\frac{\sum_{i=1}^{d}\chi_{i,n} \sigma_i(\rho(\gamma_n))^{-2}}{\sum_{i=1}^{d}\zeta_{i,n}^2\sigma_i(\rho(\gamma_n))^{-2}}=\frac{\sum_{i=1}^{d-1} \chi_{i,n}^2 \frac{\sigma_d(\rho(\gamma_n))^2}{\sigma_i(\rho(\gamma_n))^2}+\chi_{d,n}^2}{\sum_{i=1}^{d-1}\zeta_{i,n}^2\frac{\sigma_d(\rho(\gamma_n))^2}{\sigma_i(\rho(\gamma_n))^2}+\zeta_{d,n}^2}. \end{align*} We deduce that $\lim_{n}\frac{||\rho(\gamma_{n})^{\ast}h^{-t}u||}{||\rho(\gamma_{n})^{\ast}h^{-t}e_1||}=0$ and hence by Proposition \ref{contractiondilation2} \textup({ii)} we conclude that $$\lim_{n \rightarrow \infty} \big|\big|\varphi_{t_{n}}(X_{u}^{-})\big|\big|_{\varphi_{t_{n}}(x)}=0.$$ The sequence we started with was arbitrary, therefore the (lift of the) geodesic flow (see Def. \ref{Def-of-Anosov-rep}) on $\sigma_{\ast}\mathcal{E}^{-}$ is weakly contracting. By Lemma \ref{contractiondilation} we conclude that the flow on $\sigma_{\ast}\mathcal{E}^{+}$ is weakly dilating. The compactness of $\Gamma \backslash \hat{\Gamma}$ implies that the geodesic flow on $\sigma_{\ast}\mathcal{E}^{+}$ (resp. $\sigma_{\ast}\mathcal{E}^{-}$) is uniformly dilating (resp. contracting). Finally, we conclude that $\rho$ is $P_{\theta}$-Anosov with Anosov limit maps $\xi^{+}$ and $\xi^{-}$.\end{proof}

\begin{proof}[Proof of Corollary \ref{CCartan}] Assume that conditions (i) and (ii) hold. Let $\tau_{\theta}: G \rightarrow \mathsf{GL}_d(\mathbb{R})$ be an irreducible and $\theta$-proximal representation as in Proposition \ref{higherdimension}. In order to show that $\rho$ is $\theta$-Anosov, it suffices to check that $\rho'=\tau_{\theta} \circ \rho$ is $P_1$-Anosov. By using \cite[Thm. 5.3 (1)]{GGKW} (see also Lemma \ref{existence}), there exists a pair of continuous, $\rho'$-equivariant maps $\xi^{+}:\partial_{\infty}\Gamma \rightarrow \mathbb{P}(\mathbb{R}^d)$ and $\xi^{-}:\partial_{\infty}\Gamma \rightarrow \mathsf{Gr}_{d-1}(\mathbb{R}^d)$ satisfying the Cartan property. Let $x,y \in \partial_{\infty}\Gamma$ be two distinct points and $(\gamma_n)_{n \in \mathbb{N}}$ a sequence of elements of $\Gamma$ with $x=\lim_{n}\gamma_{n} $ and $y=\lim_{n} \gamma_{n}^{-1}$. \hbox{Condition (ii), shows that} $$\sup_{n \in \mathbb{N}}\Big(2 \log \sigma_1(\rho'(\gamma_n))-\log \sigma_1(\rho'(\gamma_n^2)) \Big)<+\infty.$$ By Proposition \ref{calc} we have that $\textup{dist}(\xi^{+}(x),\xi^{-}(y))\cdot \textup{dist}(\xi^{+}(y),\xi^{-}(y))>0$ so the pair $(\xi^{+}(x),\xi^{-}(y))$ is transverse. The maps $\xi^{+}$ and $\xi^{-}$ are transverse, $\rho'$ is $P_1$-divergent by \textup{(i)}, hence, it follows by Theorem \ref{maintheorem} that $\rho'$ is $P_1$-Anosov.

\par Conversely, part \textup{(i)} follows immediately by  Theorem \ref{mainproperties} (i). Note that there is $N_{\alpha}\geq 1$ such that $N_{\alpha}\omega_{\alpha}$ is the highest weight of an irreducible proximal representation $\tau_{\alpha}:G \rightarrow \mathsf{GL}_d(\mathbb{R})$. There is a constant $C_0>0$, depending only on $\tau_{\alpha}$ such that $\big|N_{\alpha} \omega_{\alpha}( \mu(h) )-\log \sigma_1 (\tau_{\alpha}(h))\big|\leq C_0$ for every $h\in G$. By Proposition \ref{Gromovproduct1} (i) and using the fact that  for every $h\in G$, $\omega_{\alpha}(2N_{\alpha}\mu(h)-N_{\alpha}\mu(h^2))\geq 2\log\sigma_1(\tau_{\alpha}(h))-\log\sigma_1(\tau_{\alpha}(h^2))-3C_0\geq -3C_0$, we can find $B,b>0$ such that for every $\alpha \in \theta$ and $\gamma \in \Gamma$ we have \begin{align*} \omega_{\alpha}\big( 2\mu (\rho(\gamma))-\mu (\rho(\gamma^2))\big) &\leq 3C_0 N_{\alpha}^{-1}+  \omega_{\alpha}\big(2\mu (\rho(\gamma))+2\mu (\rho(\gamma^{-1}))-\mu (\rho(\gamma^2))-\mu (\rho(\gamma^{-2})) \big)\\ &  \leqslant B(\gamma \cdot \gamma^{-1})_{e}+b.\end{align*} This concludes the proof of the corollary. \end{proof}

Let $\Gamma$ be a word hyperbolic group and $H$ be a subgroup of $\Gamma$. The group $H$ is {\em quasiconvex} in $\Gamma$ if and only if $H$ is finitely generated and quasi-isometrically embedded in $\Gamma$. In this case, there exists a continuous injective $H$-equivariant map $\iota_{H}:\partial_{\infty}H \xhookrightarrow{} \partial_{\infty}\Gamma$ called the \emph{Cannon-Thurston map} extending the inclusion $H \xhookrightarrow{} \Gamma$. 

\begin{proof}[Proof of Theorem \ref{Zariskidense}] Corollary \ref{Cartan} shows that the representation $\rho$ is $P_{\theta}$-divergent and $\xi^{+}$ satisfies the Cartan property. Since $\iota_{H}$ is an $H$-equivariant embedding, the map $\xi^{+}\circ \iota_{H}$ also satisfies the Cartan property. Theorem \ref{maintheorem} shows that the representation $\rho|_{H}$ is $P_{\theta}$-Anosov. \end{proof}

Example \ref{nontransverse} provides a Zariski dense surface group representation \hbox{$\rho_1:\pi_1(S_g) \rightarrow \mathsf{PSL}_4(\mathbb{R})$} which is not $P_1$-Anosov and admits a pair of continuous $\rho_1$-equivariant maps $(\xi^{+}, \xi^{-})$. The representation $\rho_1$ is $P_1$-divergent and $\rho_1(\gamma)$ is $P_1$-proximal for every $\gamma \in \pi_1(S_g)$ non-trivial. However, for every finitely generated free subgroup $F$ of $\pi_1(S_g)$, the maps $\xi^{+}\circ \iota_F$ and $\xi^{-}\circ \iota_F$ are transverse and $\rho_1 |_{F}$ is $P_1$-Anosov.

\section{Strongly convex cocompact subgroups of $ \mathsf{PGL}_d(\mathbb{R})$} \label{Scc}
In this section, we prove Theorem \ref{stronglyconveccocompact}. For our proof we need the following proposition characterizing $P_1$-Anosov representations in terms of the Gromov product under the assumption that the group preserves a properly convex domain with strictly convex and $C^1$-boundary.

\begin{proposition} \label{main} Let $\Gamma$ be a word hyperbolic subgroup of $\mathsf{PGL}_d(\mathbb{R})$ which preserves a strictly convex domain $\Omega$ of $\mathbb{P}(\mathbb{R}^d)$ with $C^1$-boundary. Then the following are equivalent.\\
\noindent $\textup{(i)}$ The natural inclusion $\Gamma \xhookrightarrow{} \mathsf{PGL}_d(\mathbb{R})$ is $P_1$-Anosov.\\
\noindent $\textup{(ii)}$ There exist constants $J,k>0$ such that for every $\gamma, \delta \in \Gamma$, $$J^{-1}(\gamma \cdot \delta)_e-k \leqslant (\gamma \cdot \delta)_{\varepsilon_1} \leqslant J(\gamma \cdot \delta)_e+k.$$\end{proposition}

\begin{proof} \textup{(ii)} $\Rightarrow$ \textup{(i)}. We observe that $\Gamma$ is a discrete subgroup of $\mathsf{PGL}_d(\mathbb{R})$. Let $(\gamma_n)_{n \in \mathbb{N}}$ be an infinite sequence of elements of $\Gamma$ and $x_0 \in \Omega$. We may pass to a subsequence such that $\lim_{n}\gamma_{k_n}x_0 \in \partial \Omega$ exists. Since $\partial \Omega$ is strictly convex we conclude that $\lim_{n}\gamma_{k_n}x_0$ is independent of the basepoint $x_0$. Therefore, as in \cite[Lem. 7.5]{DGK0} or Lemma \ref{mainlemma}, we conclude that $\lim_{n}\frac{\sigma_2}{\sigma_1}(\gamma_{k_n})=0$ and $\Gamma$ has to be $P_1$-divergent.
\par Now let $(\gamma_n)_{n \in \mathbb{N}}, (\delta_n)_{n \in \mathbb{N}}$ be two sequences of elements of $\Gamma$ converging to $x \in \partial_{\infty}\Gamma$. We claim that the limits $\lim_{n}\gamma_n x_0$, $\lim_{n}\delta_n x_0$ exist and are equal. Note that the limits will be independent of the choice of $x_0$. We may write $$\gamma_{n}=w_{\gamma_n}\exp(\mu(\gamma_{n}))w_{\gamma_n}' \ \ \ \textup{and} \ \ \ \delta_{n}=w_{\delta_n}\exp(\mu(\delta_{n}))w_{\delta_n}' $$ where $w_{\gamma_n}, w_{\gamma_n}',w_{\delta_n}, w_{\delta_n}' \in \mathsf{PO}(d)$. Since $\Gamma$ is $P_1$-divergent, there exist subsequences $(\gamma_{k_{n}})_{n \in \mathbb{N}}$, $(\delta_{s_{n}})_{n \in \mathbb{N}}$ such that $a_{1}=\lim_{n}\gamma_{k_{n}}x_0=\lim_{n}\Xi_{1}^{+}(\gamma_{k_n})$, $a_2=\lim_{n} \delta_{s_n}x_0=\lim_{n}\Xi_{1}^{+}(\delta_{s_n})$, $\lim_{n}\Xi_{1}^{-}(\gamma_{k_n})\\=a_{1}^{-}$ and $\lim_{n}\Xi_{1}^{-}(\delta_{s_n})=a_{2}^{-}$, where $\Xi_{1}^{+}(\gamma_{k_{n}})=[w_{\gamma_{k_n}}e_1]$ and $\Xi_{1}^{-}(\gamma_{k_{n}})=[w_{\gamma_{k_n}}e_{d}^{\perp}]$. Proposition \ref{calc} and the fact that $(\gamma_{k_{n}}\cdot \delta_{s_{n}})_{\varepsilon_1} \rightarrow +\infty$ show $$\lim_{n \rightarrow \infty} \textup{dist}\big(\Xi_{1}^{+}(\gamma_{k_{n}}),\Xi_{1}^{-}(\delta_{s_{n}}) \big) \cdot \textup{dist}\big(\Xi_{1}^{+}(\delta_{s_{n}}),\Xi_{1}^{-}(\gamma_{k_{n}}) \big)=0$$ so either $a_1\in a_2^{-}$ or $a_2 \in a_1^{-}$. Using the same argument, we see that $$\lim_{n \rightarrow \infty} \textup{dist}\big(\Xi_{1}^{+}(\gamma_{k_{n}}),\Xi_{1}^{-}(\gamma_{k_{n}}) \big)=\lim_{n \rightarrow \infty} \textup{dist}\big(\Xi_{1}^{+}(\delta_{s_{n}}),\Xi_{1}^{-}(\delta_{s_{n}}) \big)=0$$ so $a_{i} \in a_{i}^{-}$ for $i=1,2$. In each case, the previous calculation shows that $a_1,a_2 \in a_{1}^{-}$ or $a_1,a_2 \in a_{2}^{-}$. Without loss of generality, assume that $a_2 \in a_{1}^{-}$, so the projective line segment $[a_1,a_2]$ is contained in the projective hyperplane $a_{1}^{-}$ and $\overline{\Omega}$. Since $\Gamma$ is $P_1$-divergent, there exist $x_{0}^{\ast} \in \Omega^{\ast}$ such that $\lim_{n}\Xi_{1}^{-}(\gamma_{k_{n}})=\lim_{n}\gamma_{k_n}x_{0}^{\ast}$ and $a_{1}^{-} \in \partial \Omega^{\ast}$. Therefore, $a_{1}^{-}$ avoids $\Omega$. We conclude that $[a_1,a_2]$ is contained in $\partial \Omega$ and $a_1=a_2$. 
\par The previous discussion shows that for any two sequences of $(\gamma_n)_{n \in \mathbb{N}}$ and $(\delta_n)_{n \in \mathbb{N}}$ converging to $x \in \partial_{\infty}\Gamma$ the limits $\lim_{n}\gamma_nx_0$ and $\lim_{n}\delta_nx_0$ exist and are equal. We obtain a $\Gamma$-equivariant map $\xi:\partial_{\infty}\Gamma \rightarrow \mathbb{P}(\mathbb{R}^d)$ defined by the formula $\xi(\lim_{n}\gamma_n)=\lim_{n}\gamma_n x_0$. Let $x=\lim_{n}\delta_n$ and suppose $\lim_{n}x_n=x$ in $\partial_{\infty}\Gamma$. We may write $x_{n}=\lim_{m}\gamma_{n,m}$. For every $n\in \mathbb{N}$ there are $k_{n},m_n \in \mathbb{N}$, such that $(\gamma_{n,k_{n}} \cdot \delta_{m_n})_{e}>n$ and $d_{\mathbb{P}}\big(\gamma_{n,k_{n}}x_0,\xi(x_{n}) \big) \leqslant \frac{1}{n}$. Then, $\lim_{n}\gamma_{n,k_{n}}x_0$ exists and is equal to $\xi(x)=\lim_{n}\delta_{n}x_0$. It follows, that $\lim_{n}\xi(x_n)=\xi(x)$. So the map $\xi$ is continuous. By definition $\xi$ has the Cartan property.
\par The dual convex set $\Omega^{\ast}$ has strictly convex boundary since the boundary of $\Omega$ is of class $C^1$. By considering the standard identification of $\mathbb{P}((\mathbb{R}^d)^{\ast})$ with $\mathbb{P}(\mathbb{R}^d)$, we obtain a properly convex domain $\Omega'$ of $\mathbb{P}(\mathbb{R}^d)$ which is $\Gamma^{\ast}$-invariant and has strictly convex boundary. Since $(\gamma^{-t}\cdot \delta^{-t})_{\varepsilon_1}=(\gamma \cdot \delta)_{\varepsilon_1}$, we obtain a continuous \hbox{$\Gamma^{\ast}$-equivariant} limit map $\xi^{\ast}:\partial_{\infty}\Gamma \rightarrow \mathbb{P}(\mathbb{R}^d)$ satisfying the Cartan property. From $\xi^{\ast}$ we obtain a $\Gamma$-equivariant continuous map $\xi^{-}:\partial_{\infty}\Gamma \rightarrow \mathsf{Gr}_{d-1}(\mathbb{R}^d)$ as follows: if $\xi^{\ast}(x)=[k_{x}e_1]$ where $k_{x} \in \mathsf{PO}(d)$ then $\xi^{-}(x)=[k_{x}e_{1}^{\perp}]$. \par For two distinct boundary points $x,y \in \partial_{\infty}\Gamma$ denote by $(x\cdot y)_{e}$ their Gromov product. By definition, we may choose sequences $(\alpha_n)_{n \in  \mathbb{N}}, (\beta_n)_{n \in \mathbb{N}}$ in $\Gamma$ with $x=\lim_{n}\alpha_n$, $y=\lim_{n}\beta_n$ and $(x\cdot y)_{e}=\lim_{n}(\alpha_n\cdot \beta_n)_{e}$. By assumption we have that $\varliminf_n\big(\rho(\alpha_n)\cdot \rho(\beta_n)\big)_{e}\geq J^{-1}(x\cdot y)_{e}-k$ and hence by Proposition \ref{calc} we obtain the lower bound $$\textup{dist}\big(\xi(x),\xi^{-}(y)\big)\cdot \textup{dist}\big(\xi(y),\xi^{-}(x) \big) \geqslant e^{-4J (x\cdot y)_{e}-4k}>0.$$ Therefore, the pair of maps $(\xi,\xi^{-})$ is transverse. Finally, the inclusion $\Gamma \xhookrightarrow{} \mathsf{PGL}_d(\mathbb{R})$ is $P_1$-divergent, admits a pair $(\xi,\xi^{-})$  of $\Gamma$-equivariant, continuous transverse maps with the Cartan property, so Theorem \ref{maintheorem} shows that the inclusion $\Gamma  \xhookrightarrow{} \mathsf{PGL}_d(\mathbb{R})$ is $P_1$-Anosov. \par The converse is a direct consequence of Proposition \ref{Gromovproduct1}.   \end{proof}

\begin{proof}[Proof of Theorem \ref{stronglyconveccocompact}] The implication \textup{(i)} $\Rightarrow$ \textup{(ii)} follows immediately by the Svarc--Milnor lemma. Now assume that \textup{(ii)} holds. By \cite[Thm. 1.4]{DGK0} it is enough to prove that $\Gamma \xhookrightarrow{} \mathsf{PGL}_d(\mathbb{R})$ is $P_1$-Anosov. Let $x_{0} \in \mathcal{C}$. Lemma \ref{control} shows that  the orbit map $x_0 \mapsto \gamma x_0$ is a quasi-isometric embedding of $\Gamma$ into $(\mathcal{C},d_{\Omega})$, hence $\Gamma$ is word hyperbolic. By using Lemma \ref{control} we deduce that there exist constants $J,k>0$ such that for every $\gamma_1, \gamma_2 \in \Gamma$, $$J^{-1}(\gamma_1 \cdot \gamma_2)_{e}-k \leqslant \big(\rho(\gamma_1) \cdot \rho(\gamma_2)  \big)_{\varepsilon_1} \leqslant J(\gamma_1 \cdot \gamma_2)_{e}+k.$$ Proposition \ref{main} then finishes the proof.\end{proof}

\section{Distribution of singular values} \label{sub} Recall for $d\geq 2$, $(e_1,\ldots,e_d)$ denotes the canonical basis of $\mathbb{R}^d$. For $q \in \mathbb{N}$ consider $$\textup{Sym}^q\mathbb{R}^d:=\bigoplus_{k_1+\cdots+k_d=q}\mathbb{R}e_{1}^{k_1}e_2^{k_2}\cdots e_{d}^{k_d}$$ the symmetric power of $\mathbb{R}^d$. The $q$-symmetric power $\textup{sym}^{q}:\mathsf{GL}_d(\mathbb{R}) \rightarrow \mathsf{GL}(\textup{Sym}^{q}\mathbb{R}^d\big)$ is the representation defined as follows: for $g=(g_{ij})_{ij=1}^{n}\in \mathsf{GL}_d(\mathbb{R})$, define $\textup{sym}^q(g)(e_1^{k_1}\cdots e_{d}^{k_d}):=(ge_1)^{k_1}\cdots (ge_d)^{k_d}=\prod_{j=1}^{d}(\sum_{i}g_{ij}e_i)^{k_j}$ for any basis vector $e_1^{k_1}\cdots e_d^{k_d}$ of $\textup{Sym}^q\mathbb{R}^d$. 

\begin{remark}\normalfont{ For $q\in \mathbb{N}$, note that respect to the standard Cartan decomposition of $\mathsf{GL}(\textup{Sym}^q\mathbb{R}^d)$, for every $g\in \mathsf{GL}_d(\mathbb{R})$ we have that  $\sigma_{1}(\textup{sym}^q g)=(\sigma_{1}((g))^q$, $\ell_1(\textup{sym}^q g)=\ell_1(g)^q$ and $\sigma_2(\textup{sym}^2g)=\sigma_1(g)^{q-1}\sigma_2(g)$, $\ell_2(\textup{sym}^qg)=\ell_1(g)^{q-1}\ell_2(g)$. In particular, by the characterizations of Anosov representations in terms of singular value (resp. eigenvalue) gaps \cite{KLP2, BPS} (resp. \cite{kassel-potrie}), a representation $\rho:\Gamma \rightarrow \mathsf{GL}_d(\mathbb{R})$ is $P_1$-Anosov if and only if $\textup{sym}^q\rho:\Gamma \rightarrow \mathsf{GL}(\textup{Sym}^q \mathbb{R}^d))$ is $P_1$-Anosov.}\end{remark}

\par  By using Theorem \ref{maintheorem} we exhibit conditions guaranteeing that the product of two linear representations of a hyperbolic group is $P_1$-Anosov. 

\begin{theorem} \label{directsum}  Let $\Gamma$ be a word hyperbolic group and $\rho_{L}:\Gamma \rightarrow \mathsf{SL}_m(\mathbb{R})$, $\rho_{R}:\Gamma \rightarrow \mathsf{SL}_d(\mathbb{R})$ two representations. Suppose there is an infinite order element $\gamma_0 \in \Gamma$ with \hbox{$\ell_1(\rho_{L}(\gamma_0)) > \ell_1(\rho_{R}(\gamma_0))$.} Furthermore, suppose that $\rho_{L}$ is $P_1$-Anosov and $\rho_R$ satisfies one of the following conditions:
\begin{enumerate}[label=(\roman*)]
\item $\rho_{R}$ is $P_1$-Anosov.
\item $\rho_{R}(\Gamma)$ is contained in a semisimple proximal Lie subgroup of $\mathsf{SL}_d(\mathbb{R})$ of real rank $1$.\end{enumerate}

\noindent Then, the following conditions are equivalent:

\begin{enumerate}

\item The representation $\rho_{L}\times \rho_{R}:\Gamma \rightarrow \mathsf{SL}_{m+d}(\mathbb{R})$ is $P_1$-Anosov and $\rho_L$ uniformly dominates $\rho_R$.

\item $\underset{|\gamma|_{\Gamma} \rightarrow \infty}{\lim} \frac{\sigma_1(\rho_{L}(\gamma))}{\sigma_1(\rho_R(\gamma))} =+\infty$.

\item $\underset{|\gamma|_{\infty} \rightarrow \infty}{\lim}  \frac{\ell_1(\rho_{L}(\gamma))}{\ell_1(\rho_R(\gamma))} =+\infty$.

\item There exist $C,c>0$ such that for every $\gamma \in \Gamma$ non-trivial, $$ \Big|\log \sigma_1(\rho_{L}(\gamma))-\log \sigma_1(\rho_R(\gamma)) \Big|\geqslant c\log|\gamma|_{\Gamma}-C.$$

\item There exist $C,c>0$ such that  for every $\gamma \in \Gamma$ of infinite order $$ \Big|\log \ell_1(\rho_{L}(\gamma))-\log \ell_1(\rho_R(\gamma)) \Big|\geqslant c\log|\gamma|_{\infty}-C.$$\end{enumerate}\end{theorem}

\begin{proof} Let $G$ be a $P_1$-proximal Lie subgroup of $\mathsf{SL}_d(\mathbb{R})$ of real rank $1$ with Cartan projection \hbox{$\mu_{G}:G \rightarrow \mathbb{R}_{+}$}. Up to conjugation by an element of $\mathsf{GL}_d(\mathbb{R})$, we may write $G=K_{G} \exp \big(\mathbb{R}_{+}X_0 \big)K_{G}$, \hbox{$K_{G} \subset h\mathsf{SO}(d)h^{-1}$} for some \hbox{$h \in \mathsf{SL}_d(\mathbb{R})$} and $\exp(tX_0)=\textup{diag}(e^{ta_1},\ldots,e^{ta_k})$ with $a_1>a_2\geq \ldots \geq a_{d-1}>a_{d}$. The sub-additivity of the Cartan projection shows that there exists $M>0$ such that $$\Big|\log \sigma_i(g)-a_{i}\mu_{G}(g) \Big| \leqslant M$$ for every $g \in G$ and $1 \leq i \leq d$. In particular, there exists $M'>0$ such that $$\log \frac{\sigma_1(g)}{\sigma_2(g)} \geq \frac{a_1-a_2}{a_1} \log \sigma_1(g) -M'$$ for every $g \in G$. Since either (i) or (ii) holds true for $\rho_R$, we may find \hbox{$A,a>0$ such that for every $\gamma\in \Gamma$,} \begin{align*}\log \frac{\ell_1(\rho_{R}(\gamma))}{\ell_2(\rho_R(\gamma))}  \geqslant  a\log \ell_1(\rho_{R}(\gamma)), \ \log \frac{\sigma_1(\rho_{R}(\gamma))}{\sigma_2(\rho_R(\gamma))}  \geqslant  a\log \sigma_1(\rho_{R}(\gamma))-A .\end{align*} 

 Let $\rho:=\rho_{L} \times \rho_{R}$. We obtain continuous, $\rho$-equivariant and transverse maps $\xi_{LR}^{+}:\partial_{\infty}\Gamma \rightarrow \mathbb{P}(\mathbb{R}^{m+d})$ and $\xi_{LR}^{+}:\partial_{\infty}\Gamma \rightarrow \mathsf{Gr}_{m+d-1}(\mathbb{R}^{m+d})$ defined as follows: $$\xi^{+}_{LR}(x)=\xi_{L}^{+}(x), \ \xi^{-}_{LR}(x)=\xi^{-}_{L}(x) \oplus \mathbb{R}^d$$ where $\xi_{L}^{+}$ and $\xi_{L}^{-}$ are the Anosov limit maps of $\rho_L$. For every element $\gamma \in \Gamma$ we observe that the following estimates hold: \begin{align} \label{bound-a} \Bigg|\log \frac{\sigma_1(\rho_L(\gamma))}{\sigma_1(\rho_R(\gamma))} \Bigg| & \geqslant  \log \frac{\sigma_1(\rho(\gamma))}{\sigma_2(\rho(\gamma))}, \ \Bigg|\log \frac{\ell_1(\rho_L(\gamma))}{\ell_1(\rho_R(\gamma))} \Bigg|  \geqslant \log \frac{\ell_1(\rho(\gamma))}{\ell_2(\rho(\gamma))},\end{align} \begin{align}
\label{bound-b} \log \frac{\sigma_1(\rho(\gamma))}{\sigma_2(\rho(\gamma))} &\geqslant \min \Bigg( \Bigg|\log \frac{\sigma_1(\rho_L(\gamma))}{\sigma_1(\rho_R(\gamma))} \Bigg|, \log \frac{\sigma_1(\rho_L(\gamma))}{\sigma_2(\rho_L(\gamma))}, \log \frac{\sigma_1(\rho_R(\gamma))}{\sigma_2(\rho_R(\gamma))}\Bigg), \\ \nonumber \ \log \frac{\ell_1(\rho(\gamma))}{\ell_2(\rho(\gamma))} &\geqslant \min \Bigg( \Bigg|\log \frac{\ell_1(\rho_L(\gamma))}{\ell_1(\rho_R(\gamma))} \Bigg|, \log \frac{\ell_1(\rho_L(\gamma))}{\ell_2(\rho_L(\gamma))}, \log \frac{\ell_1(\rho_R(\gamma))}{\ell_2(\rho_R(\gamma))}\Bigg). \end{align}

\noindent $\textup{(2)} \Rightarrow \textup{(1)}$. We observe that condition (2) and estimate (\ref{bound-b}) together show that $\rho$ is $P_1$-divergent. Since $\xi_{L}^{+}$ satisfies the Cartan property and $\sigma_1(\rho_L(\gamma))>\sigma_1(\rho_R(\gamma))$ as $|\gamma|_{\Gamma} \rightarrow \infty$,  the map $\xi_{LR}^{+}$ has the Cartan property. The maps $\xi_{LR}^{+}$ and $\xi_{LR}^{-}$ are transverse, hence Theorem \ref{maintheorem} shows that $\rho_L \times \rho_R$ is $P_1$-Anosov. $\qed$
\medskip

\noindent $\textup{(3)} \Rightarrow \textup{(1)}$. We are proving that $\textup{(3)} \Rightarrow \textup{(2)}$. Let $\rho_{L}^{ss},\rho_{R}^{ss}$ be semisimplifications of $\rho_{L},\rho_{R}$ respectively. By Proposition \ref{semisimplification}, it is enough to show that $\rho_{L}^{ss}\times \rho_{R}^{ss}$ is $P_1$-Anosov. By Theorem \ref{finitesubset} there exists $C>0$ and a finite subset $F$ of $\Gamma$ such that for every $\gamma \in \Gamma$, there exists $f \in F$ such that $|\log \ell_1(\rho_{L}(\gamma f))-\log \sigma_1(\rho_{L}^{ss}(\gamma))| \leqslant C$ and \hbox{$|\log \ell_1(\rho_{R}(\gamma f))-\log\sigma_1(\rho_{R}^{ss}(\gamma))| \leqslant C$}. 

Let $(\gamma_n)_{n \in \mathbb{N}}$ be an infinite sequence of elements of $\Gamma$. For every $n$ we choose $f_n \in F$ satisfying the previous bounds. The triangle inequality shows \hbox{$||\lambda(\rho_{L}(\gamma_n f_n)||\geqslant ||\mu(\rho_{L}(\gamma_n))||-C$}, hence $\lim_{n}|\gamma_n f_n|_{\infty}=+\infty$. Therefore, $\lim_{n} \big(\log \ell_1(\rho_{L}^{ss}(\gamma_n f_n))-\log \ell_1(\rho_{R}^{ss}(\gamma_n f_n)) \big)=+\infty$ so $\lim_{n}\big(\log \sigma_1(\rho^{ss}_{L}(\gamma_n))-\log \sigma_1(\rho_{R}^{ss}(\gamma_n))\big)=+\infty$. The claim now follows by $\textup{(2)} \Rightarrow \textup{(1)}$. $\qed$
\medskip

\noindent $\textup{(4)} \Rightarrow \textup{(1)}$. We first assume that $c>1$. By estimate (\ref{bound-b}), there exists a constant $C_1>0$ such that $$\log \frac{\sigma_1(\rho(\gamma))}{\sigma_2(\rho(\gamma))} \geqslant c\log|\gamma|_{\Gamma}-C_1$$ for every $\gamma \in \Gamma$. Therefore, by \cite[Thm. 5.3]{GGKW}, we obtain a $\rho$-equivariant map \hbox{$\xi:\partial_{\infty} \Gamma \rightarrow \mathbb{P}(\mathbb{R}^{m+d})$} which satisfies the Cartan property. Then, since $\rho(\gamma_0)$ is $P_1$-proximal, we have $\xi(\gamma_0^{+})=\xi_{LR}^{+}(\gamma_0^{+})$. The minimality of the action of $\Gamma$ on $\partial_{\infty}\Gamma$ shows that $\xi=\xi_{LR}^{+}$. Then $\xi_{LR}^{+}$ satisfies the Cartan property, $\xi_{LR}^{-}$ and $\xi_{LR}^{+}$ are transverse and $\rho$ is $P_1$-divergent. Theorem \ref{maintheorem} shows that $\rho$ is $P_1$-Anosov. \par Now suppose $c\leqslant 1$. We choose $n \in \mathbb{N}$ large enough and consider the symmetric powers $\textup{sym}^{n} \rho_L, \textup{sym}^{n} \rho_{R}$ of $\rho_L, \rho_R$ respectively. Then $\textup{sym}^{n} \rho_L$ is $P_1$-Anosov and $\textup{sym}^{n} \rho_{R}$ satisfies either \textup{(i) or (ii)}. Since $\log \sigma_1(\textup{sym}^{n}\rho_{R}(\gamma))=n\log\sigma_1(\rho_{R}(\gamma))$ for $\gamma \in \Gamma$, the representation $\textup{sym}^{n}\rho_L \times \textup{sym}^{n}\rho_R$ satisfies condition (3) for $c>1$. Therefore, the previous argument implies that the representation $\textup{sym}^{n} \rho_L \times \textup{sym}^{n} \rho_R$ is $P_1$-Anosov. Therefore, by estimate (\ref{bound-a}), we obtain constants $R,k>0$ with $${\Big|\log \sigma_1(\rho_{L}(\gamma))-\log \sigma_1(\rho_R(\gamma))  \Big| \geqslant k|\gamma|_{\Gamma}-R\geq 2\log |\gamma|_{\Gamma}-R }$$ for all but finitely many $\gamma \in \Gamma$. Again, by the argument of the previous paaragraph, we verify that $\rho$ is $P_1$-Anosov. $\qed$

\medskip

\noindent $\textup{(5)} \Rightarrow \textup{(1)}$. It is enough to prove that the semisimplification $\rho_{L}^{ss} \times \rho_{R}^{ss}$ of $\rho$ is $P_1$-Anosov. Note that the representation $\rho_{L}^{ss}$ is $P_1$-Anosov and $\rho_{R}^{ss}$ satisfies either \textup{(i)} or \textup{(ii)}. By Theorem \ref{finitesubset} there exists $L>0$ and a finite subset $F$ of $\Gamma$ such that for every $\gamma \in \Gamma$ there exists $w \in F$ with $||\lambda(\rho_L(\gamma w))-\mu(\rho_{L}^{ss}(\gamma))||\leqslant L$ and \hbox{$||\lambda(\rho_R(\gamma w))-\mu(\rho_{R}^{ss}(\gamma))||\leqslant L$}. Since $\rho_{L}$ is a quasi-isometric embedding, by using the previous inequality, we may find $M>0$ such that $|\gamma w|_{\infty} \geqslant \frac{1}{M}|\gamma|_{\Gamma}-M$, where $\gamma \in \Gamma$ and $w \in F$ are as previously. Finally, we obtain $L',c>0$ such that for every $\gamma \in \Gamma$ non-trivial we have $$\Big|\log \sigma_1(\rho^{ss}_{L}(\gamma))-\log \sigma_1(\rho^{ss}_R(\gamma)) \Big|\geqslant  c\log|\gamma|_{\Gamma}-L'.$$ Therefore, $\rho_{L}^{ss}\times \rho_{R}^{ss}$ is $P_1$-Anosov from $\textup{(4)} \Rightarrow \textup{(1)}$.  $\qed$

\medskip

\noindent $\textup{(1)} \Rightarrow \textup{(2),(3),(4),(5)}$. Since $\ell_1(\rho_L(\gamma_0))>\ell_1(\rho_R(\gamma_0))$, $\xi_{LR}^{+}(\gamma_{0}^{+})$ is the attracting fixed point of $\rho(\gamma_0)$ in $\mathbb{P}(\mathbb{R}^{m+d})$. The action of $\Gamma$ on $\partial_{\infty}\Gamma$ is minimal, hence $\xi_{LR}^{+}$ has to be the Anosov limit map of $\rho$ in $\mathbb{P}(\mathbb{R}^{m+d})$. In particular, $\xi_{LR}^{+}$ satisfies the Cartan property. This shows that for any sequence $(\gamma_n)_{n \in \mathbb{N}}$ of elements of $\Gamma$ we have $\lim_{n}\big(\log \sigma_1(\rho_L(\gamma_n))-\log \sigma_1(\rho_R(\gamma_n))\big)=+\infty$. In particular, there exists $\varepsilon>0$ such that $(1-\varepsilon)\log \ell_1(\rho_L(\gamma))  \geqslant \log \ell_1(\rho_{R}(\gamma))$ for every $\gamma \in \Gamma$. By estimates (\ref{bound-a}), (\ref{bound-b}) and Theorem \ref{mainproperties} (ii) we deduce that \textup{(3), (4), (5)} hold. \end{proof}

\noindent {\em Proof of Corollary \ref{interval}.} Given $p,q\in \mathbb{N}$ with $\textup{dil}_{-}(\rho_1,\rho_2)\leq \frac{p}{q}\leq \textup{dil}_{+}(\rho_1,\rho_2)$, consider the representation $\rho_{p,q}:=\textup{sym}^{q}\rho_1  \times \textup{sym}^{p}\rho_2$. The representation $\textup{sym}^{p}\rho_2$ is $P_1$-Anosov and $\textup{sym}^{q}\rho_1$ satisfies either condition \textup{(i) or (ii)} of Theorem \ref{directsum}. The choice of $p,q\in \mathbb{N}$ shows that the representation $\textup{sym}^{p}\rho_2$ cannot uniformly dominate $\textup{sym}^{q}\rho_1$, so $\rho_{p,q}$ cannot be $P_1$-Anosov. Then, Theorem \ref{directsum} (3) shows that for given $\epsilon>0$ and every $n \in \mathbb{N}$, we can find an element $\gamma_n \in \Gamma$ with $|\gamma_n|_{\Gamma}>n$ and $\big|q\mu_1(\rho_1(\gamma_n))-p\mu_{1}(\rho_2(\gamma_n)) \big| \leqslant \epsilon \log (\mu_1(\rho_1(\gamma_n)))$. The conclusion follows.$\qed$

\begin{remarks}\normalfont{ \textup{(i)} In Theorem \ref{directsum}, in the particular case where both $\rho_L(\Gamma)$ and $\rho_R(\Gamma)$ are contained in a proximal real rank $1$ Lie subgroup of $\mathsf{SL}_m(\mathbb{R})$ and $\mathsf{SL}_d(\mathbb{R})$ respectively, the equivalences $(1)\Leftrightarrow (2) \Leftrightarrow (3)$ are contained in \cite[Thm. 1.14]{GGKW}. In the case where $\rho_L$ and $\rho_R$ take values in $\textup{Aut}_{{\mathbb{K}}}(\mathsf{B})$ (${\mathbb{K}}=\mathbb{R}, \mathbb{C}, \mathbb{H}$) for some bilinear form $\mathsf{B}$ (see \cite[\S 7]{GGKW} for background), the implications $(1) \Leftrightarrow (2) \Leftrightarrow (3) \Rightarrow (5) \Rightarrow (4)$ of Theorem \ref{directsum} are contained in \cite[Prop. 7.13 \& Lem. 7.11 \& Thm. 1.3]{GGKW}.\\
\noindent \textup{(ii)} By Theorem \ref{finitesubset} and Corollary \ref{interval} we deduce that the closure of the set of ratios $$\Bigg\{ \frac{\log \ell_1(\rho_1(\gamma))}{\log \ell_1(\rho_2(\gamma))}:\gamma \in \Gamma_{\infty} \Bigg \}$$ is the closed interval $\big[\textup{dil}_{-}(\rho_1,\rho_2), \textup{dil}_{+}(\rho_1,\rho_2)\big]$. We may replace both $\rho_1$ and $\rho_2$ with their semisimplifications, and this fact also follows by the limit cone theorem of Benoist in \cite{benoist-limitcone, limitcone2}. In the case where $\rho_1$ and $\rho_2$ are convex cocompact into a rank 1 Lie group, the previous fact also follows by \cite[Thm. 2]{Burger}.}\end{remarks}

\section{Examples and counterexamples} \label{examples}

In this section, we discuss examples of representations of surface groups enjoying some of the properties of Anosov representations  which are not $P_1$-Anosov. The examples show that the assumptions of the main results of this paper are necessary. Throughout this section, $g\in \mathbb{N}$ denotes the genus of a surface and $S_g$ denotes \hbox{the (topological) closed orientable surface of genus $g \geq 2$.}

Recall that for a subgroup $H$ of $\mathsf{GL}_d(\mathbb{R})$, containing a $P_1$-proximal element, we denote by $\Lambda_H$ its $P_1$-proximal limit set in $\mathbb{P}(\mathbb{R}^d)$.

\begin{Example} \label{ex1}There exists a strongly irreducible representation $\rho:\pi_1(S_g)\rightarrow \mathsf{SL}_{12}(\mathbb{R})$ which satisfies the following properties: 
\begin{enumerate}
\item $\rho$ is a quasi-isometric embedding, $P_1$-divergent and preserves a properly convex domain $\Omega$ of $\mathbb{P}(\mathbb{R}^{12})$.
\item $\rho$ admits continuous, injective, $\rho$-equivariant maps $$(\xi_1,\xi_{11}): \partial_{\infty}\pi_1(S_g) \rightarrow \mathbb{P}(\mathbb{R}^{12})\times \mathsf{Gr}_{11}(\mathbb{R}^{12})$$ satisfying the Cartan property. The proximal limit set of $\rho(\pi_1(S_g))$ in $\mathbb{P}(\mathbb{R}^{12})$ is $\xi_{1}(\partial_{\infty}\pi_1(S_g))$ and does not contain projective line segments.
\item $\rho$ admits continuous, $\rho$-equivariant  maps $$(\xi_4,\xi_8): \partial_{\infty}\pi_1(S_g) \rightarrow \mathsf{Gr}_4(\mathbb{R}^{12})\times \mathsf{Gr}_{8}(\mathbb{R}^{12})$$ which are transverse.
\item  $\rho$ is not $P_k$-Anosov for any $k=1,\ldots,11$.
\end{enumerate}
\medskip

\noindent The previous example shows that the assumption of transversality in Theorem \ref{maintheorem} is necessary. Moreover, the maps $\xi_4$ and $\xi_8$ are transverse although $\rho$ is not $P_4$-Anosov, therefore the Zariski density assumption  in Theorem \ref{Zariskidense} cannot be dropped.

\begin{proof} Let $g\geq 2$ and $\phi:S_g \rightarrow S_g$ a pseudo-Anosov homeomorphism. The mapping torus $M_{\phi}$ of $S_g$ with respect to $\phi$ is a closed $3$-manifold whose fundamental group is isomorphic to the HNN extension $$\pi_1(M_{\phi})=\Big \langle \pi_1(S_g),t \ \Big| \ tht^{-1}=\phi_{\ast}(h), \ h \in \pi_1(S_g) \Big \rangle$$ where $\phi_{\ast}$ is a representative of the well-defined outer automorphism of $\pi_1(S_g)$, induced by $\phi$. Thurston in \cite{Thurston} (see also Otal \cite{Otal}) proved that there exists a convex cocompact representation $\rho_0:\pi_1(M_{\phi}) \rightarrow \mathsf{PO}(3,1)$. The representation $\rho_0$ lifts to a $P_1$-Anosov representation in $\mathsf{SL}_4(\mathbb{R})$ which we continue to denote by $\rho_0$ and let $\rho_{\textup{Fiber}}:=\rho_{0}|_{\pi_1(S_g)}$. By a result of Cannon-Thurston \cite{CannonThurston}, there exists a continuous $\pi_1(S_g)$-equivariant surjection \hbox{$\theta: \partial_{\infty}\pi_1(S_g) \twoheadrightarrow \partial_{\infty}\pi_1(M_{\phi})$}. By precomposing $\theta$ with the Anosov limit map of $\rho_0$ in $\mathbb{P}(\mathbb{R}^4)$, we obtain a $\rho_{\textup{Fiber}}$-equivariant continuous map $\xi_{\textup{Fiber}}:\partial_{\infty}\pi_1(S_g) \rightarrow \mathbb{P}(\mathbb{R}^4)$. 

Fix a pants decomposition of $S_g$ and let $\gamma_0 \in \pi_1(S_g)$ be an element representing a separating simple closed curve on this decomposition. We claim that there is a Zariski dense, Hitchin representation \hbox{$\rho_{\textup{H}}:\pi_1(S_g) \rightarrow \mathsf{SL}_3(\mathbb{R})$} with $\ell_1(\rho_{\textup{H}}(\gamma_0))=\lambda^2$, $\rho_{\textup{H}}(\gamma_0)=\textup{diag}(\lambda^2,1,\lambda^{-2})$ and $\lambda:= \ell_1 (\rho_{\textup{Fiber}}(\gamma_0))$. To see this, using the fixed pants decomposition of $S_g$, we can fix a discrete faithful representation $j_0:\pi_1(S_g)\rightarrow \mathsf{SL}_2(\mathbb{R})$ such that the modulus of the first eigenvalue of $j_0(\gamma_0)$ is equal to $\lambda$. By composing $j_0$ with the irreducible representation $\textup{sym}^2:\mathsf{SL}_2(\mathbb{R})\rightarrow \mathsf{SL}_3(\mathbb{R})$, we obtain the Fuchsian representation $\textup{sym}^2j_0$ such that $\textup{sym}^2j_0(\gamma_0)$ is conjugate to the matrix $\textup{diag}(\lambda^2,1,\lambda^{-2})$. Then bending along the curve representing $\gamma_0$, gives a Zariski dense Hitchin representation $\rho_{\textup{H}}:\pi_1(S_g)\rightarrow \mathsf{SL}_3(\mathbb{R})$, arbitrarily close to $\textup{sym}^2j_0$, with $\rho_{\textup{H}}(\gamma_0)=\textup{sym}^2j_0(\gamma_0)$.

We claim that $\rho= \rho_{\textup{Fiber}}\otimes \rho_{\textup{H}}: \pi_1(S_g) \rightarrow \mathsf{SL}_{12}(\mathbb{R})$ satisfies the required properties. Consider $\otimes:\mathsf{SO}(3,1) \times \mathsf{SL}_3(\mathbb{R}) \rightarrow \mathsf{SL}_{12}(\mathbb{R})$ the irreducible tensor product representation $(g_1,g_2)\mapsto g_1 \otimes g_2$. Let $G$ be the Zariski closure of $\rho_{\textup{Fiber}} \times \rho_{\textup{H}}$ into $\mathsf{SO}(3,1) \times \mathsf{SL}_3(\mathbb{R})$. Note that the projection of the identity component $G^0$ into $\mathsf{SO}(3,1)$ (resp. $\mathsf{SL}_3(\mathbb{R})$) is normalized by $\rho_{\textup{Fiber}}(\pi_1(S_g))$ (resp. $\rho_{\textup{H}}(\pi_1(S_g))$), so it has to be surjective. Since the Zariski closures of $\rho_{\textup{Fiber}}$ and $\rho_{\textup{H}}$ are simple and not locally isomorphic, it follows by Goursat's lemma that $G=\mathsf{SO}(3,1)\times \mathsf{SL}_3(\mathbb{R})$. We conclude that $\rho$ is strongly irreducible.
\par We obtain a properly convex domain $\Omega$ of $\mathbb{P}(\mathbb{R}^{12})$ preserved by $\rho(\pi_1(S_g))$ as follows. Let $\Omega_1$ and $\Omega_2$ be properly convex domains of $\mathbb{P}(\mathbb{R}^4)$ and $\mathbb{P}(\mathbb{R}^3)$ preserved by $\rho_{\textup{Fiber}}(\pi_1(S_g))$ and $\rho_{\textup{H}}(\pi_1(S_g))$ respectively, and $\Omega_{i}'$ a properly convex cone lifting $\Omega_{i}$ for $i=1,2$. The compact set $$\mathcal{C}:=\Big\{ [u_1 \otimes u_2]\in  \mathbb{P}(\mathbb{R}^{4}\otimes \mathbb{R}^3)  : u_1 \in \overline{\Omega_{1}'},  u_2 \in \overline{\Omega_{2}'} \Big\}$$ is connected, spans $\mathbb{R}^{12}$ and is contained in an affine chart $\mathbb{A}\subset \mathbb{P}(\mathbb{R}^{4}\otimes \mathbb{R}^3)\simeq \mathbb{P}(\mathbb{R}^{12})$. We finally take $\Omega$ to be the interior of the convex hull of $\mathcal{C}$ in $\mathbb{A}$. 
\par  The representations $\rho_{\textup{Fiber}}$ and $\rho_{\textup{H}}$ are $P_1$-divergent hence $\rho$ is also $P_1$-divergent as \begin{align*}\sigma_1(\rho(\gamma))&=\sigma_1(\rho_{\textup{Fiber}}(\gamma))\sigma_1(\rho_{\textup{H}}(\gamma)) \ \ \forall \gamma \in \Gamma \\ \sigma_2(\rho(\gamma))&=\max\big\{  \sigma_1(\rho_{\textup{Fiber}}(\gamma))\sigma_2(\rho_{\textup{H}}(\gamma)), \sigma_1(\rho_{\textup{H}}(\gamma))   \big\}  \ \ \forall \gamma \in \Gamma. \end{align*} In addition, since $\rho_{\textup{H}}$ is a quasi-isometric embedding, we deduce that  $\rho$ is also a quasi-isometric embedding. Let \hbox{$\xi_{\textup{H}}: \partial_{\infty}\pi_1(S_g) \rightarrow \mathbb{P}(\mathbb{R}^3)$} and \hbox{$\xi_{\textup{H}}^{-}: \partial_{\infty}\pi_1(S_g) \rightarrow \mathsf{Gr}_2(\mathbb{R}^3)$} be the Anosov limit maps of $\rho_{\textup{H}}$. The map \hbox{$\xi_1:\partial_{\infty}\pi_1(S_g)\rightarrow \mathbb{P}(\mathbb{R}^{12})$} defined as $$\xi_1(x)=\big[k_{x}e_1 \otimes k_{x}'e_1\big]$$ where $\xi_{\textup{Fiber}}(x)=[k_{x}e_1]$ and $\xi_{\textup{H}}(x)=[k_{x}'e_1]$, is continuous and $\rho$-equivariant. Since $\rho$ is strongly irreducible, the proof of Corollary \ref{Cartan} shows that the map $\xi_1$ satisfies the Cartan property. The image of $\xi_1$ is the $P_1$-proximal limit set $\Lambda_{\rho(\pi_1(S_g))}$ of $\rho(\pi_1(S_g))$ in $\mathbb{P}(\mathbb{R}^{12})$, since $\Gamma$ acts minimally on $\partial_{\infty}\pi_1(S_g)$. Similarly, the dual reprsentation $\rho^{\ast}=\rho_{\textup{Fiber}}^{\ast}\otimes \rho_{\textup{H}}^{\ast}$ admits a $\rho^{\ast}$-equivariant map $\xi_{1}^{\ast}:\partial_{\infty}\pi_1(S_g) \rightarrow \mathbb{P}(\mathbb{R}^{12})$, so we obtain the $\rho$-equivariant map $\xi_{11}$.
\par The maps $\xi_{4}:\partial_{\infty}\pi_1(S_g) \rightarrow \mathsf{Gr}_{4}(\mathbb{R}^{12})$ and $\xi_{8}:\partial_{\infty}\pi_1(S_g) \rightarrow \mathsf{Gr}_{8}(\mathbb{R}^{12})$ defined as $$\xi_{4}(x)=\mathbb{R}^4\otimes \xi_{\textup{H}}(x), \ \xi_{8}(x)=\mathbb{R}^4\otimes \xi_{\textup{H}}^{-}(x) \ \ x\in \partial_{\infty}\pi_1(S_g),$$ are, by their definition, $\rho$-equivariant, continuous and transverse. Also for every $x \in \partial_{\infty}\pi_1(S_g)$ we have $\xi_1(x) \in \xi_4(x)$, hence $\xi_1$ is injective. It follows that $\xi_1(\partial_{\infty}\pi_1(S_g))=\Lambda_{\rho(\pi_1(S_g))} \cong S^1$. For $x \neq y$ the projective line segment $[\xi_{\textup{H}}(x),\xi_{\textup{H}}(y)]$ intersects $\Lambda_{\rho_{\textup{H}}(\Gamma)}$ at the set $\{\xi_{\textup{H}}(x),\xi_{\textup{H}}(y)\}$, hence $[\xi_{1}(x),\xi_{1}(y)]\cap \Lambda_{\rho(\Gamma)}=\{\xi_{1}(x),\xi_{1}(y)\}$. To see this, assume for $x_1,x_2,x_3\in \partial_{\infty}\pi_1(S_g)$, $x_2\neq x_3$, and $\xi_{\textup{H}}(x_i)=[u_i]$ and $\xi_{\textup{Fiber}}(x_i)=[v_i]$. If $v_{1}\otimes u_1\in \mathbb{R} v_2 \otimes u_2+\mathbb{R} v_3\otimes u_3$,  then $e_i \otimes u_1 \in \mathbb{R} e_i\otimes u_2+\mathbb{R} e_i\otimes u_3$, where $i\in \{1,\ldots,4\}$ is any index such that $\langle v_1,e_i\rangle \neq 0$. This implies $\xi_{\textup{H}}(x_1)\in \xi_{\textup{H}}(x_2)\oplus \xi_{\textup{H}}(x_3)$, hence $x_1=x_2$ or $x_1=x_3$.
\par The choice of the element $\gamma_0 \in \pi_1(S_g)$ such that $\ell_1(\rho_{\textup{H}}(\gamma_0))=\lambda^2$, $\lambda=\ell_1(\rho_{\textup{Fiber}}(\gamma_0))$, shows that the moduli of eigenvalues of $\rho(\gamma_0)=\rho_{\textup{Fiber}}(\gamma_0)\otimes \rho_{\textup{H}}(\gamma_0)$ in non-increasing order are $$\lambda^3,\lambda^2, \lambda^2,\lambda, \lambda ,1,1, \lambda^{-1},\lambda^{-1}, \lambda^{-2}, \lambda^{-2},\lambda^{-3}.$$ Thus, $\rho(\gamma_0)$ is not $P_k$-proximal for $k=2,4,6$ and $\rho$ is not $P_k$-Anosov for $k=2,4,6$. Let $\delta \in \pi_1(S_g)$ be a non-trivial element. Since $\phi$ is pseudo-Anosov, the infinite sequence of elements $(\phi_{\ast}^{(n)}(\delta))_{n \in \mathbb{N}}\subset \pi_1(S_g)$ has the property that $(|\phi_{\ast}^{(n)}(\delta)|_{\infty})_{n \in \mathbb{N}}$ is unbounded (where $|\cdot|_{\infty}$ is the stable translation length with respect to a fixed word metric on $\pi_1(S_g)$). By the definition of $\rho_{\textup{Fiber}}$, as $\rho_{\textup{Fiber}}(\phi_{\ast}^{(n)}(\delta))$ is conjugate to $\rho_{0}(\delta)$ for every $n$, there is $M>0$ such that $\frac{\ell_1(\rho_{\textup{Fiber}}(\phi_{\ast}^{(n)}(\delta)))}{\ell_2(\rho_{\textup{Fiber}}(\phi_{\ast}^{(n)}(\delta)))} \leqslant M$ for every $n \in \mathbb{N}$. Then, it is straightforward to check that the ratios $\Big(\frac{\ell_i(\rho(\phi_{\ast}^{(n)}(\delta)))}{\ell_{i+1}(\rho(\phi_{\ast}^{(n)}(\delta)))}\Big)_{n\in \mathbb{N}}$ are uniformly bounded for $i=1,3,5$, so $\rho$ is not $P_k$-Anosov for $k=1,3,5$. \end{proof} \end{Example}

\begin{Example}\normalfont {\em Necessity of the Cartan property.} \label{cp-necessity}The representation $\rho \times \rho_{\textup{H}}: \pi_1(S_g) \rightarrow \mathsf{SL}_{15}(\mathbb{R})$ (where $\rho$ and $\rho_{\textup{H}}$ are from Example \ref{ex1}) is $P_1$-divergent since $$\frac{\sigma_1((\rho\times \rho_{\textup{H}})(\gamma))}{\sigma_2((\rho\times \rho_{\textup{H}})(\gamma))}=\frac{\sigma_1(\rho_{\textup{Fiber}}(\gamma))\sigma_1(\rho_{\textup{H}}(\gamma))}{\max\big\{  \sigma_1(\rho_{\textup{Fiber}}(\gamma))\sigma_2(\rho_{\textup{H}}(\gamma)), \sigma_1(\rho_{\textup{H}}(\gamma))   \big\}} \longrightarrow +\infty$$ as $|\gamma|_{\pi_1(S_g)}\rightarrow +\infty$.

In addition, the product $\rho\times \rho_{\textup{H}}$ admits a pair of continuous, $(\rho\times \rho_{\textup{H}})$-equivariant, compatible and transverse maps $\xi^{+}:\partial_{\infty} \pi_1(S_g) \rightarrow \mathbb{P}(\mathbb{R}^{15})$ and \hbox{$\xi^{-}:\partial_{\infty}\pi_1(S_g) \rightarrow \mathsf{Gr}_{14}(\mathbb{R}^{15})$}, induced from the Anosov limit maps of $\rho_{\textup{H}}$, i.e. $\xi^{+}(x)=\{0\}\times \xi_{\textup{H}}(x)$ and $\xi^{-}(x)=\mathbb{R}^{12}\times \xi_{\textup{H}}^{-}(x)$, $x\in \partial_{\infty}\pi_1(S_g)$. However, $\rho \times \rho_{\textup{H}}$ is not $P_1$-Anosov since $\rho$ cannot uniformly dominate $\rho_{\textup{H}}$. This shows that the assumption of the Cartan property for the map $\xi^{+}$ in Theorem \ref{maintheorem} is necessary. \end{Example}

\begin{Example} \label{complex-quaternionic}\normalfont {\em Necessity of regularity of $\partial \Omega$ in Proposition \ref{main}}. Let $n \geqslant 2$ and $\Gamma$ be a convex cocompact subgroup of $\mathsf{SU}(n,1)\subset \mathsf{SL}_{n+1}(\mathbb{C})$. Let \hbox{$\tau_{2}:\mathsf{GL}_{n+1}(\mathbb{C}) \xhookrightarrow{} \mathsf{GL}_{2n+2}(\mathbb{R})$} be the standard inclusion defined as $$\tau_{2}(h):=\begin{pmatrix}
\textup{Re}(h) &-\textup{Im}(h) \\ 
\textup{Im}(h) & \textup{Re}(h)
\end{pmatrix}, \ \ h \in \mathsf{GL}_{n+1}(\mathbb{C}).$$

Note that for every $h\in \mathsf{SU}(n,1)$, since $\sigma_1(\tau_2(h))=\sigma_2(\tau_2(h))=\sigma_1(h)$ and $\sigma_{i}(h)=1$ for $i=3,\ldots,2n$, the subgroup $\tau_2(\Gamma) \subset \mathsf{SL}_{2n+2}(\mathbb{R})$ is $P_2$-Anosov but not $P_1$-Anosov (in particular not $P_1$-divergent). In addition, since $\sigma_1(\textup{sym}^2(\tau_2(\gamma)))=\sigma_1(\tau_2(\gamma))^2=\sigma_1(\gamma)^2$ for every $\gamma \in \Gamma\subset \mathsf{SU}(n,1)$, we conclude that there exist $J,k>0$ such that $$J^{-1}(\gamma_1 \cdot \gamma_2)_{e}-k \leq \Big(\textup{sym}^2(\tau_2(\gamma_1)) \cdot \textup{sym}^2(\tau_2(\gamma_2)) \Big)_{\varepsilon_1}\leq J(\gamma_1 \cdot \gamma_2)_e+k$$ for every $\gamma_1, \gamma_2 \in \Gamma$. Moreover, $\textup{sym}^2\big(\tau_2(\Gamma) \big)$ preserves a properly convex domain in $\mathbb{P}\big(\textup{Sym}^2 \mathbb{R}^{2n+2}\big)$ but it cannot preserve a strictly convex domain since $\textup{sym}^2(\tau_2(\Gamma))\subset \mathsf{SL}(\textup{Sym}^2(\mathbb{R}^d))$ is not  $P_1$-divergent.

Similar examples are given by convex cocompact subgroups of the rank 1 Lie group $\mathsf{Sp}(n,1) \subset \mathsf{GL}_{n+1}(\mathbb{H})$, where $\mathbb{H}=\mathbb{C}\oplus \mathbb{C}j$ are Hamilton's quaternions. By using the standard embedding $\tau_4:\mathsf{GL}_{n+1}(\mathbb{H})\xhookrightarrow{} \mathsf{GL}_{4n+4}(\mathbb{R})$, $\tau_4(C+Dj)=\tau_2\Big(\begin{pmatrix}[0.3] C & -D \\ \overline{D} & \overline{C} \end{pmatrix}\Big)$, where $C+Dj\in \mathsf{GL}_{n+1}(\mathbb{H})$, $C,D\in \mathsf{Mat}_{n+1}(\mathbb{C})$, for any convex cocompact subgroup $\Delta\subset \mathsf{Sp}(n,1)$, $\tau_4|_{\Delta}$ is $P_4$-Anosov but not $P_1$-Anosov. In addition, there are $R, r>1$ such that for every $h_1,h_2 \in \Delta$,  $$R^{-1}(h_1 \cdot h_2)_{e}-r \leq \Big(\textup{sym}^2(\tau_4(h_1)) \cdot \textup{sym}^2(\tau_4(h_2)) \Big)_{\varepsilon_1}\leq R(h_1 \cdot h_2)_e+r.$$ \end{Example}

\begin{Example}\label{nontransverse}\normalfont {\em Necessity of transversality in Theorem \ref{maintheorem} in the Zariski dense case.} There exists a Zariski dense representation $\rho_1 :\pi_1(S_g) \rightarrow \mathsf{PSL}_{4}( \mathbb{R})$ which is not $P_1$-Anosov but it admits a pair of continuous $\rho_1$-equivariant maps $\xi^{+}:\partial_{\infty}\pi_1(S_g) \rightarrow \mathbb{P}(\mathbb{R}^4)$ and $\xi^{-}:\partial_{\infty}\pi_1(S_g) \rightarrow \mathsf{Gr}_{3}(\mathbb{R}^4)$. \par Let $M^3$ be a closed hyperbolic $3$-manifold which contains a totally geodesic surface. By \cite{Agol}, up to replacing $M^3$ with a finite cover, we may also assume that $M^3$ fibers over the circle (with fiber $S_g$). By \cite{johnson-millson} the natural inclusion $j:\pi_1(M^3) \xhookrightarrow{} \mathsf{PO}(3,1)$ admits a non-trivial Zariski dense deformation $j':\pi_1(M^3) \rightarrow \mathsf{PSL}_4(\mathbb{R})$ which can be chosen to be $P_1$-Anosov, thanks to the openess of Anosov representations (see \cite{labourie, GW}). Let $\xi_{1}^{+}$ and $\xi_{1}^{-}$ be the Anosov limit maps of $j'$ into $\mathbb{P}(\mathbb{R}^4)$ and $\mathsf{Gr}_{3}(\mathbb{R}^4)$ respectively. By the theorem of Cannon-Thurston \cite{CannonThurston} there exists a continuous, $\pi_1(S_g)$-equivariant map $\theta: \partial_{\infty}\pi_1(S_g) \rightarrow \partial_{\infty}\pi_1(M^3)$. The restriction $\rho_1:=j'|_{\pi_1(S_g)}$ is Zariski dense, not a quasi-isometric embedding and $\xi_{1}^{+}\circ \theta $ and $\xi_{1}^{-} \circ \theta$ are continuous, non-transverse and $\rho_1$-equivariant maps. In addition, by \cite{Canary-covering}, every finitely generated free subgroup $F$ of $\pi_1(S_g)$ is a quasiconvex subgroup of $\pi_1(M^3)$. Hence, $\iota'|_{F}$ is $P_1$-Anosov and $\xi^{+}\circ \iota_{F}$ and $\xi^{-}\circ \iota_{F}$ are transverse.  

By \cite[Thm. 7.5]{LLS}, there are also examples of Zariski dense representations $\psi:\mathsf{\Delta}\rightarrow \mathsf{SL}_3(\mathbb{R})$ of triangle reflection groups $\mathsf{\Delta}$, which admit continuous, $\psi$-equivariant, injective maps $\xi^1:\partial_{\infty}\mathsf{\Delta}\rightarrow \mathbb{P}(\mathbb{R}^3)$, $\xi^{2}:\partial_{\infty}\mathsf{\Delta}\rightarrow \mathsf{Gr}_2(\mathbb{R}^3)$ (hence $\psi$ is discrete and faithful), but $\psi$ is not $P_1$-Anosov. \end{Example}

\end{document}